\documentclass[11pt]{article}
\usepackage[english]{babel}
\usepackage{amsmath,amsthm,amssymb}
\usepackage[all]{xy}
\usepackage{setspace}
\usepackage[top=1.3in, bottom=1.6in, left=1.3in, right=1.3in]{geometry}
\usepackage{color}
\definecolor{grau}{rgb}{0.5,0.5,0.5}
\usepackage[colorlinks, linkcolor=grau, citecolor=grau, urlcolor=grau]{hyperref}
\usepackage{breakurl}
\usepackage{bibspacing}
\usepackage{leftidx}
\newtheorem{theorem}{Theorem}[section]
\newtheorem{lemma}[theorem]{Lemma}
\newtheorem{proposition}[theorem]{Proposition}

\newtheorem{corollary}[theorem]{Corollary}

\numberwithin{equation}{section}

\def\ces{(ES)}

\def\sp{\textnormal{Spec}}

\def\P1k{\mathbb P_{k}^{1}}
\def\deg{\textnormal{deg}}

\def\sp{\textnormal{Spec}}
\def\N{N_{K/\mathbb Q}}

\def\gl2{\textnormal{GL}_2}
\def\im{\textnormal{im}}
\def\nd{\textnormal{End}}
\def\ces{(ES)}

\newcommand {\OK}  {{\mathcal O_{K}}}
\newcommand {\OS}  {{\mathcal O_{S}}}

\newcommand {\QQ}  {{\mathbb Q}}
\newcommand {\ZZ}  {{\mathbb Z}}

\newcommand {\CC}  {{\mathbb C}}

\newcommand {\el} {M}

\newcommand {\di} {n_i}
\date{}

\begin{document}
\author{Rafael von K\"anel\footnote{IH\'ES, 35 Route de Chartres, 91440 Bures-sur-Yvette, France, \ \
E-mail adress: {\sf rvk@ihes.fr}}}
\title{Modularity and integral points on moduli schemes}
\maketitle
\begin{abstract}
The purpose of this paper is to give some new Diophantine applications of modularity results. We use the Shimura-Taniyama conjecture to prove effective finiteness results for integral points on  moduli schemes of elliptic curves. For several fundamental Diophantine problems (e.g. $S$-unit and Mordell equations), this gives an effective method which does not rely on Diophantine approximation or transcendence techniques. 
We also combine Faltings' method with Serre's modularity conjecture, isogeny estimates and results from Arakelov theory, to establish the effective Shafarevich conjecture for abelian varieties of (product) $\textnormal{GL}_2$-type. In particular, we open the way for the effective study of integral points on certain higher dimensional moduli schemes.
\end{abstract}

\section{Introduction}

Starting with the key breakthroughs by Wiles \cite{wiles:modular} and by Taylor-Wiles \cite{taywil:modular}, many authors solved important Diophantine problems on using or proving modularity results. 

The purpose of this paper is to give some new Diophantine applications of modularity results. We use the Shimura-Taniyama conjecture  to prove effective finiteness results for integral points on  moduli schemes of elliptic curves. For several fundamental Diophantine problems, such as for example $S$-unit and Mordell equations, this gives an effective method which does not rely on Diophantine approximation or transcendence techniques. We also combine Faltings' method with Serre's modularity conjecture, isogeny estimates and results from Arakelov theory, to establish the effective Shafarevich conjecture for abelian varieties of (product) $\gl2$-type. In particular, we open the way for the effective study  of Diophantine equations  related to integral points on certain higher dimensional moduli schemes such as, for example, Hilbert modular varieties.
In what follows in the introduction, we describe in more detail the content of this paper. 

\subsection{Integral points on moduli schemes of elliptic curves}

To provide some motivation for the study of integral points on moduli schemes of elliptic curves, we discuss in the following section fundamental Diophantine equations which are related to such moduli schemes. For any $\beta\in \QQ$, we denote by $h(\beta)$ the usual (absolute) logarithmic Weil height of $\beta$ defined for example in \cite[p.16]{bogu:diophantinegeometry}.
\subsubsection{$S$-unit and Mordell equations}
Let $S$ be a finite set of rational prime numbers. We define $N_S=1$ if $S$ is empty and  $N_S=\prod p$ with the product taken over all $p\in S$ otherwise.
Let $\mathcal O^\times$ denote the units of $\mathcal O=\ZZ[1/N_S]$. First, we consider the classical $S$-unit equation 
\begin{equation}\label{eq:unit}
x+y=1, \ \ (x,y)\in \mathcal O^\times\times \mathcal O^\times.
\end{equation} 
The study of $S$-unit equations has a long tradition and it is known that many important Diophantine problems are encapsulated in the solutions of (\ref{eq:unit}).  For example, any upper bound  for $h(x)$ which is linear in terms of $\log N_S$  is equivalent to a version of the $(abc)$-conjecture. 
Mahler \cite{mahler:approx1}, 
Faltings \cite{faltings:finiteness} and Kim \cite{kim:siegel} proved finiteness of (\ref{eq:unit}) by completely different methods. Moreover, Baker's method \cite{baker:logarithmicforms} or a method of Bombieri  \cite{bombieri:effdioapp} both allow in principle to find all solutions of any $S$-unit equation.  We will briefly  discuss the methods of Baker, Bombieri, Faltings, Kim and Mahler in Section \ref{subsec:p1altmethods}. In addition, we now point out that Frey remarked in \cite[p.544]{frey:ternary} that the Shimura-Taniyama conjecture implies finiteness of (\ref{eq:unit}). It turns out that one can make Frey's remark in \cite{frey:ternary}   effective and one obtains for example the following explicit result  (see Corollary \ref{cor:p1}): Any solution $(x,y)$ of the $S$-unit equation $(\ref{eq:unit})$ satisfies $$h(x),h(y)\leq \frac{3}{2}n_S(\log n_S)^2+65, \ \ \ n_S=2^7N_S.$$
(After we uploaded the present paper to the arXiv, Hector Pasten informed us about his joint work
with Ram Murty \cite{mupa:modular} in which they independently obtain a (slightly) better
version of the displayed height bound (see \cite[Theorem 1.1]{mupa:modular}) by using a similar method; we refer to the comments below Corollary \ref{cor:p1} for more details. We would like to thank Hector Pasten for informing us about \cite{mupa:modular}.)
Frey uses inter alia his construction of Frey curves. This construction is without doubt brilliant, but rather ad hoc and thus works only in quite specific situations. The starting point for our generalizations are the following two observations: The solutions of (\ref{eq:unit}) correspond to integral points on the moduli scheme $\mathbb P^1_{\ZZ[1/2]}-\{0,1,\infty\}$, and the construction of Frey curves may be viewed as an explicit Par{\v{s}}in construction induced by forgetting the level structure on the elliptic curves parametrized by the points of $\mathbb P^1_{\ZZ[1/2]}-\{0,1,\infty\}$. 

We now discuss a second fundamental
Diophantine equation which is related to integral points on moduli schemes. For any nonzero $a\in \mathcal O$, one obtains a Mordell equation
\begin{equation}\label{eq:mordell}
y^2=x^3+a, \ \   (x,y)\in \mathcal O\times \mathcal O.
\end{equation}
We shall see in Section \ref{sec:mordell} that this Diophantine equation is a priori more difficult than (\ref{eq:unit}). In fact the resolution of (\ref{eq:mordell}) in $\ZZ\times \ZZ$ is equivalent to the classical problem of finding all perfect  squares and perfect cubes with given difference, which goes back at least to Bachet 1621. 
Mordell \cite{mordell:1922,mordell:1923}, Faltings \cite{faltings:finiteness} and Kim \cite{kim:cm} showed finiteness of (\ref{eq:mordell}) by using completely different proofs, and the first effective result for Mordell's equation was provided by Baker \cite{baker:mordellequation}; see Section \ref{subsec:maltmethods} where we briefly discuss methods which show finiteness of (\ref{eq:mordell}).
On working out explicitly the method of this paper for the moduli schemes corresponding to Mordell equations, we get a new effective finiteness proof for (\ref{eq:mordell}). More precisely, if
$a_S=2^83^5N_S^2\prod p^{\min(2,\textnormal{ord}_p (a))}$
with the product taken over all rational primes $p\notin S$ with $\textnormal{ord}_p (a)\geq 1$, then Corollary \ref{cor:m} proves that any solution $(x,y)$ of  $(\ref{eq:mordell})$ satisfies $$h(x),h(y)\leq h(a)+4a_S(\log a_S)^2.$$
\noindent This inequality allows in principle to find all solutions of any Mordell equation (\ref{eq:mordell}) and it provides in  particular an entirely new proof of Baker's classical result \cite{baker:mordellequation}. Moreover, the displayed estimate improves the actual best upper bounds for (\ref{eq:mordell}) in the literature  and it refines and generalizes Stark's theorem \cite{stark:mordell}; see  Section \ref{sec:mordell} for more details.

We observe that $S\mapsto \sp(\ZZ)-S$ defines a canonical bijection between the set of finite sets of rational primes and the set of non-empty open subschemes of $\sp(\ZZ)$. In what follows in this paper (except Sections \ref{sec:p1}-\ref{sec:thue}),  we will adapt our notation to the algebraic geometry setting and the symbol $S$ will denote a base scheme. 

\subsubsection{Integral points on moduli schemes of elliptic curves}\label{sec:modintro}
More generally, we now consider integral points on arbitrary moduli schemes of elliptic curves. We denote by $T$ and $S$ non-empty open subschemes of $\sp(\ZZ)$, with $T\subseteq S$. Let $Y=M(\mathcal P)$ be a moduli scheme of elliptic curves, which is defined over $S$, and let $\lvert \mathcal P\rvert_T$ be the maximal (possibly infinite) number of distinct level $\mathcal P$-structures on an arbitrary elliptic curve over $T$; see Section \ref{sec:parshin} for the definitions. We denote by $Y(T)$ the set of $T$-points of the $S$-scheme $Y$. Let $h_M$ be the pullback of the relative Faltings height by the canonical forget $\mathcal P$-map,  defined in (\ref{def:height}). Write $\nu_T=12^3\prod p^2$
with the product taken over all rational primes $p$ not in $T$. We obtain  in Theorem \ref{thm:ms} the following result.

\vspace{0.3cm}
\noindent{\bf Theorem A.}
\emph{The following statements hold. 
\begin{itemize}
\item[(i)] The cardinality of $Y(T)$ is at most $\frac{2}{3}\lvert \mathcal P \rvert_T\nu_T\prod (1+1/p)$ with the product taken over all rational primes $p$ which divide $\nu_T$.
\item[(ii)] If $P\in Y(T)$, then $h_M(P)\leq \frac{1}{4}\nu_T(\log \nu_T)^2+9$.
\end{itemize}}
\vspace{0.1cm}
If the moduli problem $\mathcal P$ is given with $\lvert \mathcal P \rvert_T<\infty$, then the explicit upper bound for the height $h_M$ in (ii) has the following application: In principle one can  determine the abstract set $Y(T)$ up to a canonical bijection; see the discussion surrounding (\ref{def:height}).
Part (i) gives a quantitative finiteness result for $Y(T)$ provided that $\lvert \mathcal P \rvert_T<\infty$.  In fact most moduli schemes of interest in arithmetic, in particular all explicit moduli schemes considered in this paper, trivially satisfy  $\lvert \mathcal P\rvert_T<\infty.$ However, any scheme over an arbitrary $\ZZ[1/2]$-scheme is a moduli scheme of elliptic curves (see Section \ref{sec:parshin})  and thus there exist many open subschemes $S\subset\sp(\ZZ)$ and moduli schemes $Y$ over $S$ such that $Y(S)$ is infinite.
 
In addition, we show that the Shimura-Taniyama conjecture :=$(ST)$ allows to deal with other classical Diophantine problems. For example, we consider cubic Thue equations, we derive an exponential version of Szpiro's discriminant conjecture for any elliptic curve over $\QQ$, and we deduce an effective Shafarevich conjecture for elliptic curves over $\QQ$.

We remark that the theory of logarithmic forms gives more general versions of the results discussed so far, see \cite{rvk:szpiro,rvk:height}. However, the approach via $(ST)$ has other advantages. For instance, in the two examples which we worked out explicitly, we obtained upper bounds with  numerical constants that are smaller than those coming from the theory of logarithmic forms. Furthermore, in the forthcoming joint work with Benjamin Matschke \cite{vkma:computation}, we will  estimate more precisely the quantities appearing in our proofs to further improve our final numerical constants.  This will allow us to practically resolve $S$-unit and Mordell equations  with ``small" parameters. In fact the practical resolution of these Diophantine equations is still a challenging problem; see for example Gebel-Peth\"o-Zimmer \cite{gepezi:mordell} for partial results on Mordell's equation. We also point out that $(ST)$ has in addition the potential to find the solutions of Diophantine equations without using height bounds. For instance, we shall see in the proof of Theorem A that integral points on moduli schemes of elliptic curves correspond to elliptic curves over $\QQ$ of bounded conductor, which in turn correspond by $(ST)$ to certain newforms of bounded level and such newforms can be computed by Cremona \cite{cremona:algorithms}. We refer to \cite{vkma:computation} for details.

\subsubsection{Principal ideas of Theorem A}  

We continue the notation of the previous section. Let $T\subseteq S\subseteq \sp(\ZZ)$ be as above and suppose that $Y=M(\mathcal P)$ is a moduli scheme  over $S$ with  $\lvert \mathcal P\rvert_T<\infty$. To describe our finiteness proofs for $Y(T)$, we denote by $M(T)$ the set of isomorphism classes of elliptic curves over $T$.  Forgetting the level structure $\mathcal P$ induces a canonical map (see Lemma \ref{lem:moduli})
\begin{equation}\label{eq:forgetful}
Y(T)\to M(T)
\end{equation}
which has fibers of cardinality at most $\lvert \mathcal P\rvert_T<\infty$. Hence to show finiteness of $Y(T)$, it suffices to control $M(T)$. This can be done in two steps: (a)  Finiteness of $M(T)$ up to isogenies  and (b) finiteness of each isogeny class of $M(T)$. Mazur-Kenku \cite{kenku:ellisogenies} implies (b), and (a) follows from $(ST)$ \cite{breuil:modular} which provides an abelian variety $J_0(\nu_T)$ over $\QQ$ of controlled dimension such that the generic fiber $E_\QQ$ of any $[E]\in M(T)$ is a quotient 
\begin{equation}\label{eq:gemstc}
J_0(\nu_T)\to E_\QQ.
\end{equation}
This leads to Theorem A (i).  To prove the explicit height bounds in Theorem A (ii), it suffices by (\ref{eq:forgetful}) to control the relative Faltings height $h(E)$ of $E_\QQ$ (see Section \ref{sec:heights}). We first work out explicitly an estimate of Frey \cite{frey:linksulm} which relies on several non-trivial results, including \cite{kenku:ellisogenies}: If $E_\QQ$ is modular, then Frey estimates  $h(E)$ in terms of the modular degree $m_f$ of the  newform $f$ associated with $E_\QQ$. The theory of modular forms allows to bound $m_f$ in terms of the level $N_E$ of $f$, and $(ST)$ says that $E_\QQ$ is modular. Hence one obtains an estimate for $h(E)$ in terms of $N_E$, which then leads to Theorem A (ii). 

To obtain upper bounds for heights on $Y(T)$ which are different to $h_M$, it remains to work out height comparisons. In the two examples discussed above, one can do this explicitly by using explicit formulas for certain (Arakelov) invariants of elliptic curves.

We emphasize that the crucial ingredients for Theorem A are (\ref{eq:forgetful}) and the ``geometric" version (\ref{eq:gemstc}) of $(ST)$ which relies inter alia on the Tate conjecture \cite{faltings:finiteness}. The other tools, such as Frey's estimate, the theory of modular forms and the isogeny results of Mazur-Kenku \cite{kenku:ellisogenies}, can be replaced by Arakelov theory and isogeny estimates; see Section \ref{subsec:principal} below.  In fact the proof of Theorem A may be viewed as an application  of a refined Arakelov-Faltings-Par{\v{s}}in method to moduli schemes of elliptic curves.
\subsection{Effective Shafarevich conjecture}

In 1983, Faltings \cite{faltings:finiteness} proved the Shafarevich conjecture \cite{shafarevich:conjecture} for abelian varieties over number fields.  It is known that an effective version of the Shafarevich conjecture would have striking Diophantine applications. For example,  
we show in Section \ref{sec:es} that the following effective Shafarevich conjecture $(ES)$ implies the effective Mordell conjecture for any curve of genus at least 2, defined over an arbitrary number field. 

Let $S$ be a non-empty open subscheme of $\sp(\ZZ)$, and let $g\geq 1$ be an integer. We denote by $h_F$  the stable Faltings height, defined in Section \ref{sec:heights}.

\vspace{0.3cm}
\noindent{\bf Conjecture ($ES$).}
\emph{There exists an effective constant $c$, depending only on $S$ and $g$, such that any abelian scheme $A$  over $S$ of relative dimension $g$ satisfies $h_F(A)\leq c.$}
\vspace{0.3cm}

We mention that Conjecture $\ces$ is widely open if $g\geq 2$ and we point out that $\ces$ implies in particular the ``classical'' effective Shafarevich conjecture for curves over arbitrary number fields; we refer to Section \ref{sec:es} for a discussion of Conjecture $\ces$.

Let $A$ be an abelian scheme over $S$ of relative dimension $g$. We say that $A$ is of  $\gl2$-type if there exists a number field $F$ of degree $[F:\QQ]=g$ together with an embedding $F\hookrightarrow \nd(A)\otimes_\ZZ \QQ.$ Here $\nd(A)$ denotes the ring of $S$-group scheme morphisms from $A$ to $A$. More generally, we say that $A$ is of product $\gl2$-type if $A$ is isogenous to a product of abelian schemes over $S$ of $\gl2$-type; see Section \ref{sec:gl2} for a discussion of abelian schemes of product $\gl2$-type.  Write $N_S=\prod p$ with the product taken over all rational primes $p$ not in $S$. We prove in Theorem \ref{thm:es} the following result. 

\vspace{0.3cm}
\noindent{\bf Theorem B.}
\emph{If $A$ is of product $\gl2$-type, then 
$h_F(A)\leq (3g)^{144g}N_S^{24}.$}
\vspace{0.3cm}

This explicit Diophantine inequality establishes in particular the effective Shafarevich conjecture $(ES)$ for all abelian schemes of product $\gl2$-type. In addition, we deduce in Corollary \ref{cor:esc} new cases of the ``classical" effective Shafarevich conjecture for curves, and we derive in Corollary \ref{cor:isoest} new isogeny estimates  for any $A$ of product $\gl2$-type. 

Next, we consider  the set $M_{\gl2,g}(S)$ formed by the isomorphism classes  of abelian schemes over $S$ of relative dimension $g$ which are of product $\gl2$-type.  We obtain in Theorem \ref{thm:qes} the following quantitative finiteness result for $M_{\gl2,g}(S)$.

\vspace{0.3cm}
\noindent{\bf Theorem C.}
\emph{The cardinality of $M_{\gl2,g}(S)$ is at most $(14g)^{(9g)^6}N_S^{(18g)^4}.$}
\vspace{0.3cm}

We deduce in Corollary \ref{cor:qisos}  explicit finiteness results for $\QQ$-isogeny classes of abelian varieties over $\QQ$ of product $\gl2$-type. Further, we mention that Brumer-Silverman \cite{brsi:number}, Poulakis \cite{poulakis:corrigendum} and Helfgott-Venkatesh \cite{heve:integralpoints} established the important special case $g=1$ of Theorem C. They used completely different arguments which in fact give a better exponent for $N_S$ if $g=1$. However, their methods crucially depend on the explicit nature of elliptic curves and they do not allow to deal with higher dimensional abelian varieties; see the discussion surrounding Proposition \ref{prop:shaf} for more details. 

We remark that the above results open the way for the effective study of classes of Diophantine equations which appear to be beyond the reach of the known effective methods. For instance, Theorems B and C are the main tools of the joint paper with Arno Kret \cite{vkkr:intpointsshimura}. Therein we combine these results with canonical forgetful maps in the sense of (\ref{eq:forgetful}), and we prove quantitative and effective finiteness results for integral points on higher dimensional moduli schemes which parametrize abelian schemes of $\gl2$-type. In particular, we work out  the case of Hilbert modular varieties.

\subsubsection{Principal ideas of Theorems B and C}\label{subsec:principal}

We continue the notation of the previous section. Let $S\subseteq \sp(\ZZ)$ and $g\geq 1$ be as above, and let $h_F$ be the stable Faltings height. Suppose that $A$ is an abelian scheme over $S$ of relative dimension $g$ which is of product $\gl2$-type. Write $A_\QQ$ for the generic fiber of $A$.  To prove Theorem B we combine ideas of Faltings \cite{faltings:finiteness} with the following tools:

\begin{itemize}
\item[(i)] If $A_\QQ$ is $\QQ$-simple, then it is a quotient
$
J_1(N)\to A_\QQ 
$ 
of the usual modular Jacobian $J_1(N)$ of some level $N$. Ribet \cite{ribet:gl2} deduced this statement from Serre's modularity conjecture \cite{khwi:serre} by using inter alia the Tate conjecture \cite{faltings:finiteness}.

\item[(ii)] Isogeny estimates for abelian varieties over number fields. These estimates were proven by the method of Faltings \cite{faltings:finiteness}, or by the transcendence method of Masser-W\"ustholz \cite{mawu:abelianisogenies,mawu:factorization};  see Section \ref{sec:var} for more details.

\item[(iii)] Bost's \cite{bost:lowerbound}  lower bound for $h_F$ in terms of the dimension, and Javanpeykar's \cite{javanpeykar:belyi} upper bound for the stable Faltings height of Belyi curves in terms of the Belyi degree and the genus. These results are based on Arakelov theory.
\end{itemize}
Let $N_A$ be the conductor of $A$, defined in Section \ref{sec:cond}. In the proof of Theorem B we first consider the case when $A_\QQ$ is $\QQ$-simple.  A result of Carayol \cite{carayol:conductor} allows to control the number  $N$ in (i). This together with  (i)-(iii) leads to an effective bound  for $h_F(A)$ in terms of $N_A$ and $g$, and then in terms of $N_S$ and $g$ since $A$ is an abelian scheme over $S$.  To reduce the general case of Theorem B to the  case when $A_\QQ$ is $\QQ$-simple, we use inter alia Poincar\'e's reducibility theorem and the isogeny estimates in (ii).

We now describe the principal ideas of Theorem C. Following Faltings \cite{faltings:finiteness} we divide our quantitative finiteness proof for $M_{\gl2,g}(S)$ into two parts: (a) Finiteness of $M_{\gl2,g}(S)$ up to isogenies and (b) finiteness of each isogeny class of $M_{\gl2,g}(S)$. To prove (a) we use (i) and we show that any $\QQ$-simple ``factor" $A_i$ of $A_\QQ$ is a quotient $$J_1(\nu)\to A_i,$$ where $\nu$ is an integer depending only on $S$ and $g$.  To show (b) we combine Theorem B with an estimate of Masser-W\"ustholz \cite{mawu:abelianisogenies,mawu:factorization} for the minimal degree of isogenies of abelian varieties which is based on  transcendence theory. In fact we use here the most recent version of the Masser-W\"ustholz estimate, due to Gaudron-R\'emond \cite{gare:isogenies}.

We remark that results from transcendence theory are not crucial to prove Conjecture $\ces$ for abelian schemes $A$ over $S$ of product $\gl2$-type (resp. to effectively estimate  $\lvert M_{\gl2,g}(S)\rvert$). However, they lead to an upper  bound for $h_F(A)$ (resp. for $\lvert M_{\gl2,g}(S)\rvert$) which is exponentially (resp. double exponentially) better in terms of $N_S$ and $g$, than the estimate which would follow by using the results in (ii) based on Faltings' method.

\subsection{Plan of the paper}

 In Section \ref{sec:heights} we discuss properties of Faltings heights and of the conductor of abelian varieties over number fields. In Section \ref{sec:parshin} we give  Par{\v{s}}in constructions for moduli schemes of elliptic curves and in Section \ref{sec:var} we collect results which control the variation of Faltings heights in an isogeny class.  In Section \ref{sec:modular} we use the theory of modular forms to bound the modular degree of elliptic curves over $\QQ$. We also estimate the stable Faltings heights of certain classical modular Jacobians. Then  we prove in Section \ref{sec:hcell} an explicit height conductor inequality for elliptic curves over $\QQ$ and we derive some applications. In Section \ref{sec:finiteness} we give our effective finiteness results for integral points on moduli schemes of elliptic curves. 
In Section \ref{sec:gl2} we prove a height conductor inequality for abelian varieties over $\QQ$ of product $\gl2$-type. Finally, we establish in Section \ref{sec:es} the effective Shafarevich conjecture $\ces$ for abelian schemes of product $\gl2$-type and we deduce some applications.

We mention that the setting of certain preliminary sections will be more general than
is necessary for the proofs of the main results of this paper, since we wish also to look ahead to future work \cite{rvk:height,vkkr:intpointsshimura}. 

\subsection{Acknowledgements}

I would like to thank Richard Taylor for answering several questions, in particular for proposing a first strategy to prove Lemma \ref{lem:moddeg}. Many thanks go to Bao le Hung, Arno Kret, Benjamin Matschke, Richard Taylor, Jack Thorne and Chenyan Wu for motivating discussions.  Parts of the results were obtained when I was a member (2011/12) at the IAS Princeton, supported by the NSF under agreement No. DMS-0635607. I am grateful to the IAS  and the IH\'ES for providing excellent working conditions. 
Also, I would like to apologize for the long delay between the first presentation (2011) of the initial results on $S$-unit and Mordell equations   and the completion (2013) of the manuscript. The delay resulted from the attempt to understand the initial examples in a way which is more conceptual and which is suitable for generalizations.

\subsection{Conventions and notations}

We identify a nonzero prime ideal of the ring of integers $\OK$ of a number field $K$ with the corresponding finite place $v$ of $K$ and vice versa.  
We write $N_v$ for the number of elements in the residue field of $v$, we denote by $v(\mathfrak a)$ the order of $v$ in a fractional ideal $\mathfrak a$ of $K$ and we write $v\mid \mathfrak a$ (resp. $v\nmid \mathfrak a$) if $v(\mathfrak a)\neq 0$ (resp. $v(\mathfrak a)=0$). If $A$ is an abelian variety over $K$ with semi-stable reduction at all finite places of $K$, then we say that $A$ is semi-stable. 
 
Let $S$ be an arbitrary scheme. We often identity an affine scheme $S=\sp(R)$ with the ring $R$. If $T$ and $Y$ are $S$-schemes, then we denote by $Y(T)=\textnormal{Hom}_S(T,Y)$ the set of $S$-scheme morphisms from $T$ to $Y$ and we write $Y_T=Y\times_S T$ for the base change of $Y$ from $S$ to $T$. Further, if $A$ and $B$ are abelian schemes over $S$, then we denote by $\textnormal{Hom}(A,B)$  the abelian group of $S$-group scheme morphisms from $A$ to $B$ and we write $\nd(A)=\textnormal{Hom}(A,A)$ for the endomorphism ring of $A$. Following \cite{bolura:neronmodels}, we say that $S$ is a Dedekind scheme if $S$ is a normal noetherian scheme of dimension 0 or 1. 

By $\log$ we mean the principal value of the natural logarithm and we define the maximum of the empty set and the product taken over the empty set as $1$. For any set $M$, we denote by $\lvert M\rvert$ the (possibly infinite) number of distinct elements of $M$. Let $f_1,f_2$ be real valued functions on $M$. We write $f_1\ll f_2$ if there exists a constant $c$ such that $f_1\leq cf_2$. Finally, for any map $f:\mathbb R_{>0}\to \mathbb R_{>0}$, 
we write $f_1\ll_\epsilon f_2^{f(\epsilon)}$ if for all $\epsilon>0$ there exists a constant $c(\epsilon)$, depending only on $\epsilon$, such that $f_1\leq c(\epsilon)f_2^{f(\epsilon)}.$

\section{Height and conductor of abelian varieties}\label{sec:heights}
Let $K$ be a number field and let $A$ be an abelian variety over $K$. In the first part of this section, we recall the definition of the relative and the stable Faltings height of $A$, and we review fundamental properties of these heights. In the second part, we define the conductor $N_A$ of $A$ and we recall useful properties of $N_A$.

\subsection{Faltings heights}

We begin to define the relative and stable Faltings height of $A$ following \cite[p.354]{faltings:finiteness}.
If $A=0$ then we set $h(A)=0$. We now assume that $A$ has positive dimension $g\geq 1$. Let $B$ be the spectrum of the ring of integers of $K$. We denote by $\mathcal A$ the N\'eron model of $A$ over $B$, with zero section $e:B\to \mathcal A$.
Let $\Omega^g$ be the sheaf of relative differential $g$-forms of $\mathcal A/ B$.
We now metrize the line bundle $\omega=e^*\Omega^g$ on $B$. For any embedding $\sigma: K\hookrightarrow \mathbb C$, we denote by $A^\sigma$ the base change of $A$ to $\mathbb C$ with respect to $\sigma$. We choose a nonzero global section $\alpha$ of $\omega$. Let $\lVert\alpha_\sigma\rVert_\sigma$ be the positive real number that satisfies
$$\lVert\alpha_\sigma\rVert^2_\sigma=\left(\frac{i}{2}\right)^g\int_{A^\sigma(\mathbb C)}\alpha_\sigma\wedge \overline{\alpha_\sigma},$$ 
where $\alpha_\sigma$ denotes the holomorphic differential form on $A^\sigma$ which is induced by $\alpha$.
Then the relative Faltings height $h(A)$ of $A$ is the real number defined by
\begin{equation*}
[K:\QQ]h(A)=\log\left|\omega/\alpha\omega\right|-\sum\log\lVert\alpha_\sigma\rVert_\sigma
\end{equation*}
with the sum taken over all embeddings $\sigma:K\hookrightarrow \mathbb C$. 
The product formula assures that this definition does not depend on the choice of $\alpha$.
The relative Faltings height is compatible with products of abelian varieties: If $A'$ is an abelian variety over $K$, then  
$
h(A\times_K A')=h(A)+h(A').
$
To see the behaviour of $h$ under base change we take a finite field extension $L$ of $K$.  The universal property of N\'eron models implies
\begin{equation}\label{eq:strict}
h(A_L)\leq h(A).
\end{equation}
This inequality can be strict 
and thus the height $h$ is in general not stable under base change.
To obtain a stable height we may (see \cite{grra:neronmodels}) and do take a finite extension $L'$ of $K$ such that $A_{L'}$  is semi-stable.
The stable Faltings height $h_F(A)$ of $A$ is defined as 
\begin{equation*}
h_F(A)=h(A_{L'}).
\end{equation*}
This definition does not depend on the choice of $L'$, since the formation of the identity components of the corresponding semi-stable N\'eron models commutes with the induced base change. 
In particular, inequality (\ref{eq:strict}) becomes an equality when $h$ is replaced by $h_F$. Further, we define $h_F(0)=0$. We shall need an effective lower bound for $h_F(A)$ in terms of the dimension $g$ of $A$. 
An explicit result of Bost \cite{bost:lowerbound} gives
\begin{equation}\label{eq:bost}
-\frac{g}{2}\log(2\pi^2)\leq h_F(A).
\end{equation}
See for example \cite[Corollaire 8.4]{gare:periods} and  notice that $h_F(A)=h_B(A)-\frac{g}{2}\log \pi$ where $h_B$ denotes the height which appears in the statement of \cite[Corollaire 8.4]{gare:periods}.

We shall state several of our results in terms of $h_F$ or $h$ and therefore we now briefly discuss important differences between these heights. From  (\ref{eq:strict}) we deduce that $h_F(A)\leq h(A).$ Further, as already observed, the height  $h_F$ has the advantage over $h$ that it is stable under base change.
On the other hand, $h_F$ has in general weaker finiteness properties. 
For instance, there are only finitely many $K$-isomorphism classes of elliptic curves over $K$ of bounded $h$, while $h_F$ is bounded on the infinite set given by the $K$-isomorphism classes of elliptic curves of any fixed $j$-invariant in $K$.

More generally, let $S$ be a connected Dedekind scheme with field of fractions $K$. If $A$ is an abelian scheme over $S$, then we define the stable and relative Faltings height of $A$ by $h_F(A)=h_F(A_K)$ and $h(A)=h(A_K)$ respectively. Here $A_K$ is the generic fiber of $A$.

\subsection{Conductor}\label{sec:cond}
We first define the conductor $N_A$ of an arbitrary abelian variety  $A$ over any number field $K$. Let $v$ be a finite place of $K$. We denote by $f_v$ the usual conductor exponent of $A$ at $v$, see for example \cite[Section 2.1]{serre:conductor} for a definition. The conductor $N_A$ of $A$ is defined by 
\begin{equation}\label{def:conductor}
N_A=\prod N_v^{f_v}
\end{equation}
with the product taken over all finite places $v$ of $K$. In particular,  $f_v(0)=0$ and $N_0=1$. We now recall some useful properties of $f_v$ and $N_A$. It holds that $f_v=0$ if and only if $A$ has good reduction at $v$. 
Furthermore, if $A'$ is an abelian variety over $K$ which is $K$-isogenous to $A$, then $f_v(A)=f_v(A')$ and thus $N_A=N_{A'}$. 
Finally, if $A'$ is an abelian variety over $K$ and if $C=A\times_K A'$, then  $f_v(C)=f_v(A)+f_v(A')$ and hence $N_{C}=N_AN_{A'}$. 

We shall need an explicit upper bound for $f_v$ in terms of $g=\dim(A)$ and $K$. Brumer-Kramer \cite{brkr:conductor} obtained such a bound by refining earlier work of Serre \cite[Section 4.9]{serre:representations} and of Lockhart-Rosen-Silverman \cite{lorosi:conductor}. To state the main result of \cite{brkr:conductor} we have to introduce some notation. Let $p$ be the residue characteristic of $v$, let $e_v=v(p)$ be the ramification index of $v$, and let $n$ be the largest integer that satisfies $n\leq 2g/(p-1)$. We define $\lambda_p(n)=\sum i r_ip^i$ for $\sum r_i p^i$ the $p$-adic expansion of $n=\sum r_i p^i$ with integers $0\leq r_i\leq p-1$. 
Then \cite[Theorem 6.2]{brkr:conductor} gives
\begin{equation}\label{eq:abcondineq}
f_v\leq 2g+e_v\bigl(pn+(p-1)\lambda_p(n)\bigl).
\end{equation} 
Furthermore, the examples in \cite{brkr:conductor} show that  (\ref{eq:abcondineq}) is best possible in a strong sense.

More generally, if $S$ is a connected Dedekind scheme with field of fractions $K$ and if $A$ is an abelian scheme over $S$, then we define the conductor $N_A$ of $A$ by $N_A=N_{A_K}$.

\section{Par{\v{s}}in constructions: Forgetting the level structure}\label{sec:parshin}
Par{\v{s}}in \cite{parshin:construction} discovered a link between the Mordell and the Shafarevich conjecture
which is now commonly known as Par{\v{s}}in construction or Par{\v{s}}in trick. 
This link gives a finite map from the set of rational points of $X$ into the integral points of a certain moduli space, 
where $X$ is a curve of genus at least two which is defined over a number field.

In the first part of this section, we use the moduli problem formalism to obtain tautological Par{\v{s}}in constructions for moduli schemes of elliptic curves. 
In the second part, we explicitly work out this idea for $\mathbb P^1-\{0,1,\infty\}$ and once punctured Mordell elliptic curves. 
This results in completely explicit Par{\v{s}}in constructions for these hyperbolic curves.

\subsection{Moduli schemes}

We begin to introduce some notation and terminology. Let $S$ be an arbitrary scheme. An elliptic curve over $S$ is an abelian scheme over $S$ of relative dimension one.
A morphism of elliptic curves over $S$ is a morphism of abelian schemes over $S$. 
We denote by $$\el(S)$$ the set of isomorphism classes of elliptic curves over $S$.  
On following Katz-Mazur \cite[p.107]{kama:moduli}, we write  $(Ell)$ for the category of elliptic curves over variable base-schemes: The objects are elliptic curves over schemes and the morphisms are given by cartesian squares of elliptic curves. Let $(Sets)$ be the category of sets and let $\mathcal P$ be a contravariant functor from $(Ell)$ to $(Sets)$.  
We say that $\mathcal P$ is a moduli problem on $(Ell)$ and we define  
\begin{equation}\label{def:ps}
\left|\mathcal P\right|_S=\sup\left|\mathcal P(E/S)\right|
\end{equation}
with the supremum taken over all elliptic curves $E$ over $S$. 
In other words,  $\lvert \mathcal P\rvert_S$ is the maximal (possibly infinite) number of distinct level $\mathcal P$-structures on an arbitrary elliptic curve over $S$.  A scheme $M(\mathcal P)$ is called a moduli scheme (of elliptic curves) if there exists a moduli problem $\mathcal P$ on $(Ell)$ which is representable by an elliptic curve over $M(\mathcal P)$. 
The following lemma may be viewed as a tautological Par{\v{s}}in construction for moduli schemes.

\begin{lemma}\label{lem:moduli}
Suppose $Y=M(\mathcal P)$ is a moduli scheme, defined over a scheme $S$. If $T$ is a $S$-scheme, then there is a map  $Y(T)\to M(T)$ with fibers of cardinality at most $\left|\mathcal P\right|_T$.
\end{lemma}
\begin{proof}
We notice that the statement is intuitively clear, since $Y(T)$ is essentially the set of elliptic curves over $T$ with ``level $\mathcal P$-structure'' and the map is essentially ``forgetting the level $\mathcal P$-structure''.  We now verify that this intuition is correct. 

By assumption, there exists a contravariant functor $\mathcal P$ from $(Ell)$ to $(Sets)$ which is representable by an elliptic curve over $Y$. Suppose  $E$ and $E'$ are elliptic curves over a scheme $Z$, 
with $\alpha\in\mathcal P(E)$ and $\alpha'\in\mathcal P(E')$.
Then the pairs $(E,\alpha)$ and $(E',\alpha')$ are called isomorphic 
if there exists an isomorphism $\varphi:E\to E'$ of objects in $(Ell)$ with $\mathcal P(\varphi)(\alpha')=\alpha$.
Let $F(Z)$ be the set of isomorphism classes of such pairs $(E,\alpha)$.
Then $Z\mapsto F(Z)$ defines a contravariant functor from the category of schemes to $(Sets)$, which is representable by $Y$ since $\mathcal P$ is representable by an elliptic curve over $Y$. 
Thus we obtain an inclusion $Y(T)\hookrightarrow F(T)$, 
which composed with   
$$F(T)\to M(T): [(E,\alpha)]\mapsto [E]$$ gives a map $Y(T)\to M(T)$. 
Suppose $\{[(E_i,\alpha_i)], 1\leq i\leq n\}$ is the fiber of this map over a point in $M(T)$. Then all $E_i$ are isomorphic objects of $(Ell)$. 
Therefore, after applying suitable isomorphisms of objects in $(Ell)$, we may and do assume that all $E_i$ coincide. 
This shows that $n\leq \left|\mathcal P\right|_T$ and then we conclude Lemma \ref{lem:moduli}.
\end{proof}

We call the map constructed in Lemma \ref{lem:moduli} the forget $\mathcal P$-map. To discuss some fairly general examples of moduli schemes we consider an arbitrary scheme $Y$. If there exists an elliptic curve $E$ over $Y$, then $Y=M(\mathcal P)$ is a moduli scheme with $\mathcal P=\textnormal{Hom}_{(Ell)}(-,E)$.  This shows in particular that any $\ZZ[1/2]$-scheme $Y$ is a moduli scheme, since there exists an elliptic curve $A$ over $\ZZ[1/2]$ and the base change $A_Y$ is an elliptic curve over $Y$. 
Next, we discuss a classical example of a moduli problem. Let $N\geq 1$ be an integer and consider the ``naive'' level $N$ moduli problem $\mathcal P_N$ from $(Ell)$ to $(Sets)$, defined by
\begin{equation*}
E/S\mapsto \{S\textnormal{-group-scheme isomorphisms } (\ZZ/N\ZZ)^2\tilde{=}E[N]\}.
\end{equation*}
Here we view $(\ZZ/N\ZZ)^2$ as a constant $S$-group-scheme  and $E[N]$ is the kernel of the $S$-homomorphism ``multiplication by $N$'' on the elliptic curve $E$ over $S$. If $\mathcal P_N(E/S)$ is non-empty and if $S$ is connected, then we explicitly compute  
\begin{equation}\label{eq:compgn}
\mathcal P_N(E/S)\cong \{\ZZ/N\ZZ\textnormal{-bases of } (\ZZ/N\ZZ)^2\}.
\end{equation} If $N\geq 3$ then \cite[Corollary 4.7.2]{kama:moduli} gives that $\mathcal P_N$ is a representable moduli problem on $(Ell)$, with moduli scheme $Y(N)=M(\mathcal P_N)$ a smooth affine curve over $\sp(\ZZ[1/N])$.

In the remaining of this section, we give two propositions. 
Their proofs consist essentially of working out explicitly Lemma \ref{lem:moduli} for particular moduli schemes, see the remarks given below the proofs of Propositions \ref{prop:p1par} and \ref{prop:mpar} respectively.

\subsection{Explicit constructions}\label{sec:explicitconstr}

We introduce and recall some notation. Let $K$ be a number field and write $B$ for the spectrum of the ring of integers $\OK$ of $K$. 
In the remaining of this section, we denote by $$S\to B$$ either a non-empty open subscheme of $B$ or the spectrum of the function field $K$ of $B$ 
and we write $\mathcal O=\OS(S)$.
Let $E$ be an elliptic curve over $S$. We denote by $h(E)$ and by $h_F(E)$
the relative and the stable Faltings height of the generic fiber $E_K$ of $E$ respectively, see Section \ref{sec:heights} for the definitions.
Let $N_E$ be the conductor of $E_K$ defined in Section \ref{sec:cond} and let $\Delta_E$ be the norm from $K$ to $\QQ$ of the usual minimal discriminant ideal of $E_K$ over $K$. 
We observe that $h(E),h_F(E),N_E$ and $\Delta_E$ define real valued functions on $\el(S)$. 

Let $Y=M(\mathcal P)$ be a moduli scheme defined over $S$, let $T$ be a non-empty open subscheme  of $S$ and let $\phi:Y(T)\to M(T)$ be the forget $\mathcal P$-map from Lemma \ref{lem:moduli}.  On pulling back the relative Faltings height $h$ by $\phi$, we get a height $h_M$ on $Y(T)$ defined by
\begin{equation}\label{def:height}
h_M(P)=h(\phi(P)), \ \ \ P\in Y(T).
\end{equation}
The height $h_M$ has the following properties: If $\left|\mathcal P\right|_T<\infty$, then Lemma \ref{lem:moduli} together with Lemma \ref{lem:wei} below shows that there exist only finitely many $P\in Y(T)$ with $h_M(P)$ bounded. Furthermore, if  $\mathcal P$ is given with $\left|\mathcal P\right|_T<\infty$, then the proof of Lemma \ref{lem:moduli} together with Lemma \ref{lem:wei} below implies that one can in principle determine, up to a canonical bijection, the set of points $P\in Y(T)$ with $h_M(P)$ effectively bounded.

Let $D_K$ be the absolute value of the discriminant of $K$ over $\QQ$, let $d=[K:\QQ]$ be the degree of $K$ over $\QQ$ and let $h_K$ be the cardinality of the class group of $B$. 
We define $$N_T=\prod N_v$$ with the product taken over all $v\in B-T$; notice that $N_T=\infty$ if $S=T=\sp(K)$. Further, we say that any nonzero  $\beta\in K$ is invertible on $T$ if $\beta$ and $\beta^{-1}$ are both in $\mathcal O_
T(T)$. For any vector $\beta$ with coefficients in $K$, we denote by $h(\beta)$ the usual absolute logarithmic Weil height of $\beta$ which is defined in \cite[1.5.6]{bogu:diophantinegeometry}.

\subsubsection{$S$-unit equations}\label{subsec:s-unit}

We continue the notation introduced above and we now give an explicit Par{\v{s}}in construction for ``$S$-unit equations''. The solutions of such equations correspond to $S$-points of $$\mathbb P^1_S-\{0,1,\infty\}=\sp(\mathcal O[z,1/(z(1-z))]).$$
To simplify notation we write $X=\mathbb P^1_S-\{0,1,\infty\}$. For any $P\in X(S)$, we define $h(P)=h(z(P))$.\footnote{If $Z$ is an affine $S$-scheme, $P\in Z(S)$ and $f\in \mathcal O_Z(Z)$, then $f(P)\in \mathcal O$ denotes the image of $f$ under the ring morphism $\mathcal O_Z(Z)\to \mathcal O$ which corresponds to $P:S\to Z$.} We say that a map of sets is finite if all its fibers are finite.

\begin{proposition}\label{prop:p1par}
Suppose that $T$ is an open subscheme of $S$, with $2$ invertible on $T$. Then there exists a finite map $\phi: X(S)\rightarrow \el(T)$ with the following properties.
\begin{enumerate}
 \item[(i)]  Suppose $P\in X(S)$ and $[E]=\phi(P)$. Then it holds $ N_E\leq 2^{6d}3^{5d} N_T^2$ and  \\  $h(P)\leq 6h_F(E)+3\log\bigl(\max(1,h_F(E))\bigl)+42.$
 \item[(ii)] There is  an elliptic curve $E'$ over $K$ that satisfies $h_F(E')=h_F(E)$ and \\ $N_{E'}\leq 2^{7d}3^{5d}D_K^{h_K-1}N_T.$
 \item[(iii)] If $B$ has trivial class group, then $E'$ extends to an elliptic curve over $T$ \\ 
and $N_{E'}\mid 2^{7d}N_T$. If $K=\QQ$, then $h(P)\leq 6h(E')+11.$  
\end{enumerate}
\end{proposition}

In this article, we shall use Proposition \ref{prop:p1par} only for one dimensional $S$ and $T$. However, the height inequalities obtained in this proposition may be also of interest for $S=T=\sp(K)$. We mention that the number 6 in these height inequalities is optimal. 

To prove Proposition \ref{prop:p1par} we shall use inter alia the following lemma. 

\begin{lemma}\label{lem:dehe}
If $E$ is an elliptic curve over $S$, then $\log \Delta_E\leq 12d(h(E)+4/3).$
\end{lemma}
\begin{proof}
For any embedding $\sigma:K\hookrightarrow \mathbb C$, we take $\tau_\sigma\in \mathbb C$ such that the base change of $E_K$ to $\mathbb C$ with respect to $\sigma$ takes the form $\mathbb C/(\ZZ+\tau_\sigma\ZZ)$ and such that $\im(\tau_\sigma)\geq \sqrt{3}/2$. 
We write $q=\exp(2\pi i \tau_\sigma)$ and
$\Delta(\tau_\sigma)=q\prod_{n=1}^{\infty}(1-q^n)^{24}$. From \cite[Proposition 1.1]{silverman:arithgeo} we get
\begin{equation*}
\log \Delta_E=12dh(E)+\sum\log\left|(2\pi)^{12}\Delta(\tau_\sigma)\im(\tau_\sigma)^6\right|
\end{equation*}
with the sum taken over all embeddings $\sigma:K\hookrightarrow \mathbb C$. Here  $\lvert\cdot \rvert$ denotes the complex absolute value. 
Further, on using the elementary inequalities $\log \left|\Delta(\tau_\sigma)/q\right|\leq 24\left|q\right|/(1-\left|q\right|)$ and $\left|q\right|\leq \exp(-\pi\sqrt{3})$, we deduce the estimate 
\begin{equation}\label{eq:-delta}
\log \left|(2\pi)^{12}\Delta(\tau_\sigma)\im(\tau_\sigma)^6\right|\leq 16. 
\end{equation}
This together with the displayed formula for $\log \Delta_E$ implies the statement. 
\end{proof}
We remark that the proof shows in addition that Faltings' delta invariant $\delta(E_\CC)$ \cite[p.402]{faltings:arithmeticsurfaces}  of a compact connected Riemann surface $E_\CC$ of genus one satisfies $$\delta(E_\CC)\geq -9.$$ Indeed, this follows directly from (\ref{eq:-delta}) and Faltings' explicit formula \cite[Lemma c), p.417]{faltings:arithmeticsurfaces} for $\delta(E_\CC)$. 
We mention that it is an important open  problem to obtain explicit lower bounds, in terms of the genus, for the Faltings delta invariant of compact connected Riemann surfaces of arbitrary positive genus.

We shall need  an estimate for the conductor. If $v$ is a closed point of $B$ and if $f_v$ denotes the conductor exponent at $v$ of an elliptic curve over $K$ (see Section \ref{sec:cond}), then
\begin{equation}\label{eq:cond23}
f_v\leq 2+6v(2) \textnormal{ if } v\mid 2 \textnormal{ and } f_v\leq 2+3v(3)  \textnormal{ if } v\mid 3.
\end{equation}
This follows directly from the result of Brumer-Kramer which we stated in (\ref{eq:abcondineq}).

\begin{proof}[Proof of Proposition \ref{prop:p1par}]
We observe that if $X(S)$ is empty, then all statements are trivial. Hence we may and do assume that $X(S)$ is not empty. We denote by $Y$ the spectrum  of $\ZZ[\lambda,1/(2\lambda(1-\lambda))]$ for $\lambda$ an ``indeterminate''. Then we observe that
$$y^2=x(x-1)(x-\lambda )$$ 
defines an (universal) elliptic curve $\mathcal E$ over $Y$. We take $P\in X(S)$. 
On using that $Y_T\cong X_T$, we obtain a morphism $T\to Y_T$ induced by $P$. 
Let $E$ be the fiber product of $\mathcal E_{Y_T}\to Y_T$ with this morphism $T\to Y_T$. Then  $E$ is  an elliptic curve over $T$ and therefore we see that $$P\mapsto [E]$$ defines a map $\phi:X(S)\rightarrow \el(T).$ 
If $P'\in X(S)$ satisfies $\phi(P)=\phi(P')$, then it follows that $z(P')=(z_1-z_2)/(z_3-z_2)$ with pairwise distinct $z_1,z_2,z_3\in \{0,1,z(P)\}$. Thus $\phi$ is finite.

We now prove (i). In what follows we write $\lambda$ for $z(P)$ to simplify notation. The $j$-invariant $j$ of the generic fiber $E_K$ of $E$ satisfies 
\begin{equation*}
j=2^8\frac{(\lambda^2-\lambda+1)^3}{(\lambda^2-\lambda)^{2}}. 
\end{equation*}
This implies that $2v(\lambda)=v(j)-8v(2)$ for any finite place $v$ of $K$ with $v(\lambda)\leq -1$ and that $\lvert \sigma(\lambda)\rvert^2\leq \lvert\sigma(j)\rvert$ for any embedding $\sigma:K\hookrightarrow \CC$ with $\lvert \sigma(\lambda)\rvert\geq 2$, where $\lvert\cdot \rvert$ denotes the complex absolute value. 
We deduce $$h(P)\leq h(j)/2+5\log 2.$$ 
Furthermore, Pellarin's \cite[p.240]{pellarin:isogenies} explicit calculation of the constant in Silverman's \cite[Proposition 2.1]{silverman:arithgeo} leads to
\begin{equation}\label{eq:pellarin}
h(j)\leq 12h_F(E)+6\log\bigl(\max(1,h_F(E))\bigl)+75.84. 
\end{equation} 
This implies an upper bound for $h(P)$ as stated in (i). Next, we prove the claimed estimate for the conductor $N_E$ of $E$. This estimate holds trivially if $T=\sp(K)$, and we now assume that $T\neq \sp(K)$. In what follows we denote by $v$  a closed point of $B$. Let $f_v$ be the conductor exponent of  $E_K$ at $v$. If $v\in T$, then $E_K$ has good reduction at $v$,
since  $E\to T$ is smooth and projective, and 
we obtain $f_v=0$. Thus the estimates in (\ref{eq:cond23}) for $f_v$ if $v\mid 6$ combined with $f_v\leq 2$ if $v\nmid 6$ lead to an upper bound for $N_E$ as stated in (i).

To show (ii) we observe that the statement is trivial if $T=\sp(K)$. Hence  we may and do assume that $T\neq \sp(K)$. As in the proof of \cite[Lemma 4.1]{rvk:hyperelliptic}, we see that Minkowski's theorem  gives  an open subscheme $U$ of $B$ with the following properties. There are at most $h_K-1$ points in $B-U$, any $v\in B-U$ satisfies $N_v^2\leq D_K$ and the class group of $U$ is trivial. 
Then we may and do take coprime elements $l,m\in\mathcal O_U(U)$ such that $$\lambda=l/m.$$ Let
$E'$ be an elliptic curve over $K$ defined by the Weierstrass equation $y^2=x(x-l)(x-m)$. 
We observe that $E'$ is geometrically isomorphic to $E_K$. This implies that the $j$-invariant of $E'$ coincides with $j$ and $h_F(E')=h_F(E)$. 
We now prove the claimed estimate for the conductor $N_{E'}$ of $E'$. Let $\Delta$ and $c_4$ be the usual quantities  associated to the above Weierstrass equation of $E'$, see  \cite[p.42]{silverman:aoes}. 
They take the form $$\Delta=2^4(lm(l-m))^2 \textnormal{ and } c_4=2^4((l-m)^2+lm).$$  
Let $f'_v$ be the conductor exponent of $E'$ at  $v$. First, we assume that $v\in U$ with $v\nmid 2$. 
If $v(\Delta)\geq 1$, then it follows that $v(c_4)=0$, since $l,m\in \mathcal O_U(U)$ are coprime and $v\nmid 2$. 
This implies that the above Weierstrass equation is minimal at $v$ 
and then \cite[p.196]{silverman:aoes} proves that $E'$ is semi-stable at $v$.
We conclude  $f'_v\leq 1$. 
Next, we assume $v\in U\cap T$. In the proof of (i) we showed $f_v=0$. This implies that $f'_v=0$, since $E_K$ is geometrically isomorphic to $E'$ and $E'$ is semi-stable at $v$. On combining the above observations, we deduce
$$
N_{E'}\leq 2^{-d}N_T\prod N_v^{f'_v} 
$$
with the product taken over all $v\in B$ such that  $v\in B-U$ or $v\mid 2$. 
Therefore, on using the properties of $U$, we see that the estimates in (\ref{eq:cond23}) for $f'_v$ if $v\mid 6$ combined with $f'_v\leq 2$ if $v\nmid 6$ imply an upper bound for $N_{E'}$ as claimed in (ii).

It remains to prove (iii). We notice that the first assertion of (iii) is trivial if $T=\sp(K)$. If $B$ has trivial class group, then we can take $U=B$ in the proof of (ii): It follows that $f'_v=0$ for any closed point $v\in T$ and that $$N_{E'}\leq 2^{-d}N_T\prod N_v^{2+6v(2)}$$ 
with the product taken over all $v\in B$ with $v\mid 2$.
This shows that $E'$ is the generic fiber of an elliptic curve over $T$ and that $N_{E'}\leq 2^{7d}N_T$. If $K=\QQ$, then we obtain that $h(P)\leq 1/2\log \lvert\Delta\rvert-2\log 2$, 
and \cite[p.257]{silverman:aoes} shows that $\lvert\Delta\rvert\leq 2^{12}\Delta_{E'}$.
Therefore Lemma \ref{lem:dehe} proves (iii). 
This completes the proof of Proposition \ref{prop:p1par}.\end{proof}
We remark that the elliptic curve $\mathcal E$ over $Y$, which appears in the above proof, represents the moduli problem $\mathcal P=[Legendre]$ on $(Ell)$ defined in \cite[p.111]{kama:moduli}. The moduli scheme $Y$ is defined over $\sp(\ZZ[1/2])$. If $2$ is invertible on $S$ and if $T=S$, then it follows that $X(S)=Y(S)$  and that the map $\phi:X(S)\to M(T)$ in Proposition \ref{prop:p1par} coincides with the map $Y(T)\to M(T)$ in Lemma \ref{lem:moduli}. 
However, to get our explicit inequalities in Proposition \ref{prop:p1par} it is necessary to take into account the particular shape of $\mathcal P=[Legendre]$.

\subsubsection{Mordell equations}\label{subsec:mordell}

We continue the notation introduced above and we now give an explicit Par{\v{s}}in construction for  Mordell equations.
For any nonzero $a\in \mathcal O$, we obtain that 
\begin{equation*}
Z=\sp\bigl(\mathcal O[x,y]/(y^2-x^3-a)\bigl)
\end{equation*}
defines  an affine Mordell curve  over $S$.
To state our next result we have to introduce some additional notation. If $P\in Z(S)$ then we write $h(P)=h(x(P))$. 
Let $R_K$ be the regulator of $K$ and let $r_K$ be the rank of the free part of the group of units $\mathcal O_K^\times$ of $\OK$. We define 
\begin{equation*}
\kappa=\log (D_K)/2d+79R_K(r_K!)r_K^{3/2}\log d,
\end{equation*}
and we observe that $\kappa=0$ when $K=\QQ$. The origin of the constant $\kappa$ shall be explained below Lemma \ref{lem:wei}. To measure the number $a\in\mathcal O$, we use inter alia the quantity
\begin{equation*}
r_2(a)=\prod N_v^{\min(2,v(a))} 
\end{equation*}
with the product taken over all closed points $v\in S$ with $v(a)\geq 1$. We observe that $\log r_2(a)\leq dh(a)$ and if $a\in \OK$, then $r_2(a)\leq \N(a)$ for $\N$ the norm from $K$ to $\QQ$.
 
\begin{proposition}\label{prop:mpar}
Suppose that $T$ is an open subscheme of $S$, with $6a$ invertible on $T$. Then there is a  map $\phi: Z(S)\rightarrow \el(T)$ with the following properties.
\begin{enumerate}
 \item[(i)] The map $\phi$ is finite. Furthermore, if $\pm 1$ are the only 12th roots of unity \\
in $K$, then $\phi$ is injective.
 \item[(ii)] Suppose $P\in Z(S)$ and $[E]=\phi(P)$. Then it holds $N_E\leq 2^{6d}3^{3d}N_T^{2}$ and \\
 $h(P)\leq \frac{1}{3}h(a)+8h(E)+2\log\bigl(\max(1,h_F(E))\bigl)+8\kappa+36.$
 \item[(iii)] If, in addition, $T=\sp(\mathcal O[1/(6a)])$, then $N_E\leq 2^{8d}3^{5d}D_KN_S^2r_2(a).$
\end{enumerate}
\end{proposition}

To prove Proposition \ref{prop:mpar} we shall use a lemma which relates heights of elliptic curves.
We recall that $E_K$ denotes the generic fiber of an elliptic curve $E$ over $S$. Let  $W$ be a Weierstrass model of $E_K$ over $B$ with discriminant $\Delta_{W}$, 
see for example  \cite[Section 9.4.4]{liu:ag} for a definition of $W$ and $\Delta_{W}$.
To measure $W$ we take the height 
\begin{equation}\label{def:w}
h(W)=\frac{1}{12}\inf_{\epsilon\in\mathcal O_K^\times} h(\epsilon^{12}c_4^3,\epsilon^{12}c_6^2),
\end{equation}
where $c_4$ and $c_6$ are the usual quantities of a defining Weierstrass equation of $W$, see \cite[p.42]{silverman:aoes}. 
It turns out that the definition of $h(W)$ does not depend on the choice of the defining Weierstrass equation of $W$. We obtain the following lemma.

\begin{lemma}\label{lem:wei}
Suppose that $E$ is an elliptic curve over $S$. Then there exists a Weierstrass model $W$ of $E_K$ over $B$ that satisfies $$h(W)\leq h(E)+\frac{1}{2}\log\left(\max(1,h_F(E))\right)+\kappa+7.$$
\end{lemma}
If $K=\QQ$, then this lemma would follow on calculating the constants in Silverman's \cite[Proposition 2.1, Corollary 2.3]{silverman:arithgeo}.
However, the proof of \cite[Corollary 2.3]{silverman:arithgeo} does not generalize directly to arbitrary $K$, 
since it uses that the ring of integers of $\QQ$ has class number one and unit group $\{\pm 1\}$.
To deal with arbitrary $K$ we apply a classical theorem of Minkowski and a result which is based on estimates for certain fundamental units of $\OK$. 
This leads to a dependence of the constant $\kappa$ on $D_K$, $d$ and on $R_K$, $r_K$, $d$. 
\begin{proof}[Proof of Lemma \ref{lem:wei}]
On combining \cite[p.264]{silverman:aoes}  with a classical result of Minkowski, we obtain a Weierstrass model $W$ of $E_K$ over $B$ of discriminant $\Delta_W$  such that
\begin{equation}\label{eq:amin}
\Delta_{W}\OK=\mathfrak a^{12}\mathfrak D
\end{equation}
for $\mathfrak D$ the minimal discriminant ideal of $E_K$ and $\mathfrak a\subseteq \OK$ an  ideal with $\N(\mathfrak a)^2\leq D_K$. 
For any nonzero $\beta\in \OK$, an application of \cite[Lemma 3]{gyyu:sunits}\footnote{This result relies on estimates for certain fundamental units of $\OK$.} with $n=12$ gives $\epsilon\in\mathcal O_K^\times$ such that $dh(\epsilon^{12}\beta)\leq\log\N(\beta)+12d\kappa-6\log(D_K)$. Hence, on using  (\ref{eq:amin}), we obtain a defining Weierstrass equation of $W$, with quantities $c_4,c_6$ and discriminant $\Delta$, such that
\begin{equation}\label{eq:discup}
dh(\Delta)\leq \log \Delta_E+12d\kappa.
\end{equation}
We write $\Delta_E=\Delta_1\Delta_2$ with $\Delta_1=\exp\bigl(12d(h(E)-h_F(E))\bigl)$  the ``unstable discriminant''  and $\Delta_2=\Delta_E\Delta_1^{-1}$  the $d$-th power of the ``stable discriminant''. 
Let $j$ be the $j$-invariant of $E_K$. Since $c_4,c_6\in\OK$ satisfy $c_6^2=c_4^3-1728\Delta$ and $j=c_4^3/\Delta$, we see
\begin{equation*}
dh(c_4^3,c_6^2)=\log\prod\left( \left|\Delta\right|_\sigma  \max(\left|j\right|_\sigma,\left|j-1728\right|_\sigma,\left|\Delta\right|^{-1}_\sigma)\right),
\end{equation*}
and  Kodaira-N\'eron \cite[p.200]{silverman:aoes} gives $dh(j)=\log \Delta_2+\sum \max(1,\left|j\right|_\sigma)$. 
Here the product and the sum are both taken over all embeddings $\sigma:K\hookrightarrow\mathbb C$ and $\left|\beta\right|_\sigma$ denotes the complex absolute value of $\sigma(\beta)$ for $\beta\in K$.
Then, on splitting the product according to $\left|\Delta\right|^{-1}_\sigma>\left|j\right|_\sigma+1728$ and $\left|\Delta\right|^{-1}_\sigma\leq\left|j\right|_\sigma+1728$,
we deduce from (\ref{eq:discup}) the estimate $$h(c_4^3,c_6^2)\leq \log\Delta_1/d+h(j)+12\kappa+\log(2\cdot1728).$$
Hence, on combining $12h(W)\leq h(c_4^3,c_6^3)$,  (\ref{eq:pellarin}) and $\log\Delta_1+12dh_F(E)=12dh(E)$, we see that  $W$  has the desired property. This completes the proof of Lemma \ref{lem:wei}.
\end{proof}
The proof shows in addition that one can take in Lemma \ref{lem:wei} any Weierstrass model $W$ of $E_K$ over $B$ with $\N(\Delta_{W})\leq D_K^6\Delta_E$. A defining Weierstrass equation of such a $W$ is  called a quasi-minimal Weierstrass equation of $E_K$, see \cite[p.264]{silverman:aoes}.

\begin{proof}[Proof of Proposition \ref{prop:mpar}]
If $Z(S)$ is empty, then all statements are trivial. Hence we may and do assume that $Z(S)$ is not empty. We write $b=-a/1728$. Let $Y$ be the spectrum of $\ZZ[1/6,c_4,c_6,b,1/b]/(1728b-c_4^3+c_6^2)$ for $c_4$ and $c_6$ ``indeterminates''. We observe that
$$t^2=s^3-27c_4s-54c_6$$
defines  an (universal) elliptic curve $\mathcal E$ over $Y$. We take $P\in Z(S)$.  
On using that $Y_T\cong Z_T$, 
we obtain a morphism $T\to Y_T$ induced by $P$. We denote by $E$ the fiber product of $\mathcal E_{Y_T}\to Y_T$ with this morphism  $T\to Y_T$. It follows that $E$ is an elliptic curve over $T$ 
and then 
$P\mapsto [E]$ defines a map $\phi:Z(S)\to \el(T)$.

To prove (i) we observe that $E$ is a Weierstrass model of its generic fiber $E_K$.
Hence we see that if $P'\in Z(S)$ satisfies $\phi(P')=\phi(P)$, then there is $u\in K$ with $u^{4}x(P')=x(P)$ and $u^{6}y(P')=y(P)$, and thus $u^{12}a=a$ since $P,P'\in Z(S)$. Therefore we deduce (i). 

We now show (ii). Let $W$ be the Weierstrass model of $E_K$ over $B$ from Lemma \ref{lem:wei}. We denote by $\Delta,c_4,c_6$ the quantities of a defining Weierstrass equation of $W$, which we constructed in the proof of Lemma \ref{lem:wei}. We point out that one should not confuse these $c_4,c_6\in \OK$ with the ``indeterminates'' which appear in the proof of (i).
On using that $E$ is a Weierstrass model of $E_K$ over $T$, we see that there exists $u\in K$ that satisfies
\begin{equation*}
b=u^{12}\Delta, \ \ \ x(P)=u^{4}c_4. 
\end{equation*}
Thus Lemma \ref{lem:wei}, (\ref{eq:discup}) and Lemma \ref{lem:dehe} lead to an upper bound for $h(P)$ as stated in (ii).
To estimate the conductor $N_E$ of $E$ we take a closed point $v$ of $B$. Let $f_v$ be the conductor exponent of $E_K$ at $v$. 
If $v\in T$, then $f_v=0$ since $E$ is a smooth projective model of $E_K$ over $T$,
and if $v\nmid 6$, then $f_v\leq 2$. Thus (\ref{eq:cond23}) implies an estimate for $N_E$ as claimed in (ii).

To prove (iii) we may and do assume that $T=\sp(\mathcal O[1/(6a)])$. Let $U$ (resp. $U'$) be the set of points $v\in S-T$ with $v\nmid 6$ such that $E_K$ has (resp. has not) semi-stable reduction at $v$. We define $\Omega=\prod_{v\in U}N_v$ and $\Omega'=\prod_{v\in U'}N_v^2$ and then we deduce
$$
N_E\leq 2^{8d}3^{5d}N_S^2\cdot\Omega\cdot\Omega'.
$$
To control the unstable part $\Omega'$ we may and do assume that $U'$ is not empty. We take $v\in U'$. The classification of Kodaira-N\'eron \cite[p.448]{silverman:aoes} gives $v(\Delta)\geq 2$. 
Since $P\in Z(S)$ we see that $E$ extends to a Weierstrass model of $E_K$ over $S$, with discriminant $6^{12}b$. 
Hence, if $W$ is minimal at $v$, then $2\leq v(\Delta)\leq v(b)=v(a).$ Further, (\ref{eq:amin}) implies $\prod N_v^2\leq D_K$ with the product taken over all closed points $v\in B$ with $W$  not minimal at $v$. We conclude
$$\Omega'\leq D_K\prod N_v^2$$
with the product taken over all $v\in U'$ such that $v(a)\geq 2$. To estimate the stable part $\Omega$ we use our assumption that $T=\sp(\mathcal O[1/(6a)])$. 
This assumption implies that any $v\in S-T$ with $v\nmid 6$ satisfies $v(a)\geq 1$.
Therefore we obtain $$\Omega\leq \prod N_v$$ with the product taken over all $v\in U$ such that $v(a)\geq 1$. 
On combining the displayed inequalities, we deduce (iii). This completes the proof of Proposition \ref{prop:mpar}.\end{proof}

We conclude this section with the following remarks.
The elliptic curve $\mathcal E$ over $Y$, 
which appears in the proof of Proposition \ref{prop:mpar}, represents a moduli problem $[\Delta=b]$ on $(Ell)$. Here the moduli problem $[\Delta=b]$ is defined similarly as $[\Delta=1]$ in \cite[p.70]{kama:moduli}, but with 1 replaced by the number $b$ which appears in the proof of Proposition \ref{prop:mpar}.

The above propositions show that to solve $S$-unit and Mordell equations, it suffices to estimate effectively $h(E)$ in terms of $N_E$ for any elliptic curve $E$ over $K$. In this paper we shall prove such estimates for $K=\QQ$, see \cite{rvk:height} for arbitrary number fields $K$.

In the special case of $S$-unit and Mordell equations, it is possible to give ad hoc Par{\v{s}}in constructions  which do not use the moduli problem formalism. For example, ``Frey-Hellegoarch curves'' provide in principle such a construction for $S$-unit equations. However, using the moduli problem formalism gives more conceptual constructions, which generalize several known examples such as  ``Frey-Hellegoarch curves''.

\section{Variation of Faltings heights under isogenies}\label{sec:var}

In this section, we collect results which control the variation of  Faltings heights under isogenies. These results are rather direct consequences of theorems in the literature.

Let $K$ be a number field and let $A$ be an abelian variety over $K$ of dimension $g\geq 1$.  We denote by $h_F(A)$ the stable Faltings height of $A$ and by $h(A)$ the relative Faltings height of $A$, see Section \ref{sec:heights} for the definitions.  The results of Faltings \cite[Lemma 5]{faltings:finiteness} and  Raynaud \cite[Corollaire 2.1.4]{raynaud:abelianisogenies} provide that any $K$-isogeny $\varphi:A\to A'$ of abelian varieties over $K$ satisfies
\begin{equation}\label{eq:hdegdiff}
\lvert h(A)-h(A')\rvert\leq \frac{1}{2}\log \deg(\varphi).
\end{equation}
Let $N_A$ be the conductor of $A$  defined in Section \ref{sec:cond}, let $D_K$ be the absolute value of the discriminant of $K$ over $\QQ$ and let $d=[K:\QQ]$ be the degree of $K$ over $\QQ$.

\begin{lemma}\label{lem:abisogvar}
Suppose $A'$ is an abelian variety defined over $K$ which is $K$-isogenous to $A$.  Then the following statements hold. 
\begin{itemize}
\item[(i)] There exists an effective constant $\mu$, depending only on $g,N_A,d$ and $D_K$, such that $\lvert h_F(A)-h_F(A')\rvert\leq \mu.$
\item[(ii)] If $K=\QQ$ and $g=1$, then $\lvert h(A)-h(A')\rvert\leq \frac{1}{2}\log 163$.
\item [(iii)] Suppose $K=\QQ$ and  $A$ is semi-stable. Then any abelian subvariety $C$ of $A$ satisfies $h(C)\leq h(A)+\frac{g}{2} \log(8\pi^2).$
\end{itemize}
\end{lemma}
The main ingredients for the proof of this lemma are as follows. Raynaud \cite{raynaud:abelianisogenies}  proved Lemma \ref{lem:abisogvar} (i) for semi-stable abelian varieties. His proof relies on refinements of certain arguments in Faltings \cite{faltings:finiteness}; these refinements are due to Par{\v{s}}in and Zarhin.
To prove (i) we reduce the problem to the semi-stable case established in \cite{raynaud:abelianisogenies}. For this reduction we use the semi-stability  criterion of Grothendieck-Raynaud \cite{grra:neronmodels}, the criterion of N\'eron-Ogg-Shafarevich \cite{seta:goodreduction} and Dedekind's discriminant theorem. 
To show (ii) we combine the inequality (\ref{eq:hdegdiff}) with Mazur's \cite{mazur:qisogenies}  classification of cyclic $\QQ$-isogenies of elliptic curves over $\QQ$, see also Kenku \cite{kenku:ellisogenies}. We deduce (iii) from  Bost's explicit lower bound for $h_F$ in (\ref{eq:bost}) and a result of Ullmo-Raynaud given in \cite[Proposition 3.3]{ullmo:j0n}.

\begin{proof}[Proof of Lemma \ref{lem:abisogvar}]
To prove (i) we let $L=K(A[15])$ be the field of definition of the $15$-torsion points of $A$. The semi-stable reduction criterion \cite[Proposition 4.7]{grra:neronmodels} 
shows that $A_L$ is semi-stable. Let $D_L$ be the absolute value of the discriminant of $L$ over $\QQ$ and let $l$ be the relative degree of $L$ over $K$. We denote by $\mathcal T$ the set of finite places of $L$ where $A_L$ has bad reduction. Let $\ell$ and $\ell'$ be the smallest rational primes such that any place in $T$ has residue characteristic different to $\ell$ and $\ell'$.
An application of \cite[Th\'eor\`eme 4.4.9]{raynaud:abelianisogenies} with the $L$-isogenous abelian varieties $A_L$ and $A'_L$ implies 
\begin{equation}\label{eq:raynaud}
\lvert h_F(A_L)-h_F(A'_L)\rvert\leq \mu'
\end{equation}
for $\mu'$ an effective constant depending only on $D_L,l,d,\lvert \mathcal T\rvert,\ell,\ell'$ and $g$. We now estimate these quantities effectively in terms of $g,N_A,d$ and $D_K$.
The criterion of N\'eron-Ogg-Shafarevich \cite[Theorem 1]{seta:goodreduction} implies that $L=K(A[15])$ is unramified over all finite places $v$ of $K$ such that $v\nmid 15$ and such that $A$ has good reduction at $v$. Thus \cite[Lemma 6.2]{rvk:szpiro}, which is based on Dedekind's discriminant theorem, gives 
$$
D_L\leq (D_KN_A)^{l}(15l^{t+2d})^{ld}
$$
for $t$  the number of finite places of $K$ where  $A$ has bad reduction. 
It holds $\lvert \mathcal T\rvert\leq lt$, and  it is known that $l$ can be explicitly controlled in terms of $g$ (see \cite{grra:neronmodels}). Further, the explicit prime number theorem in \cite{rosc:formulas} gives effective upper bounds for $t,\ell$ and $\ell'$ in terms of $N_A$.  
We conclude that $\mu'$ is bounded from above by an effective constant $\mu$ which depends only on $g,N_A,d$ and $D_K$. Then (\ref{eq:raynaud}) and the stability of $h_F$  prove (i).

To show (ii) we assume that  $K=\QQ$ and $g=1$. Let $\varphi:A\to A'$ be  a $\QQ$-isogeny of minimal degree among all $\QQ$-isogenies  $A\to A'$. This isogeny $\varphi$ is cyclic, since otherwise it factors through multiplication by an integer which contradicts the minimality of $\deg(\varphi)$.
Therefore \cite[Theorem 1]{kenku:ellisogenies} gives $\deg(\varphi)\leq 163$ and then (\ref{eq:hdegdiff}) implies (ii).

To prove (iii) we assume that $K=\QQ$ and that $A$ is semi-stable. Let $C$ be an abelian subvariety of $A$. Then there exists a short exact sequence  $$0\to C\to A\to D\to 0$$ of abelian varieties over $\QQ$. 
The semi-stability of $A$ provides that $C$ and $D$ are semi-stable as well, see for example \cite[p.182]{bolura:neronmodels}.
Therefore \cite[Proposition 3.3]{ullmo:j0n} implies that $h(C)\leq h(A)-h(D)+g\log 2$ and  then the   lower bound  for $h(D)$ given in (\ref{eq:bost}) leads to statement (iii). This completes the proof of Lemma \ref{lem:abisogvar}. 
\end{proof}

We point out that Faltings' proof of the Tate conjecture, and its refinement due to Par{\v{s}}in-Zarhin-Raynaud \cite{raynaud:abelianisogenies} which is applied in Lemma \ref{lem:abisogvar} (i), 
both do not use Diophantine approximation or transcendence techniques. On the other hand,
Masser-W\"ustholz \cite{mawu:abelianisogenies} gave a new proof of the Tate conjecture on using their isogeny estimates which rely on transcendence theory. Bost-David (see for example \cite[p.121]{bost:abelianisogenies}) showed that these isogeny estimates are effective, and completely explicit constants are given by Gaudron-R\'emond  \cite{gare:isogenies}. For example, if  $A'$ is an abelian variety over $K$ which is $K$-isogenous to $A$, then \cite[Th\'eor\`eme 1.4]{gare:isogenies}  combined with  (\ref{eq:hdegdiff}) gives  
\begin{equation}\label{eq:garehdiff}
\lvert h(A)-h(A')\rvert\leq 2^{10}g^3\log \bigl((14g)^{64g^2}d\max(h_F(A),\log d,1)^2\bigl).
\end{equation}
We remark that on calculating the constant $\mu$ in Lemma \ref{lem:abisogvar} (i) explicitly, it turns out that  Lemma \ref{lem:abisogvar} (i) improves (\ref{eq:garehdiff}) in some cases, and vice versa in other cases.

\section{Modular forms and modular curves}\label{sec:modular} 
In the first part of this section, we collect results from the theory of cusp forms. In the second part, we work out an explicit upper bound for the modular degree of newforms with rational Fourier coefficients. In the third part, we give explicit upper bounds for the stable Faltings heights of the Jacobians of certain classical modular curves.

\subsection{Cusp forms}\label{sec:cuspforms}
We begin to collect results for cusp forms which are given, for example, in the books of Shimura \cite{shimura:automorphic} or Diamond-Shurman \cite{dish:modular}. We take an integer $N\geq 1$ and we consider the classical congruence subgroup $\Gamma_0(N)\subset \textnormal{SL}_2(\ZZ)$. Let $S_2(\Gamma_0(N))$  be the complex vector space of cusp forms of weight 2 with respect to $\Gamma_0(N)$. 
We denote by  $X_0(N)$ and $X(1)$ smooth, projective and geometrically connected models over $\QQ$ of the modular curves associated to $\Gamma_0(N)$ and $\textnormal{SL}_2(\ZZ)$ respectively. Throughout Section \ref{sec:modular} we denote by $d$ the degree of the natural projection  $X_0(N)\to X(1)$. 
The dimension of $S_2(\Gamma_0(N))$  coincides with the genus $g$ of $X_0(N)$. 
Furthermore, it holds 
\begin{equation}\label{eq:destimate}
g\leq d/12 \ \ \ \textnormal{ and } \ \ \  d=
N\prod(1+1/p)  \textnormal{ if } N\geq 2
\end{equation}
with the product taken over all rational primes $p$ which divide $N$. 
Let $f\in S_2(\Gamma_0(N))$ be a nonzero cusp form. If  $\textnormal{div}(f)$ denotes the usual rational divisor on $X_0(N)_\CC$ of $f$, then 
\begin{equation}\label{eq:rr}
\deg(\textnormal{div}(f))=d/6
\end{equation}
for $\deg(\textnormal{div}(f))\in\QQ$ the degree of $\textnormal{div}(f)$.
For any integer $n\geq 1$, we denote by $a_n(f)$ the $n$-th Fourier coefficient of $f$. We say that $f$ is normalized if $a_1(f)=1$.

We next review properties of the basis of $S_2(\Gamma_0(N))$ constructed by Atkin-Lehner in  \cite[Theorem 5]{atle:hecke}. 
Let $S_2(\Gamma_0(N))^{\textnormal{new}}$ be the new subspace of $S_2(\Gamma_0(N))$ and let  $S_2(\Gamma_0(N))^{\textnormal{old}}$ be the old subspace of $S_2(\Gamma_0(N))$. 
There is a decomposition
$$S_2(\Gamma_0(N))=S_2(\Gamma_0(N))^{\textnormal{new}}\oplus S_2(\Gamma_0(N))^{\textnormal{old}}$$ 
which is orthogonal with respect to the Petersson inner product $(\cdot \,,\cdot)$ on $S_2(\Gamma_0(N))$.  
We say that $f$ is a newform of level $N$ if $f\in S_2(\Gamma_0(N))^{\textnormal{new}}$ is  normalized and if $f$ is an eigenform for all Hecke operators on $S_2(\Gamma_0(N))$. 
The set $\mathcal B^\textnormal{new}$ of newforms of level $N$ is an orthogonal basis of $S_2(\Gamma_0(N))^{\textnormal{new}}$ with respect to $(\cdot \, ,\cdot)$. 
Moreover, there exists a basis $\mathcal B^{\textnormal{old}}$ of $S_2(\Gamma_0(N))^{\textnormal{old}}$ with the property that any $f\in \mathcal B^{\textnormal{old}}$ takes the form
\begin{equation}\label{eq:oldbasis}
f(\tau)=f_M(m\tau), \ \ \tau\in \mathbb C, \ \textnormal{im}(\tau)>0
\end{equation}
with $M\in\ZZ_{\geq 1}$ a proper divisor of $N$, with $m\in \ZZ_{\geq 1}$ a divisor of $N/M$  and with $f_M$ a newform of level $M$. Conversely, any $f\in S_2(\Gamma_0(N))$ which is of the form (\ref{eq:oldbasis}) is in $\mathcal B^{\textnormal{old}}$.
We say that $\mathcal B=\mathcal B^{\textnormal{new}}\cup\mathcal B^{\textnormal{old}}$ is the Atkin-Lehner basis for $S_2(\Gamma_0(N))$. 

\subsection{Modular degree}\label{sec:modulardegree}

Let $f\in S_2(\Gamma_0(N))$ be a newform of level $N\geq 1$, with all Fourier coefficients rational integers. In this section, we estimate the modular degree of $f$ in terms of $N$. 

We begin with the definition of the modular degree $m_f$ of $f$.
Let $J_0(N)=\textnormal{Pic}^0(X_0(N))$ be the Jacobian variety of $X_0(N)$. We denote by $\mathbb T_\ZZ$ the subring of the endomorphism ring of $J_0(N)$, which is generated over $\ZZ$ by the usual Hecke operators $T_n$ for all $n\in \ZZ_{\geq 1}$. 
Let $I_f$ be the kernel of the ring homomorphism $\mathbb T_\ZZ\to \ZZ[\{a_n(f)\}]$ which is induced by $T_n\mapsto a_n(f)$.
The image $I_fJ_0(N)$ of $J_0(N)$ under $I_f$ is connected and the quotient 
\begin{equation}\label{def:ef}
E_f=J_0(N)/I_fJ_0(N)
\end{equation}
is an abelian variety over $\QQ$ of dimension $[\QQ(\{a_n(f)\}):\QQ]=1$. 
The cusp $\infty$ of $X_0(N)$ is a $\QQ$-rational point of $X_0(N)$. We denote by $\iota:X_0(N)\hookrightarrow J_0(N)$ the usual embedding over $\QQ$, which maps the cusp $\infty$ to the zero element of $J_0(N)$.
On composing the embedding $\iota$ with the natural projection $J_0(N)\to J_0(N)/I_fJ_0(N)$, we obtain a finite morphism $$\varphi_f: X_0(N)\to E_f.$$ 
The modular degree $m_f$ of $f$ is defined as the degree of the finite morphism $\varphi_f$. 

To estimate $m_f$ we shall use properties of the congruence number $r_f$ of $f$. We recall that $r_f$ is the largest integer such that there exists a cusp form  $f_c\in S_2(\Gamma_0(N))$, with rational integer Fourier coefficients, which satisfies
\begin{equation}\label{eq:congruencedef}
(f,f_c)=0 \ \textnormal{ and } \ a_n(f)\equiv a_n(f_c)\bmod (r_f), \ n\geq 1.
\end{equation}
It is known that the modular degree $m_f$ and the congruence number $r_f$ are related. For example, the arguments in Zagier's article \cite[Section 5]{zagier:modularparam} 
give
\begin{equation}\label{eq:div}
m_f\mid r_f.
\end{equation}
We note that Zagier's arguments are based on ideas of Ribet, see \cite{ribet:congruences} and the references therein. In fact Zagier \cite[p.381]{zagier:modularparam}  attributes the divisibility result (\ref{eq:div}) to Ribet. Further, we mention that  Cojocaru-Kani \cite[Theorem 1.1]{coka:modulardegree} gave a detailed exposition of a proof of (\ref{eq:div}) and Agashe-Ribet-Stein   generalized (\ref{eq:div}) in \cite[Theorem 3.6]{agrist:congruence}.

For any real number $r$, we define $\lfloor r\rfloor=\max(m\in \ZZ, m\leq r)$, and for any integer $n$, we denote by $\tau(n)$ the number of positive integers which divide $n$. The author is grateful to Richard Taylor for proposing a first strategy to prove an upper bound for $m_f$.  

\begin{lemma}\label{lem:moddeg}
Let $N\geq 1$ be an integer. Suppose $f\in S_2(\Gamma_0(N))$ is a newform of level $N$, with all Fourier coefficients rational integers. Then the following statements hold.
\begin{itemize}
\item[(i)] The modular degree $m_f$ of $f$ satisfies $\log m_f\leq \frac{1}{2}N(\log N)^2.$
\item[(ii)] More precisely, let $g$ be the genus of $X_0(N)$ and let $d$ be the degree of the natural projection $X_0(N)\to X(1)$. Then there exists a subset $J\subset \{1,\dotsc,\lfloor d/6+1\rfloor\}$ of cardinality $g$, which is independent of $f$, such that $m_f\leq g!\prod_{j\in J}\tau(j)j^{1/2}.$
\end{itemize}
\end{lemma}
\begin{proof}
We first show (ii). It follows from (\ref{eq:div}) that $m_f\leq r_f$. To estimate $r_f$ we reduce the problem to solve (by Cramer's rule) explicitly a system of linear Diophantine equations.

Let $I=\{1,\dotsc,g\}$, let $J$ be a finite non-empty set of positive integers and put $\delta=\lvert J\rvert$.  We write $l=\lfloor d/6+1\rfloor$ and we denote by  $\mathcal B=\{f_i, i\in I\}$ the Atkin-Lehner basis for $S_2(\Gamma_0(N))$, see Section \ref{sec:cuspforms}.  To show that the linear morphism $F(J)=(a_j(f_i)): \mathbb C^{\delta}\to \mathbb C^g$ is surjective for $J=\{1,\dotsc,l\}$, we assume the contrary and deduce a contradiction. If $F(J)$ is not surjective for $J=\{1,\dotsc,l\}$, then we obtain a nonzero $f_0\in  S_2(\Gamma_0(N))$ with Fourier expansion
$\sum_{n\geq l}a_n(f_0)q^n$.
Hence, $f_0$ vanishes at $\infty$ of order at least $l>d/6$ and this contradicts (\ref{eq:rr}). 
We conclude that $F(J)$ is surjective for $J=\{1,\dotsc,l\}$. Therefore we may and do take  $J\subset\{1,\dotsc,l\}$ such that $F=F(J)$ is an isomorphism.

We claim that $r_f\leq \lvert \det(F)\rvert$. To verify this claim we take $(k_i)\in \mathbb C^{g}$ such that $f_c\in S_2(\Gamma_0(N))$ from (\ref{eq:congruencedef}) takes the form
$
f_c=\sum_i k_if_i.
$
Properties of $\mathcal B$ show that we may and do take $f_1=f$ and that $( f,f_i)=0$ for any $i\geq 2$. This implies that $k_1=0$, since $( f,f_c)=0$ by (\ref{eq:congruencedef}). Therefore, on comparing Fourier coefficients, we see that $a_j(f_c)=\sum_{i\geq 2} k_ia_j(f_i)$ for all $ j\in J.$ Then (\ref{eq:congruencedef}) gives $y=(y_j)\in \ZZ^{g}$ such that any  $x=(x_j)\in \mathbb C^{g}$ satisfies
\begin{equation}\label{eq:fundamental}
\sum_{i\geq 2}k_i(f_i,x)=(f_c,x)=(f,x)+(y,x)r_f,
\end{equation}
where $(h,x)=\sum_{j\in J}a_j(h)x_j$ for $h\in S_2(\Gamma_0(N))$ and  $(y,x)=\sum_{j\in J}y_jx_j$. We write $b=(-1,0,\dotsc,0)\in \mathbb C^{g}$. It follows that if $x=(x_j)\in\mathbb C^{g}$ satisfies
$
F(x)=b
$, then $(f,x)=-1$, and $(f_c,x)=0$ by the first equality of (\ref{eq:fundamental}). Hence, the second equality of (\ref{eq:fundamental}) shows that any solution $x=(x_j)\in\mathbb C^{g}$ of
$
F(x)=b$
satisfies  
\begin{equation}\label{eq:rf}
1=(y,x)r_f.
\end{equation}
The determinant $\det(F)$ of the isomorphism $F$ is nonzero. Thus Cramer's rule gives  $\xi\in \ZZ[\{a_j(f_i)\}]^g$  
such that the unique solution $x=F^{-1}(b)$ of $F(x)=b$ takes the form
\begin{equation}\label{eq:x}
x=\xi\det(F)^{-1}.
\end{equation}
To prove that $\det(F)^2\in\ZZ$ we use (\ref{eq:oldbasis}). It gives that $a_j(f_i)$ is a coefficient of a newform. Thus it is an eigenvalue of a certain Hecke operator. This implies that all $a_j(f_i)$ are algebraic integers. 
Hence $\det(F)$ and all entries of $\xi$ are algebraic integers. Further, Galois conjugates of newforms are newforms of the same level. 
Therefore, on using properties of the basis $\mathcal B$ discussed in Section \ref{sec:cuspforms}, we see that any element $\sigma$ of the absolute Galois group of $\QQ$ ``permutes'' the rows of the matrix $F$. Hence, we get that any such $\sigma$ satisfies $\sigma(\det(F))=\pm\det(F)$ and 
we deduce that $\det(F)^2\in\ZZ$ as desired. Then the formulas
(\ref{eq:rf}) and (\ref{eq:x}) imply  $r_f^2\mid \det(F)^2$ which proves our claim $r_f\leq \lvert\det(F)\rvert$.

To estimate $\lvert\det(F)\rvert$ we use the Ramanujan-Petersson bounds for Fourier coefficients, which hold in particular for any newform, and thus for all $f_i\in \mathcal B$ by (\ref{eq:oldbasis}). 
These bounds imply that $\det(F)\leq g!\prod_{j\in J}\tau(j)j^{1/2}$ 
and then the above inequalities give (ii). 

It remains to prove (i). Any elliptic curve over $\QQ$ has conductor at least $11$.  
Therefore we may and do assume that $N\geq 11$. 
Next, we observe that any integer $n\geq 1$ satisfies the elementary inequalities: $\frac{1}{n}\sum_{k=1}^n\tau(k)\leq 1+\log n$ and $\prod (1+1/p)\leq 1+ (\log n)/(2\log 2)$ with the product taken over all rational primes $p$ which divide $n$. Further, (\ref{eq:destimate}) shows that $2g\leq \lfloor 1+d/6\rfloor=l$ and hence (ii) implies that  $m_f\leq (g!l!)^{1/2}\prod\tau(j)$ with the product taken over the elements $j$ of a set $J\subset \{1,\dotsc,l\}$ of cardinality $g$. Then the above inequalities and (\ref{eq:destimate}) lead to (i). This completes the proof of Lemma \ref{lem:moddeg}.
\end{proof}
Frey  \cite[p.544]{frey:ternary} remarked without proof that it is easy  to show the asymptotic bound $\log m_f\ll N\log N$. 
It seems that this estimate is still very far from being optimal. In fact Frey \cite{frey:linksulm} and Mai-Murty \cite{mamu:weakmodular} showed that a certain polynomial upper bound for $m_f$ in terms of $N$ is equivalent to a certain version of the $abc$-conjecture.  

The above proof shows in addition that the inequalities of Lemma \ref{lem:moddeg} hold with $m_f$ replaced by the congruence number $r_f$ of $f$. We note that Murty \cite[Corollary 6]{murty:strongmodular} used a similar method to prove a slightly weaker upper bound for $r_f$ in terms of $N$. Further, we mention that Agashe-Ribet-Stein proved in \cite[Theorem 2.1]{agrist:congruence} that any rational prime number $p$, with $\textnormal{ord}_p(N)\leq 1$, satisfies $\textnormal{ord}_p(m_f)=\textnormal{ord}_p(r_f)$. Moreover, they conjectured in \cite[Conjecture 2.2]{agrist:congruence} that $\textnormal{ord}_p(\frac{r_f}{m_f})\leq \frac{1}{2}\textnormal{ord}_p(N)$ for all rational prime numbers $p$.

\subsection{Faltings heights of Jacobians of modular curves}\label{sec:hf}
In this section, we give explicit upper bounds for the stable Faltings heights of the Jacobians of certain classical modular curves in terms of their level. These upper bounds are based on a result of Javanpeykar given in \cite{javanpeykar:belyi}.

We begin to state the result of Javanpeykar. Let $X$ be a smooth, projective and connected curve over $\bar{\QQ}$ of genus $g$, where $\bar{\QQ}$ is an algebraic closure of $\QQ$. We denote by $\mathbb P^1$  the projective line over $\bar{\QQ}$ and we let $\mathcal D$ be the set of degrees of finite morphisms $X\to \mathbb P^1$  which are unramified outside $0,1,\infty$. Belyi's theorem \cite{belyi:theorem} shows that $\mathcal D$ is non-empty. The Belyi degree $\deg_B(X)$ of $X$ is defined by
$\deg_B(X)=\min \mathcal D.$ 
Let $\textnormal{Pic}^0(X)$ be the Jacobian of $X$, and let $h_F$ be the stable Faltings height defined in Section \ref{sec:heights}. We recall that $h_F(0)=0$ and then Javanpeykar's inequality \cite[Theorem 1.1.1]{javanpeykar:belyi} gives  
$$
h_F(\textnormal{Pic}^0(X))\leq 13\cdot 10^6\deg_B(X)^5g.
$$
We point out that $h_F(\textnormal{Pic}^0(X))$ is well-defined, since the height $h_F$ is stable.
Let $\Gamma\subset \textnormal{SL}_2(\ZZ)$ be a congruence subgroup. The associated modular curve has a smooth, projective and connected model $X(\Gamma)$ over $\bar{\QQ}$. Let $g_\Gamma$ be the genus of $X(\Gamma)$, and let $\epsilon_\infty$ be the number of cusps of $X(\Gamma)$. The inclusion $\Gamma\subset \textnormal{SL}_2(\ZZ)$ induces a natural projection $X(\Gamma)\to X(1)_{\bar{\QQ}}$ and the degree $d_\Gamma$ of this projection satisfies 
\begin{equation}\label{eq:index}
g_\Gamma\leq 1+\frac{d_\Gamma}{12}-\frac{\epsilon_\infty}{2}, \ \ \ \ \ d_\Gamma=\begin{cases}[\textnormal{SL}_2(\ZZ):\Gamma] & \textnormal{if } \textnormal{-id}\in \Gamma,\\ [\textnormal{SL}_2(\ZZ):\Gamma]/2 & \textnormal{if } \textnormal{-id}\notin \Gamma, \end{cases}
\end{equation}
where $[\textnormal{SL}_2(\ZZ):\Gamma]$ denotes the index of the subgroup $\Gamma\subset\textnormal{SL}_2(\ZZ)$ and $\textnormal{id}\in \textnormal{SL}_2(\ZZ)$ denotes the identity.
Furthermore, the projection $X(\Gamma)\to X(1)_{\bar{\QQ}}$ ramifies at most over the two elliptic points of $X(1)_{\bar{\QQ}}$ or over the cusp of $X(1)_{\bar{\QQ}}$, 
 and it holds that $X(1)_{\bar{\QQ}}\cong \mathbb P^1$. 
Therefore it follows that $\deg_B(X(\Gamma))\leq d_\Gamma$ 
and then the  displayed estimate for $h_F(\textnormal{Pic}^0(X))$ implies
\begin{equation}\label{eq:ari}
h_F(J(\Gamma))\leq 13\cdot 10^6d_\Gamma^5g_\Gamma
\end{equation}
for $J(\Gamma)=\textnormal{Pic}^0(X(\Gamma))$  the Jacobian of $X(\Gamma)$. 

For any integer $N\geq 1$, we consider the classical congruence subgroups $\Gamma_1(N)\subset \textnormal{SL}_2(\ZZ)$ and $\Gamma(N)\subset \textnormal{SL}_2(\ZZ)$, and to ease notation we write $J_1(N)=J(\Gamma_1(N))$ and $J(N)=J(\Gamma(N))$. Further, we let $J_0(N)$ be the modular Jacobian defined in Section \ref{sec:modulardegree}. On combining the above results, we obtain the following lemma.

\begin{lemma}\label{lem:hj1n}
If $N\geq 1$ is an integer, then $$h_F(J_0(N))\leq 7\cdot 10^7(N\log N)^6, \ \ \ h_F(J_1(N))\leq 17\cdot 10^3 N^{12} , \ \ \ h_F(J(N))\leq 17\cdot 10^3N^{18}.$$
\end{lemma}
\begin{proof}
We recall that $h_F(0)=0$. Hence, to prove the claimed inequalities, we may and do assume that the Jacobians are non-trivial. On combining (\ref{eq:destimate})  and (\ref{eq:ari}), we obtain an upper bound for $h_F(J_0(N))$ as stated. 
For $\Gamma=\Gamma_1(N)$ or $\Gamma=\Gamma(N)$ there exist standard formulas which express $[\textnormal{SL}_2(\ZZ):\Gamma]$ and $\epsilon_\infty$ in terms of $N$, see for example \cite{shimura:automorphic} or \cite{dish:modular}. 
These formulas together with (\ref{eq:index}) and (\ref{eq:ari}) imply upper bounds for $h_F(J_1(N))$ and $h_F(J(N))$ as claimed. This completes the proof of Lemma \ref{lem:hj1n}.
\end{proof}

To conclude this section we discuss results in the literature which are related to Lemma \ref{lem:hj1n}. We begin with a theorem of Ullmo and we put $g=g_{\Gamma_0(N)}$. If $N\geq 1$ is a square-free integer, then  \cite[Th\'eor\`eme 1.2]{ullmo:j0n} gives the asymptotic upper bound 
\begin{equation}\label{eq:ullmo}
h_F(J_0(N))\leq \frac{g}{2}\log N+o(g\log N).
\end{equation}
Further, if $N\geq 1$ is a square-free integer, with $2\nmid N$ and $3\nmid N$,  then Jorgenson-Kramer provide in \cite[Theorem 6.2]{jokr:delta} the asymptotic formula
\begin{equation}\label{eq:jk}
h_F(J_0(N))=\frac{g}{3}\log N+o(g\log N).
\end{equation}
On combining (\ref{eq:index}) with the above displayed results, one can slightly improve the bounds of Lemma \ref{lem:hj1n} for special integers $N$. However, our proofs of the Diophantine results in the following sections require  bounds for all integers $N\geq 1$ and thus the above discussed results of Ullmo and Jorgenson-Kramer are not sufficiently general for our purpose. 
\section{Height and conductor of elliptic curves over $\QQ$}\label{sec:hcell}

In the first part of this section, we give explicit exponential versions of Frey's height conjecture and of Szpiro's discriminant conjecture for elliptic curves over $\QQ$. We also derive an effective version 
of Shafarevich's conjecture for elliptic curves over $\QQ$. 
In the second part,  we prove Propositions \ref{prop:he} and  \ref{prop:shaf} on combining the Shimura-Taniyama conjecture with lemmas obtained in previous sections.

\subsection{Height, discriminant and conductor inequalities}\label{subsec:he}

Let $E$ be an elliptic curve over $\QQ$. We denote by $N_E$ the conductor of $E$, and we denote by $h(E)$  the relative Faltings height of $E$. See Section \ref{sec:heights} for the definitions of $N_E$ and $h(E)$. 
We now can state the following proposition which gives an exponential version of Frey's height conjecture \cite[p.39]{frey:linksulm} for all elliptic curves over $\QQ$.

\begin{proposition}\label{prop:he}
If $E$ is an elliptic curve over $\QQ$, then $$h(E)\leq \frac{1}{4}N_E(\log N_E)^2+9.$$
\end{proposition}
Let $K$ be a number field. On using a completely different method, which is based on the theory of logarithmic forms, we established in \cite[Theorem 2.1]{rvk:height} a version of Proposition \ref{prop:he} for arbitrary elliptic curves over $K$. However, in the case of elliptic curves $E$ over $\QQ$, \cite[Theorem 2.1]{rvk:height} provides only the weaker inequality $h(E)\leq (25N_E)^{162}$.

As in Section \ref{sec:explicitconstr}, we denote by $\Delta_E$ the norm of the usual minimal discriminant ideal of $E$. Our next result provides an explicit exponential 
version of Szpiro's discriminant conjecture \cite[p.10]{szpiro:spe} for elliptic curves over $\QQ$.  

\begin{corollary}\label{cor:de}
Any elliptic curve $E$ over $\QQ$ satisfies $$\log\Delta_E\leq 3N_E(\log N_E)^2+124.$$
\end{corollary}
\begin{proof}
This follows from Proposition \ref{prop:he}, since $\log\Delta_E\leq 12h(E)+16$ by Lemma \ref{lem:dehe}.  
\end{proof}
On combining Arakelov theory for arithmetic surfaces with the theory of logarithmic forms, we obtained in \cite{rvk:szpiro}  versions of Corollary \ref{cor:de} for all hyperelliptic (and certain more general) curves over $K$. In the case of elliptic curves $E$ over $\QQ$, we see that Corollary \ref{cor:de} improves the inequality $\log \Delta_E\leq (25N_E)^{162}$  
provided by \cite[Theorem 3.3]{rvk:szpiro}.

To state our next corollary we denote by $h(W)$ the height of a Weierstrass model $W$ of $E$ over $\sp(\ZZ)$, defined in (\ref{def:w}).
Let $S$ be a non-empty open subscheme of $\sp(\ZZ)$ and
$$
\nu_S=12^3N_S^2, \ \ \ \ N_S=\prod p 
$$
with the product taken over all rational primes $p$ not in $S$. We say that an arbitrary elliptic curve $E$ over $\QQ$ has good reduction over $S$ if $E$ has good reduction at all rational primes in $S$\footnote{This definition is equivalent to the  classical notion of good reduction outside a finite set $\mathcal S$ of rational prime numbers. Indeed $S=\sp(\ZZ)-\mathcal S$ has the structure of a non-empty open subscheme of $\sp(\ZZ)$, and $E$ has good reduction outside $\mathcal S$ if and only if $E$ has good reduction over $S$.}.  It turns out that the number $\nu_S$ has the property that any elliptic curve $E$ over $\QQ$, with good reduction over $S$, has conductor $N_E$ dividing $\nu_S$.  The Diophantine inequality in Proposition \ref{prop:he} leads to the following fully effective version of the Shafarevich conjecture \cite{shafarevich:conjecture} for elliptic curves over $\QQ$.

\begin{corollary}\label{cor:shaf}
If $[E]$ is a $\QQ$-isomorphism class of elliptic curves over $\QQ$ with good reduction over $S$, then there exists a Weierstrass  model $W$ of $E$ over $\sp(\ZZ)$ that satisfies
$$h(W)\leq \frac{1}{2}\nu_S(\log \nu_S)^2.$$
In particular, there exist only finitely many $\QQ$-isomorphism classes of elliptic curves over $\QQ$ with good reduction over $S$ and these classes can be determined effectively.
\end{corollary}
\begin{proof}
We take a $\QQ$-isomorphism class $[E]$ of elliptic curves over $\QQ$, with good reduction over $S$. Lemma \ref{lem:wei} gives a Weierstrass model $W$ of $E$ over $\sp(\ZZ)$ that satisfies $$h(W)\leq h(E)+\frac{1}{2}\log\left(\max(1,h_F(E))\right)+7,$$
where $h_F(E)$ is the stable Faltings height of $E$.
Further, it holds that $h_F(E)\leq h(E)$ and (\ref{eq:cond23}) leads to $N_E\mid \nu_S$. 
Thus Proposition \ref{prop:he} implies Corollary \ref{cor:shaf}.
\end{proof}
The first effective version of the Shafarevich conjecture for elliptic curves over $\QQ$ is due to Coates \cite[p.426]{coates:shafarevich}. He applied the theory of logarithmic forms. This  theory is also used in \cite[Theorem]{rvk:hyperelliptic} which provides a version of Corollary \ref{cor:shaf} for arbitrary hyperelliptic curves over $K$. In the case of elliptic curves over $\QQ$, Corollary \ref{cor:shaf} improves the actual best bound $h(W)\leq (2N_S)^{1296}$ which was obtained in \cite[Theorem]{rvk:hyperelliptic}. 

We mention that \cite[Section 2]{rvk:height} gives in addition effective asymptotic versions of the above results: $h(E)\ll_\epsilon N_E^{21+\epsilon}$,  $\log \Delta_E\ll_\epsilon N_E^{21+\epsilon}$ and $h(W)\ll_\epsilon N_S^{21+\epsilon}$. 
Further, it is discussed in \cite[Section 2]{rvk:height} that the exponent $21+\epsilon$ is optimal for the known methods which are based on the theory of logarithmic forms. Thus these methods can not produce inequalities as strong as those in Proposition \ref{prop:he}, Corollary \ref{cor:de} and Corollary \ref{cor:shaf}.

We denote by $N(S)$ the number of $\QQ$-isomorphism classes of elliptic curves over $\QQ$, with good reduction over $S$. The explicit height estimate in Corollary \ref{cor:shaf} implies  an explicit upper bound for $N(S)$. However, this bound would be exponential in terms of $\nu_S$. The following Proposition \ref{prop:shaf} gives an explicit upper bound for $N(S)$ which is polynomial in terms of $\nu_S$. The proof uses inter alia the Shimura-Taniyama conjecture  and a result of Mazur-Kenku \cite{kenku:ellisogenies} on $\QQ$-isogeny classes of elliptic curves.
\begin{proposition}\label{prop:shaf}
It holds that $N(S)\leq \frac{2}{3}\nu_S\prod_{p\mid \nu_S}(1+1/p)$ with the product taken over all rational primes $p$ which divide $\nu_S$. 
\end{proposition} 
We now discuss bounds for $N(S)$ in the literature. The estimate $N(S)\ll_\epsilon N_S^{1/2+\epsilon}$ was obtained by Brumer-Silverman \cite[Theorem 1]{brsi:number} and Poulakis established in \cite[Theorem 2]{poulakis:corrigendum} an explicit upper bound for $N(S)$. One observes that Proposition \ref{prop:shaf} is better than Poulakis' result when $N_S\leq 2^{65}$,
and is worse when $N_S$ is sufficiently large. However, for sufficiently large $N_S$ the actual best estimate is due to Ellenberg, Helfgott and Venkatesh \cite{heve:integralpoints,elve:classgroup}. Namely, on refining the proof of \cite[Theorem 4.5]{heve:integralpoints} with the upper bound in \cite[Proposition 3.4]{elve:classgroup}, one obtains 
\begin{equation}\label{eq:ehv}
N(S)\ll N_S^{0.1689}. 
\end{equation}
Furthermore, Brumer-Silverman \cite{brsi:number} observed that one can considerably improve (\ref{eq:ehv}) on assuming $(*)$: If $E$ is an elliptic curve over $\QQ$, with vanishing $j$-invariant, then the $L$-function $L(E,s)$ of $E$ satisfies the ``Generalized Riemann Hypothesis" and the rank of $E(\QQ)$ is at most the order of vanishing of $L(E,s)$ at $s=1$. More precisely, \cite[Theorem 4]{brsi:number} gives that $(*)$ implies $N(S)\ll_\epsilon N_S^\epsilon$;  notice that the ``Generalized Riemann Hypothesis" together with the ``Birch and Swinnerton-Dyer conjecture" implies $(*)$.

We point out that the methods of Brumer-Silverman,  Helfgott-Venkatesh and Poulakis are entirely different from the method which is used in the proof of Proposition \ref{prop:shaf}. For example, to obtain Diophantine finiteness, they use the following tools: Brumer-Silverman \cite{brsi:number} apply an estimate of Evertse-Silverman \cite{evsi:uniformbounds} based on Diophantine approximation, 
Helfgott-Venkatesh \cite{heve:integralpoints} use a bound of Hajdu-Herendi \cite{hahe:elliptic} relying on the theory of logarithmic forms, and Poulakis \cite{poulakis:corrigendum} applies an estimate of Evertse \cite{evertse:sunits} based again on Diophantine approximation.

\subsection{Proof of Propositions \ref{prop:he} and \ref{prop:shaf}}\label{sec:proofhe}

Our main tool in the proof of Proposition \ref{prop:he} is the Shimura-Taniyama conjecture. Building on the work of Wiles \cite{wiles:modular} and Taylor-Wiles \cite{taywil:modular},  Breuil-Conrad-Diamond-Taylor \cite{breuil:modular} proved this conjecture for all elliptic curves over $\QQ$. The modularity result in \cite{breuil:modular} implies the following  version of the Shimura-Taniyama conjecture.
For any integer $N\geq 1$, let $X_0(N)$ be the modular curve defined in Section \ref{sec:cuspforms}. Suppose $E$ is an elliptic curve over $\QQ$ with conductor $N=N_E$. 
Then there exists a finite morphism
\begin{equation}\label{eq:st}
X_0(N)\to E
\end{equation}
of curves over $\QQ$. We mention that the implication \cite{breuil:modular} $\Rightarrow$ (\ref{eq:st}) uses inter alia the Tate conjecture which was established by Faltings in \cite{faltings:finiteness}.

To prove Proposition \ref{prop:he} we use a strategy of Frey \cite{frey:linksulm}. In the first part, we apply Lemma \ref{lem:abisogvar} (ii) to pass to an elliptic curve  over $\QQ$, which is $\QQ$-isogenous to $E$ and which is an ``optimal quotient''. In the second part, we consider a formula which involves inter alia $h(E)$, the modular degree of the newform attached to $E$ by (\ref{eq:st}), and the ``Manin constant'' of $E$. In the third part, we estimate the quantities which appear in this formula. Here we use inter alia the bound for the modular degree in Lemma \ref{lem:moddeg} and a result of Edixhoven in \cite{edixhoven:manin} which says that the ``Manin constant'' of $E$ is an integer.

\begin{proof}[Proof of Proposition \ref{prop:he}]  Let $E$ be an elliptic curve over $\QQ$ with conductor $N=N_E$.   

1. The version of the Shimura-Taniyama conjecture in (\ref{eq:st}) gives a finite morphism
$\varphi: X_0(N)\to E$
of smooth projective curves over $\QQ$. We recall that $J_0(N)=\textnormal{Pic}^0(X_0(N))$ denotes the Jacobian of $X_0(N)$. By Picard functoriality, the morphism $\varphi$ induces a surjective $\QQ$-morphism of abelian varieties 
\begin{equation}\label{eq:j0nq}
\psi: J_0(N)\to E.
\end{equation}
Let $A$ be the identity component of the kernel of $\psi$. It is an abelian subvariety of $J_0(N)$. 
Thus, on using for example the standard argument via Poincar\'e's reducibility theorem, we obtain an elliptic curve $E'$ over $\QQ$ which is $\QQ$-isogenous to $E$ and a surjective morphism $\psi':J_0(N)\to E'$ of abelian varieties over $\QQ$ with kernel $A$.  
Since $E$ and $E'$ are $\QQ$-isogenous, it follows that $E'$ has conductor $N$ and Lemma \ref{lem:abisogvar} (ii) gives 
\begin{equation}\label{eq:mazurbound}
\left|h(E)-h(E')\right|\leq \frac{1}{2}\log 163.
\end{equation}
The kernel of  $\psi':J_0(N)\to E'$ is $A$, which is connected. An elliptic curve over $\QQ$ with this property is called an optimal quotient of $J_0(N)$ or a (strong) Weil curve. As in Section \ref{sec:modulardegree}, we denote by $\iota:X_0(N)\hookrightarrow J_0(N)$  the usual embedding which maps $\infty$ to the zero element of $J_0(N)$. To simplify the exposition we write $E$ and $\varphi$ for $E'$ and $\psi'\circ\iota$ respectively.

2. It is known by Frey \cite[p.45-47]{frey:linksulm} that the degree $\deg(\varphi)$ of $\varphi$ is related to $h(E)$.  
A precise relation can be established as follows. We denote by $\mathcal E$  the N\'eron model of $E$ over $B=\sp(\ZZ)$.
Since $\ZZ$ is a principal ideal domain, the line bundle $\omega=\omega_{\mathcal E/B}$ on $B$ from Section \ref{sec:heights} takes the form $\omega\cong\alpha \ZZ$ with a global differential one form $\alpha$ of $\mathcal E$. 
Then, on recalling the definition of the relative Faltings height $h(E)$ in Section \ref{sec:heights}, we compute
$$h(E)=-\frac{1}{2}\log\left(\frac{i}{2}\int_{E(\mathbb C)}\alpha\wedge \overline{\alpha}\right).$$
As in Section \ref{sec:cuspforms}, we denote by $S_2(\Gamma_0(N))$ the cusp forms of weight 2 for $\Gamma_0(N)$  and by $(\cdot \,,\cdot)$ the Petersson inner product on $S_2(\Gamma_0(N))$. The pullback $\varphi^*\alpha$ of $\alpha$ under $\varphi$ defines a differential on $X_0(N)$.
It takes the form $\varphi^*\alpha=c\cdot 2\pi i  fdz$ with $c\in \QQ^\times$  and $f\in S_2(\Gamma_0(N))$ a newform of level $N$ with Fourier coefficients $a_n(f)\in \ZZ$ for all $n\in \ZZ_{\geq 1}$. 
After adjusting the sign of $\alpha$, we may and do assume that $c$ is positive. 
The number $c$ is the Manin constant of the optimal quotient $E$.  By definition, it holds  $$(f,f)=\frac{i}{2}\int_{X_0(N)(\CC)}fdz\wedge \overline{fdz}.$$ 
The elliptic curve $E_f$ over $\QQ$, which is associated to $f$ in (\ref{def:ef}), is $\QQ$-isogenous to $E$. Indeed, this follows for example from \cite[Korollar 2]{faltings:finiteness} since by construction  the $L$-functions of $E$ and $f$, of $f$ and $E_f$, and thus of $E$ and $E_f$, have the same Euler product factors for all but finitely many primes. 
Furthermore, $E$ is an optimal quotient of $J_0(N)$ by 1. and $E_f$ is an optimal quotient of $J_0(N)$, since the kernel $I_fJ_0(N)$ (see Section \ref{sec:modulardegree}) of the natural projection $J_0(N)\to E_f$ is connected. 
Therefore it follows that the modular degree $m_f$ of $f$, defined in Section \ref{sec:modulardegree}, satisfies $m_f=\deg(\varphi).$ 
Then, on using that $\varphi^*\alpha=c\cdot 2\pi ifdz$ and on  integrating over $X_0(N)(\mathbb C)$, we see that
the change of variable formula and the above displayed formulas for $h(E)$ and $(f,f)$ lead to 
\begin{equation}\label{eq:fund}
h(E)=\frac{1}{2}\log  m_f-\frac{1}{2}\log (f,f)-\log (2\pi c).
\end{equation}
We now estimate the quantities which appear on the right hand side of this formula.

3.  It follows from \cite[Lemme 3.7]{abul:comparison}, or from \cite[p.262]{silverman:arithgeo}, that $(f,f)\geq e^{-4\pi}/(4\pi).$  
Further, Edixhoven showed in \cite[Proposition 2]{edixhoven:manin} that the Manin constant $c$ of the optimal quotient $E$ of $J_0(N)$ satisfies $c\in\ZZ$ and thus we obtain that $\log (2\pi c)\geq \log (2\pi)$. Then the above lower bound for $(f,f)$,  the formula (\ref{eq:fund}) and the estimate for $m_f$ in Lemma \ref{lem:moddeg} (i) prove Proposition \ref{prop:he} for the optimal quotient $E$ of $J_0(N)$. Finally, on using the reduction in 1. and  (\ref{eq:mazurbound}), we deduce Proposition \ref{prop:he} for all elliptic curves over $\QQ$.
\end{proof}

The main ingredients for the following proof of Proposition \ref{prop:shaf} are the Shimura-Taniyama conjecture and a result of Mazur-Kenku \cite{kenku:ellisogenies} on $\QQ$-isogeny classes of elliptic curves over $\QQ$.

\begin{proof}[Proof of Proposition \ref{prop:shaf}]
Let $E$ be an elliptic curve over $\QQ$, with good reduction over $S$. We write $N_E$ for the conductor of $E$, and we denote by $J_0(N)=\textnormal{Pic}^0(X_0(N))$ the Jacobian of the modular curve $X_0(N)$ for $N\geq 1$ (see Section \ref{sec:modular}). There exists a finite morphism $X_0(\nu_S)\to X_0(N_E)$ of curves over $\QQ$, since $N_E$ divides $\nu_S$ by (\ref{eq:cond23}). 
Picard functoriality gives a surjective morphism $J_0(\nu_S)\to J_0(N_E)$ of abelian varieties over $\QQ$, and as in (\ref{eq:j0nq}) we see that the Shimura-Taniyama conjecture provides that $E$ is a $\QQ$-quotient of $J_0(N_E)$. 
Thus there exists a surjective morphism $$J_0(\nu_S)\to E$$ of abelian varieties over $\QQ$.
Then Poincar\'e's reducibility theorem shows that $E$ is $\QQ$-isogenous to a $\QQ$-simple ``factor'' of $J_0(\nu_S)$. 
Furthermore, the dimension of $J_0(\nu_S)$ coincides with the genus $g$ of the modular curve $X_0(\nu_S)$, and the abelian variety $J_0(\nu_S)$ has at most $g$ $\QQ$-simple  ``factors'' up to $\QQ$-isogenies. 
Therefore we see that there exists a set of elliptic curves over $\QQ$ with the following properties: 
This set has cardinality at most $g$ and for any elliptic curve $E$ over $\QQ$, with good reduction over $S$, there exists an elliptic curve in this set which is $\QQ$-isogenous to $E$. 
Further,  Mazur-Kenku \cite[Theorem 2]{kenku:ellisogenies} give that each $\QQ$-isogeny class of elliptic curves over $\QQ$ contains at most 8 distinct $\QQ$-isomorphism classes of elliptic curves over $\QQ$. 
On combining the results collected above, we deduce that $N(S)\leq 8g$ and then the upper bound for $g$  in (\ref{eq:destimate}) implies  Proposition \ref{prop:shaf}.
\end{proof}

In the following section, we shall combine Proposition \ref{prop:he} or Proposition \ref{prop:shaf} with the Par{\v{s}}in constructions from Section \ref{sec:parshin} to obtain explicit Diophantine finiteness results.
\section{Integral points on moduli schemes}\label{sec:finiteness}
In the first part of this section, we give in Theorem \ref{thm:ms} an effective finiteness 
result for integral points on moduli schemes of elliptic curves. In the second and third part, we refine the method of Theorem \ref{thm:ms} for the moduli schemes  corresponding to $\mathbb P^1-\{0,1,\infty\}$ and to once punctured Mordell elliptic curves. This leads to effective versions of Siegel's theorem for $\mathbb P^1-\{0,1,\infty\}$ and once punctured Mordell elliptic curves, which provide explicit height upper bounds for the solutions of $S$-unit and Mordell equations. We also discuss additional applications. In particular, we consider cubic Thue equations. 

\subsection{Moduli schemes}

To state our result for integral points on moduli schemes of elliptic curves, we use the notation and terminology which was introduced in Section \ref{sec:parshin}. 

Let $T$ and $S$ be non-empty open subschemes of $\sp(\ZZ)$, with $T\subseteq S$. We write $\nu_T=12^3\prod p^2$ with the product taken over all rational primes $p$ not in $T$. For any moduli problem $\mathcal P$ on $(Ell)$, we denote by $\lvert \mathcal P\rvert_T$ the maximal (possibly infinite) number of distinct level $\mathcal P$-structures on an arbitrary elliptic curve over $T$; see (\ref{def:ps}).  We suppose that $Y=M(\mathcal P)$ is a moduli scheme of elliptic curves, which is defined over $S$. Let $h_M$ be  the pullback  of the relative Faltings height by  the canonical forget $\mathcal P$-map, defined in (\ref{def:height}).

\begin{theorem}\label{thm:ms}
The following statements hold. 
\begin{itemize}
\item[(i)] The cardinality of $Y(T)$ is at most $\frac{2}{3}\lvert \mathcal P \rvert_T\nu_T\prod (1+1/p)$ with the product taken over all rational primes $p$ which divide $\nu_T$.
\item[(ii)] If $P\in Y(T)$, then $h_M(P)\leq \frac{1}{4}\nu_T(\log \nu_T)^2+9$.
\end{itemize}
\end{theorem}

We refer to  Section \ref{sec:modintro} for a  discussion of this theorem. In addition, we now mention that for many classical moduli problems $\mathcal P$ on $(Ell)$ it is possible to express $\lvert \mathcal P\rvert_T$ in terms of more conventional data, where $T$ is an arbitrary scheme which is connected.  
For example, if $\mathcal P_N$ is the ``naive" level $N$ moduli problem on $(Ell)$ considered in Section \ref{sec:parshin}, then (\ref{eq:compgn}) shows that $\lvert\mathcal P_N\rvert_T$ is an explicit function in terms of the level $N\geq 1$.

It is quite difficult, when not impossible, to compare Theorem \ref{thm:ms} with quantitative or effective  finiteness results in the literature, since these results hold in different settings. One can mention for example the quantitative result of Corvaja-Zannier \cite{coza:siegel} for hyperbolic curves which relies on Schmidt's subspace theorem, or the effective result of Bilu \cite{bilu:modular} for certain modular curves which is based on the theory of logarithmic forms. 

\begin{proof}[Proof of Theorem \ref{thm:ms}]
To prove (i) we denote by $M(T)$ be the set of isomorphism classes of elliptic curves over $T$. Let $M(T)_\QQ$ be the set of $\QQ$-isomorphism classes of elliptic curves over $\QQ$, with good reduction over $T$. We now show that there exists a bijection $$M(T)\cong M(T)_\QQ.$$ 
Any elliptic curve over $\QQ$ has good reduction over $T$ if and only if it is the generic fiber of an elliptic curve over $T$. 
Further, any elliptic curve over $T$ is the N\'eron model of its generic fiber (see for example \cite[p.15]{bolura:neronmodels}), and $M(T)$ is in bijection with the set of isomorphism classes of $T$-schemes generated by elliptic curves over $T$. Hence the N\'eron mapping property proves that $M(T)\cong M(T)_\QQ$ 
and thus Proposition \ref{prop:shaf} implies $$\lvert M(T)\rvert\leq \frac{2}{3}\nu_T\prod(1+1/p)$$ with the product taken over all rational primes $p$ which divide $\nu_T$. It follows from Lemma \ref{lem:moduli} that $\lvert Y(T)\rvert\leq \rvert \mathcal P\rvert_T\lvert M(T)\rvert$ and then we deduce Theorem \ref{thm:ms} (i).

To show (ii) we take $P\in Y(T)$ and we write $[E]=\phi(P)$ for $\phi:Y(T)\to M(T)$ the canonical forget $\mathcal P$-map from Lemma \ref{lem:moduli}. The conductor $N_E$ of the generic fiber $E_\QQ$ of $E$ takes the form $N_E=\prod p^{f_p}$ with $f_p$ the conductor exponent of $E_\QQ$ at a rational prime $p$, see Section \ref{sec:cond}. It holds that $f_p\leq 2$ for $p\geq 5$ and (\ref{eq:cond23}) gives that $f_2\leq 8$ and $f_3\leq 5$. Furthermore, if $p\in T$ then we get that $f_p=0$ since $E_\QQ$ extends to an abelian scheme over $T$. On combining the above results, we deduce that $N_E\mid \nu_T$. An application of Proposition \ref{prop:he} with $E_\QQ$ gives that $h_M(P)\leq \frac{1}{4}N_E(\log N_E)^2+9$ which together with $N_E\leq \nu_T$ implies assertion (ii). This completes the proof of Theorem \ref{thm:ms}.
\end{proof}
On replacing in the proof of Theorem \ref{thm:ms} (i) the explicit estimate from Proposition \ref{prop:shaf} by the asymptotic bound (\ref{eq:ehv}) of Ellenberg-Helfgott-Venkatesh, we obtain the following  version of Theorem \ref{thm:ms} (i): If $Y=M(\mathcal P)$ is a moduli scheme, defined over $S$, then 
$$\lvert Y(T)\rvert \ll \lvert \mathcal P\rvert_TN_T^{0.1689}$$ 
for $N_T$ the product of all rational primes $p$ not in $T$. Furthermore, the discussion surrounding (\ref{eq:ehv}) shows that  the ``Birch and Swinnerton-Dyer conjecture" together with the ``Generalized Riemann Hypothesis" implies that for all $\epsilon>0$ there exists a constant $c(\epsilon)$, depending only on $\epsilon$, such that $\lvert Y(T)\rvert\leq c(\epsilon) \lvert \mathcal P\rvert_TN_T^{\epsilon}.$

 We notice that the complement of $S$ in $\sp(\ZZ)$ is a finite set of rational prime numbers. For the remaining of Section \ref{sec:finiteness}, we will adapt our notation to the classical number theoretic setting and in (\ref{sec:p1}-\ref{sec:thue}) the symbol $S$ will denote a finite set of rational prime numbers. 

\subsection{$\mathbb P^1-\{0,1,\infty\}$: $S$-unit equations}\label{sec:p1}

In the first part of this section, we briefly review alternative methods which give finiteness for integral points of $\mathbb P^1-\{0,1,\infty\}$, or equivalently for the number of solutions of $S$-unit equations. In the second and third part, we establish in Corollary \ref{cor:p1} an explicit upper bound in terms of $S$ for the height of the solutions of $S$-unit equations, and we compare this bound with the actual best results in the literature. In the last part, we discuss  upper bounds for the number of solutions of $S$-unit equations.

Let $S$ be a finite set of rational primes, let $N_S=\prod p$ with the product taken over all $p\in S$,
and let $\mathcal O^\times$ be the units of $\mathcal O=\ZZ[1/N_S]$. We recall the $S$-unit equation (\ref{eq:unit})
\begin{equation*}
x+y=1, \ \ (x,y)\in \mathcal O^\times\times\mathcal O^\times.
\end{equation*}
Before we apply the method of this paper to $S$-unit equations (\ref{eq:unit}), we briefly review in the following subsection alternative methods which give finiteness of (\ref{eq:unit}).
\subsubsection{Alternative methods}\label{subsec:p1altmethods}

The first finiteness proof for $S$-unit equations (\ref{eq:unit}) goes back to Mahler \cite{mahler:approx1}. He used the method of Diophantine approximations (Thue-Siegel). Another proof of Mahler's theorem was obtained by
Faltings,  whose general finiteness theorems in \cite{faltings:finiteness} cover in particular (\ref{eq:unit}). Faltings studied semi-simple $\ell$-adic Galois representations associated to abelian varieties.
Recently, Kim \cite{kim:siegel} gave a new finiteness proof of (\ref{eq:unit}). He used Galois representations associated to the unipotent \'etale and de Rham fundamental group of $\mathbb P^1-\{0,1,\infty\}$. 
The methods of Faltings, Kim and Thue-Siegel (Mahler) are a priori not effective. 
The first effective finiteness proof of (\ref{eq:unit}) was given\footnote{For instance, Coates explicit result \cite{coates:shafarevich}, which was published in 1970, implies an effective height upper bound for the solutions of (\ref{eq:unit}).} by Baker's method, using the theory of logarithmic forms; see for example Baker-W\"ustholz \cite{bawu:logarithmicforms}.
Another effective finiteness proof of (\ref{eq:unit}) is due to Bombieri-Cohen \cite{boco:effdioapp2}. They generalized Bombieri's method in \cite{bombieri:effdioapp}, which uses effective Diophantine approximations on the multiplicative group $\mathbb G_m$ (Thue-Siegel principle).
The methods of Baker and Bombieri both give  explicit upper bounds for the heights of the solutions of (\ref{eq:unit}) in terms of $S$, and they both allow to deal with $S$-unit equations in any number field.
So far, the theory of logarithmic forms, which was extensively polished and sharpened over the last 47 years, produces slightly better bounds than Bombieri's method.
On the other hand, Bombieri's method is relatively new and is essentially self-contained; see Bombieri-Cohen \cite{boco:effdioapp3}. 

\subsubsection{Effective resolution}\label{subsec:p1eff}

To state and discuss our effective result for $S$-unit equations we put  $n_S=2^7N_S$. Let $h(\beta)$ be the usual absolute logarithmic Weil height of any $\beta\in \QQ$. This height is defined for example in \cite[p.16]{bogu:diophantinegeometry}. We obtain the following corollary. 

\begin{corollary}\label{cor:p1}
Any solution $(x,y)$ of the $S$-unit equation $(\ref{eq:unit})$ satisfies $$h(x),h(y)\leq \frac{3}{2}n_S(\log n_S)^2+65.$$
\end{corollary}
\begin{proof}
We use the notation and terminology of Section \ref{sec:parshin}. The discussion in (\ref{ass:2}) shows that we may and do assume that $2$ is invertible on $T=\sp(\mathcal O)$. Write $$X=\mathbb P^1_T-\{0,1,\infty\}=\sp(\mathcal O[z,1/(z(z-1))])$$ for $z$ an ``indeterminate". We suppose that $(x,y)$ satisfies   
(\ref{eq:unit}). Then we see that there exists $P\in X(T)$ with $z(P)=x$. Thus an application of Proposition \ref{prop:p1par} with $P$ and $T$  gives an elliptic curve $E'$ over $T$ that satisfies (write $E=E'$) $$h(x)\leq 6h(E)+11 \ \textnormal{ and } \ N_{E}\leq n_S.$$ Here $N_E$ is the conductor of $E$  and $h(E)$ is the relative Faltings height of $E$, see  Section \ref{sec:explicitconstr} for the definitions. Proposition \ref{prop:he} provides that $h(E)\leq \frac{1}{4}N_{E}(\log N_E)^2+9$. Then the displayed inequalities imply the claimed upper bound for $h(x)$, and then for $h(y)$ by symmetry. This completes the proof of Corollary \ref{cor:p1}.
\end{proof}
As already mentioned in the introduction, this corollary is an effective version of Frey's remark in \cite[p.544]{frey:ternary}. (We presented Corollary \ref{cor:p1} and its proof in various seminars and conferences in Princeton (Sept. 2011 and Jan. 2012), New York (Feb. 2012), Michigan (March 2012), Hong Kong (June 2012), Paris (Oct. 2012) and Zurich (May 2013). After we uploaded the present paper to the arXiv in October 2013, Hector Pasten informed us about his joint work with Ram Murty (\cite{mupa:modular}, submitted Nov. 2012) which was published online in July 2013 and which was presented including the proof in a seminar in Kingston (March 2012) and in a workshop in Toronto (Nov. 2012); we thank Hector Pasten for informing us about \cite{mupa:modular}. The main results of \cite{mupa:modular} independently establish versions of Corollary \ref{cor:p1} (and of Lemma \ref{lem:moddeg}, Proposition \ref{prop:he} and Corollary \ref{cor:de} which are used in the proof of Corollary \ref{cor:p1}) with effective bounds of the form $\ll N\log N$, while our corresponding bounds are of the (slightly) weaker form $\ll N(\log N)^2$. The method used in \cite{mupa:modular} is similar to our proof of Corollary \ref{cor:p1}. To conclude the discussion we point out that the results were obtained completely independently: We obtained the results of this paper without knowing anything of the related work of Hector Pasten and Ram Murty, and they obtained the results of \cite{mupa:modular} without knowing anything of our related work.)

Corollary \ref{cor:p1} allows in principle to find all solutions of any $S$-unit equation (\ref{eq:unit}). 
To discuss a practical aspect of Corollary \ref{cor:p1}, we observe that  any $u\in \mathcal O^\times$ satisfies $u=\prod p^{u_p}$  with the product taken over all $p\in S$ and $u_p=\textnormal{ord}_p(u)$. Therefore any $S$-unit equation may be viewed as an exponential Diophantine equation of the form $$\prod_{p\in S} p^{x_p}+\prod_{p\in S} p^{y_p}=1, \ \  ((x_p),(y_p))\in \ZZ^s\times \ZZ^s$$ 
for $s=\lvert S\rvert$. If $((x_p),(y_p))$ satisfies this exponential Diophantine equation, then 
Corollary \ref{cor:p1} implies that $\max_{p\in S} \lvert x_p\rvert$ and  $\max_{p\in S} \lvert y_p\rvert$ are at most $\frac{3}{2\log 2}n_S(\log n_S)^2+94$. On using additional tricks, we will improve in \cite{vkma:computation} the absolute constants $\frac{3}{2\log 2}$ and $94$  and we will transform the proof of Theorem \ref{cor:p1} into a practical algorithm to solve $S$-unit equations. 

\subsubsection{Comparison to known results}

Next, we compare Corollary \ref{cor:p1} with the actual best effective results in the literature for (\ref{eq:unit}). We notice that (\ref{eq:unit}) has no solutions when $\lvert S\rvert=0$, and $(\frac{1}{2},\frac{1}{2})$, $(2,-1)$ and $(-1,2)$ are the only solutions of (\ref{eq:unit}) when $\lvert S\rvert=1$. 
Further, we see that if (\ref{eq:unit}) has a solution, then $2\in S$. Thus, for the purpose of the comparison, we may and do assume 
\begin{equation}\label{ass:2}
s=\lvert S\rvert\geq 2 \ \ \textnormal{ and } \ \ 2\in S. 
\end{equation}
Let $(x,y)$ be a solution of the $S$-unit equation (\ref{eq:unit}). The actual best explicit height upper bound for $(x,y)$ in the literature is due to Gy\H{o}ry-Yu \cite{gyyu:sunits}. They used the  state of the art in the theory of logarithmic forms. In the case of (\ref{eq:unit}), where the number field is $\QQ$, their estimate in \cite[Theorem 2]{gyyu:sunits} becomes
$$h(x),h(y)\leq 2^{10s+22}s^4q\prod\log p$$
with the product taken over all rational primes $p\in S-\{q\}$ for $q=\max S$. The right hand side of the displayed inequality is always bigger than $2^{47}$.
Hence, we see that Corollary \ref{cor:p1} improves \cite{gyyu:sunits} for sets $S$ with small $N_S$, in particular for all sets $S$ with $N_S\leq 2^{30}.$
This improvement is significant for the practical solution of $S$-unit equations, see the  discussion at the end of Section \ref{subsec:p1eff}. However, 
the result of Stewart-Yu \cite[Theorem 1]{styu:abc2}, based on the actual state of the art in the theory of logarithmic forms, gives  $$h(x),h(y)\ll N_S^{1/3}(\log N_S)^3.$$ 
We observe that this inequality of Stewart-Yu is better than Corollary \ref{cor:p1} for all sets $S$ with sufficiently large $N_S$. This concludes our comparison.

\subsubsection{Number of solutions}

To discuss explicit upper bounds for the number of solutions of $S$-unit equations (\ref{eq:unit}), we recall that $n_S=2^7N_S$. In the special case of the moduli scheme $\mathbb P^1_{\ZZ[1/2]}-\{0,1,\infty\}$, one can refine the proof of Theorem \ref{thm:ms} and one obtains the following result. 

\begin{corollary}\label{cor:p1quant}
The $S$-unit equation $(\ref{eq:unit})$ has at most $4n_S\prod_{p\in S} (1+1/p)$ solutions.
\end{corollary}
\begin{proof}
We  use the terminology and notation introduced in Section \ref{sec:parshin}. The discussion in (\ref{ass:2}) shows that we may and do assume that $2$ is invertible on $T=\sp(\mathcal O)$. Then there exists a bijection between  the set of solutions of the $S$-unit equation (\ref{eq:unit}) and $Y(T)$, where $$Y=\mathbb P^1_{\ZZ[1/2]}-\{0,1,\infty\}.$$   
We now estimate the cardinality of $Y(T)$. The remark at the end of Section \ref{subsec:s-unit} shows that $Y=M(\mathcal P)$ is a moduli scheme of elliptic curves, defined over $\sp(\ZZ[1/2])$, where $\mathcal P=[Legendre]$ is the Legendre moduli problem on $(Ell)$. 
Thus Lemma \ref{lem:moduli} gives a map $$\phi:Y(T)\to M(T),$$ with all fibers having cardinality at most $\lvert \mathcal P\rvert_T$. Here $\lvert \mathcal P\rvert_T$ is defined in (\ref{def:ps}) and $M(T)$ is the set of isomorphism classes of elliptic curves over $T$.
The arguments of Proposition \ref{prop:p1par} (iii) and of Theorem \ref{thm:ms} imply that the cardinality of $\phi(Y(T))$ is at most the number of $\QQ$-isomorphism classes of elliptic curves over $\QQ$, with conductor dividing $n_S=2^7N_S$. 
Therefore, on replacing $\nu_S$ by $n_S$ in the proof of Proposition \ref{prop:shaf}, we deduce $$\lvert \phi(Y(T))\rvert\leq \frac{2}{3}n_S\prod_{p\in S} (1+1/p).$$ 
Here we used that $2\in S$.
It follows that $ Y(T)$ has at most $4n_S\prod_{p\in S} (1+1/p)$ elements, since the fibers of $\phi$ have cardinality at most $\lvert \mathcal P\rvert_T\leq 6$. Then we conclude Corollary \ref{cor:p1quant}.
\end{proof}

We now compare Corollary \ref{cor:p1quant} to results in the literature. Evertse \cite[Theorem 1]{evertse:sunits} used  the method of Diophantine approximations to prove that any $S$-unit equation (\ref{eq:unit}) has at most $3\cdot 7^{3+2\lvert S\rvert}$ solutions. We mention that Evertse's result holds for more general unit equations in any number field, and it provides, as far as we know, the actual best upper bound in the literature for the number of solutions of (\ref{eq:unit}). 
Further, we see that Evertse's result is considerably better than Corollary \ref{cor:p1quant} for almost all sets $S$, since $3\cdot 7^{3+2\lvert S\rvert}\ll_\epsilon n_S^\epsilon.$ Notice there are sets $S$ for which Corollary \ref{cor:p1quant} improves \cite[Theorem 1]{evertse:sunits}. For example, if $S\subseteq \{2,3,5,\dotsc, 83, 89\}$ and if $S$ satisfies the reasonable assumption (\ref{ass:2}), then we observe that Corollary \ref{cor:p1quant} is better than \cite[Theorem 1]{evertse:sunits}.

\subsection{Once punctured Mordell elliptic curves: Mordell equations}\label{sec:mordell}

In the first part of this section, we briefly review alternative methods which give finiteness for integral points on once punctured Mordell elliptic curves, or equivalently for the number of $S$-integer solutions of Mordell equations. In the second and third part, we state and prove Corollary \ref{cor:m} on Mordell equations and we compare it with the actual best effective results in the literature. In the fourth and fifth part, we refine a result of Stark and we discuss explicit upper bounds for the number of solutions of Mordell equations.

We continue to denote by $S$ an arbitrary finite set of rational prime numbers and we write $\mathcal O=\ZZ[1/N_S]$ for $N_S$ the product of all $p\in S$. For any nonzero $a\in\mathcal O$, we recall that Mordell's equation (\ref{eq:mordell}) is of the form
\begin{equation*}
y^2=x^3+a, \ \ (x,y)\in \mathcal O\times\mathcal O.
\end{equation*}
This Diophantine equation is a priori more difficult than $S$-unit equations (\ref{eq:unit}).  
Indeed, elementary transformations  reduce (\ref{eq:unit}) to (\ref{eq:mordell}), while the  known (unconditional) reductions of (\ref{eq:mordell}) to controlled $S$-unit equations require to solve (\ref{eq:unit}) over field extensions.

\subsubsection{Alternative methods}\label{subsec:maltmethods}

As already mentioned in the introduction, the resolution of Mordell's equation in $\ZZ\times \ZZ$ is equivalent to the classical problem of finding all perfect squares and perfect cubes with given difference. We refer to Baker's introduction of \cite{baker:mordellequation} for a discussion (of partial resolutions) of this classical problem, which goes back at least to Bachet 1621. Mordell \cite{mordell:1922,mordell:1923} showed that (\ref{eq:mordell}) has only finitely many solutions in $\ZZ\times \ZZ$. 
He reduced the problem to Thue equations and then he applied Thue's finiteness theorem which is based on Diophantine approximations. More generally, the completely different methods of Siegel, Faltings \cite{faltings:finiteness}  and Kim \cite{kim:siegel,kim:cm} give finiteness of (\ref{eq:mordell}). 
Siegel's method uses Diophantine approximations, and the methods of Faltings and Kim are briefly described in Section \ref{subsec:p1altmethods}. 
We mention that these methods, which in fact allow to deal with considerably more general Diophantine problems, are all a priori not effective. See also Bombieri \cite{bombieri:effdioapp} and Kim's discussions in \cite{kim:effective}.  
The first effective finiteness result  for solutions in $\ZZ\times\ZZ$ of Mordell's equation (\ref{eq:mordell}) was provided by Baker \cite{baker:mordellequation}. Baker's result is based on the theory of logarithmic forms. 

\subsubsection{Effective resolution}\label{subsec:effmordell}

We now state and prove our effective result for Mordell equations. We continue to denote by $h(\beta)$ the absolute logarithmic Weil height of any $\beta\in \QQ$. To measure the set $S$ and the nonzero number $a\in\mathcal O$, we use inter alia the quantity
\begin{equation*}
a_S=2^83^5N_S^2r_2(a), \ \ \ r_2(a)=\prod p^{\min(2,\textnormal{ord}_p(a))}
\end{equation*}
with the product taken over all rational primes $p\notin S$ with $\textnormal{ord}_p(a)\geq 1$. The following corollary allows in principle to find all solutions of any Mordell equation (\ref{eq:mordell}). 

\begin{corollary}\label{cor:m}
If $(x,y)$ satisfies Mordell's equation $(\ref{eq:mordell})$, then 
$$h(x),h(y)\leq h(a)+4a_S(\log a_S)^2.$$
\end{corollary}
\begin{proof}
The proof is completely analogous to the proof of Corollary \ref{cor:p1}. We use the notation and terminology of Section \ref{sec:parshin}. Write $T=\sp(\mathcal O[1/(6a)])$ and define $$Z=\sp\bigl(\mathcal O[x_0,y_0]/(y_0^2-x_0^3-a)]\bigl)$$ for $x_0$ and $y_0$  ``indeterminates". We suppose that $(x,y)$ is a solution of   
(\ref{eq:mordell}). Then there exists a $T$-integral point $P\in Z(T)$ with $x_0(P)=x$ and thus an application of Proposition \ref{prop:mpar} with $K=\QQ$, $P$ and $T$  gives an elliptic curve $E$ over $T$ that satisfies $$h(x)\leq \frac{1}{3}h(a)+8h(E)+2\log\bigl(\max(1,h_F(E))\bigl)+36 \ \textnormal{ and } \ N_{E}\leq a_S.$$ Here $N_E$ denotes the conductor of $E$, and $h(E)$ and $h_F(E)$ denote the relative and the stable Faltings height of $E$ respectively; see  Section \ref{sec:explicitconstr} for the definitions. Proposition \ref{prop:he} provides that $h(E)\leq\frac{1}{4}N_{E}(\log N_E)^2+9$ and it holds that $h_F(E)\leq h(E)$. Therefore the displayed inequalities lead to the claimed estimate for $h(x)$, and then for $h(y)$ since $y^2=x^3+a$. This completes the proof of Corollary \ref{cor:m}. 
\end{proof}
We already pointed out in the introduction that Corollary \ref{cor:m} provides in particular an entirely new proof of Baker's classical result \cite[Theorem 1]{baker:mordellequation}. 
\subsubsection{Comparison to known results}\label{sec:effcompm}

In what follows, we compare our Corollary \ref{cor:m} with the actual best effective height upper bounds in the literature for the solutions of Mordell's equation (\ref{eq:mordell}). For this purpose, we notice that if $a\in \ZZ-\{0\}$ and if $\textnormal{rad}(a)=\prod_{p\mid a} p$ denotes the radical of $a$, then
\begin{equation}\label{eq:as}
r_2(a)\leq\lvert a\rvert \ \ \textnormal{ and } \ \ r_2(a)\mid \textnormal{rad}(a)^2.
\end{equation} 
Over the last 45 years, many authors improved the explicit bound provided by Baker \cite{baker:mordellequation}, using refinements of the theory of logarithmic forms; see  Baker-W\"ustholz \cite{bawu:logarithmicforms} for an overview. The actual best explicit upper bound is due to Hajdu-Herendi \cite{hahe:elliptic}, and  due to Juricevic \cite{juricevic:mordell} in the important special case $\mathcal O=\ZZ$. 

We first discuss the classical case $\mathcal O=\ZZ$. If  $\mathcal S=\{(10^{181},4), (10^{23},5), (10^{19},6)\}$ and if $a\in \ZZ-\{0\}$, then  Juricevic \cite{juricevic:mordell} gives that any solution $(x,y)\in \ZZ\times \ZZ$ of (\ref{eq:mordell}) satisfies $$h(x),h(y)\leq \min_{(m,n)\in \mathcal S} m\lvert a\rvert(\log \lvert a\rvert)^n.$$
On using (\ref{eq:as}), we see that Corollary \ref{cor:m} improves this inequality and therefore our corollary establishes the actual best result for (\ref{eq:mordell}) in the classical case $\mathcal O=\ZZ$. 

It remains to discuss the case of arbitrary $\mathcal O$. To state the rather complicated bound in \cite{hahe:elliptic}, we have to introduce some notation. As in \cite{hahe:elliptic}, we define
$$c_1=\frac{32}{3}\Delta^{\frac{1}{2}}(8+\frac{1}{2}\log \Delta)^4, \ \ \ c_2=10^4\cdot 256\cdot \Delta^{\frac{2}{3}}, \ \ \ \Delta=27\lvert a\rvert^2.$$ Write $c_S=7\cdot 10^{38s+86}(s+1)^{20s+35}q^{24}\max(1,\log q)^{4s+2}$ for $s=\lvert S\rvert$ and $q=\max S$.
If $s\geq 1$ and if $a\in \ZZ-\{0\}$, then   \cite[Theorem 2]{hahe:elliptic} of Hajdu-Herendi, which in fact holds more generally for any elliptic equation, gives that any solution $(x,y)$ of (\ref{eq:mordell}) satisfies
$$h(x),h(y)\leq c_Sc_1(\log c_1)^2(c_1+20sc_1+\log (ec_2)).$$
It follows from (\ref{eq:as}) that the dependence on $a\in \ZZ$ of Corollary \ref{cor:m} is of the form $\lvert  a\rvert (\log \lvert a\rvert)^2$, while \cite[Theorem 2]{hahe:elliptic} is of the weaker form $\lvert a \rvert^2(\log \lvert a\rvert)^{10}$. Further, on using again (\ref{eq:as}), we see that Corollary \ref{cor:m}  improves \cite[Theorem 2]{hahe:elliptic} for ``small'' sets $S$, in particular for all sets $S$ with $N_S\leq 2^{1200}$ or with $s\leq 12$. 
This improvement is significant for the practical resolution of Mordell equations (\ref{eq:mordell}), see for example \cite{vkma:computation}.
However, if $N_S\gg\lvert a\rvert$, then one can not say which bound is better. 
The point is that there are sets $S$ with $N_S\gg\lvert a\rvert$ for which our result is better than \cite[Theorem 2]{hahe:elliptic}, and vice versa. Finally, we  mention that (so far) all effective results for (\ref{eq:mordell}) in the literature are based on the theory of logarithmic forms, and this theory allows to deal with more general Diophantine equations over arbitrary number fields; see \cite{bawu:logarithmicforms}. This concludes our comparison.

\subsubsection{A refinement of Stark's theorem}

We now discuss a refinement of the following theorem of Stark \cite[Theorem 1]{stark:mordell}: If $a\in \ZZ-\{0\}$ then  any  $(x,y)\in \ZZ\times\ZZ$ with $y^2=x^3+a$ satisfies $$h(x),h(y)\ll_\epsilon \lvert a\rvert^{1+\epsilon},$$ 
where the implied constant is effective.
This classical estimate of Stark is based on the theory of logarithmic forms.
The following result is a direct consequence of Corollary \ref{cor:m}.

\begin{corollary}\label{cor:stark}
If $\epsilon>0$ is a real number, then there exists an effective constant $c$, depending only on $\epsilon$, such that any solution $(x,y)$ of Mordell's equation $(\ref{eq:mordell})$ satisfies $$h(x),h(y)\leq h(a)+c\cdot a_S^{1+\epsilon}.$$
\end{corollary}
On using (\ref{eq:as}) and the fact that $h(a)=\log \lvert a\rvert$ for $a\in \ZZ-\{0\}$, we see that Corollary \ref{cor:stark}  generalizes and refines Stark's theorem \cite[Theorem 1]{stark:mordell} discussed above.

We remark that Stewart-Yu \cite{styu:abc2} obtained an exponential version of the $abc$-conjecture $(abc)$, and $(abc)$ is equivalent to a certain upper bound for the height of the solutions of Mordell's equation (\ref{eq:mordell}); see for example  \cite[p.428]{bogu:diophantinegeometry}. However, by elementary reasons, all known links between  $(abc)$  and height upper bounds for the solutions of (\ref{eq:mordell}) do not work any more with exponential versions. Hence, at the time of writing, it is  not possible to improve Corollary \ref{cor:stark} on using exponential versions of $(abc)$. 

\subsubsection{Number of solutions}

Next, we  discuss explicit upper bounds for the number of solutions of Mordell's equation (\ref{eq:mordell}). In the special case of the moduli scheme  corresponding to  $(\ref{eq:mordell})$, one can refine the proof of Theorem \ref{thm:ms} and one obtains the following result.

\begin{corollary}\label{cor:mquant}
The number of solutions of $(\ref{eq:mordell})$ is at most $\frac{2}{3}a_S\prod_{p\mid a_S} (1+1/p)$ with the product taken over all rational primes $p$ which divide $a_S$.
\end{corollary}
\begin{proof}
In this proof, we use the terminology and notation which we introduced in Section \ref{sec:parshin}. We define $T=\sp(\mathcal O[1/(6a)])$ and $b=-a/1728$. It follows that the number of solutions of Mordell's equation (\ref{eq:mordell}) is at most the cardinality of $Y(T)$, where
$$Y=\sp\bigl(\ZZ[1/6,c_4,c_6,b,1/b]/(1728b-c_4^3+c_6^2)\bigl)$$ for $c_4$ and $c_6$ ``indeterminates''. We now estimate the cardinality of $Y(T)$. The remark at the end of Section \ref{subsec:mordell} gives that $Y=M(\mathcal P)$ is a moduli scheme of elliptic curves, where $\mathcal P=[\Delta=b]$ is the corresponding moduli problem on $(Ell)$. 
Thus Lemma \ref{lem:moduli} gives a map $$\phi:Y(T)\to M(T)$$ for $M(T)$ the set of isomorphism classes of elliptic curves over $T$. We notice that the map $\phi$ coincides with the map $\phi$ constructed in Proposition \ref{prop:mpar}. Further, we see that the arguments of Proposition \ref{prop:mpar} (iii) and of Theorem \ref{thm:ms} imply that $\lvert \phi(Y(T))\rvert$ is at most the number of $\QQ$-isomorphism classes of elliptic curves over $\QQ$, with conductor dividing $a_S$. 
Therefore, on replacing in the proof of Proposition \ref{prop:shaf} the number $\nu_S$ by $a_S$, we deduce $$\lvert \phi(Y(T))\rvert\leq \frac{2}{3}a_S\prod (1+1/p)$$
with the product taken over all rational primes $p$ dividing $a_S$.
Proposition \ref{prop:mpar} (i) shows that $\phi$ is injective and then the displayed inequality implies Corollary \ref{cor:mquant}.
\end{proof}

We now compare Corollary \ref{cor:mquant} with results in the literature. 
In the classical case $\mathcal O=\ZZ$, the actual best explicit upper bound for the number of solutions  of (\ref{eq:mordell})  is due to Poulakis \cite{poulakis:corrigendum}. 
We see that Corollary \ref{cor:mquant} is better than Poulakis' result when $a_S\leq 2^{180}$, and is worse when $a_S$ is sufficiently large.
However, for large $a_S$ the actual best bound follows from Ellenberg-Helfgott-Venkatesh \cite{heve:integralpoints,elve:classgroup}. On combining their results (see  \cite{rvk:height} for details), one obtains that the number of solutions of (\ref{eq:mordell}) is $$\ll c_0^s(1+\log q)^2\textnormal{rad}(a)^{0.1689}$$
for $c_0$ an absolute constant, $s=\lvert S\rvert$ and $q=\max S$. 
This asymptotic bound is better than the asymptotic estimate implied by Corollary \ref{cor:mquant}. We point out that the methods of Poulakis \cite{poulakis:corrigendum} and Helfgott-Venkatesh \cite{heve:integralpoints} are fundamentally different from the method of Corollary \ref{cor:mquant}; see the end of Section \ref{subsec:he} for a brief discussion of the Diophantine results used in the proofs of \cite{poulakis:corrigendum} and \cite{heve:integralpoints}. To conclude our comparison, we mention that Evertse-Silverman \cite{evsi:uniformbounds} applied Diophantine approximations to obtain an explicit upper bound for the number of solutions of (\ref{eq:mordell}). Their bound involves inter alia a quantity which depends on a certain class number.

\subsection{Additional Diophantine applications}\label{sec:thue}

In this section we discuss additional Diophantine applications of the Shimura-Taniyama conjecture. In particular, we consider cubic Thue equations.  

There are many Diophantine equations which can be reduced to $S$-unit or Mordell equations, such as for example (super-) elliptic Diophantine equations. Usually these reductions consist  of elementary, but ingenious, manipulations of explicit equations and  they often require to solve $S$-unit and Mordell equations over controlled field extensions $K$ of $\QQ$. Unfortunately we can not use most of the standard reductions, since our results in the previous sections only hold for $K=\QQ$.  However, we now discuss constructions which allow to reduce certain classical Diophantine problems without requiring field extensions. 

\subsubsection{Thue equations}

Let $m$ be an integer and let $f\in \ZZ[x,y]$ be an irreducible binary form of degree $n\geq 3$, with discriminant $\Delta$. We consider the classical Thue equation
\begin{equation}\label{eq:thue}
f(u,v)=m, \ \ \ (u,v)\in \ZZ\times \ZZ.
\end{equation} 
The famous result of Thue, based on Diophantine approximations, gives that (\ref{eq:thue}) has only finitely many solutions. Moreover, Baker \cite{baker:contributions} used his theory of logarithmic forms to prove an effective finiteness result for Thue equations; see \cite{bawu:logarithmicforms} for generalizations.

We now suppose that $n=3$. To prove (effective) finiteness results for (\ref{eq:thue}), we may and do assume by standard reductions that (\ref{eq:thue}) has at least one solution and that $m\Delta\neq 0$. 
Thus we get a smooth, projective and geometrically connected genus one curve 
$$X=\textnormal{Proj}\bigl(\QQ[x,y,z]/(f-mz^3)\bigl).$$
On using classical invariant theory of cubic binary forms, one obtains a finite $\QQ$-morphism $\varphi:X\to \textnormal{Pic}^0(X)$ of degree 3 and one computes 
\begin{equation*}
\textnormal{Pic}^0(X)=\textnormal{Proj}\bigl(\QQ[x,y,z]/(y^2z-x^3-az^3)\bigl), \ \ \ a=432m^2\Delta\neq 0.
\end{equation*} 
See for example Silverman \cite[p.401]{silverman:ellthue} for details.
Moreover, if $(u,v)$ satisfies (\ref{eq:thue}) and if $P$ denotes the corresponding $\QQ$-point of $X$, then the definition of $\varphi$ shows that $x(\varphi(P))$ and $y(\varphi(P))$ are both in $\ZZ$ and $z(\varphi(P))=1$. In other words, the finite $\QQ$-morphism $$\varphi:X\to \textnormal{Pic}^0(X)$$ of degree 3 reduces any cubic Thue equation (\ref{eq:thue}) to a Mordell equation (\ref{eq:mordell}) of the form $(v')^2=(u')^3+a, \ (u',v')\in\ZZ\times\ZZ$. Therefore we see that Corollary \ref{cor:mquant} gives a quantitative finiteness result for any cubic Thue equation (\ref{eq:thue}). In fact the above arguments prove more generally that any cubic Thue equation (\ref{eq:thue}) has only finitely many solutions in $\mathcal O\times \mathcal O$ for $\mathcal O=\ZZ[1/N_S]$ as in the previous sections. 
Furthermore, it seems possible to deduce from Corollary \ref{cor:m} explicit height upper bounds for the solutions of (\ref{eq:thue}) in $\mathcal O\times\mathcal O$, since all involved reductions can be made explicit.

At the time of writing, it is not clear to the author how to generalize the method in order to deal with the cases $n\geq 4$. Such generalizations would be interesting for various reasons. For example, any elliptic Diophantine equation over $\ZZ$ can be reduced to certain controlled Thue equations (\ref{eq:thue}) of degree $n=4$. This well-known reduction is ingenious, but completely elementary. It only requires the classical reduction theory of quartic binary forms over $\ZZ$ which goes back (at least) to Hermite 1848.

\section{Abelian varieties of product $\gl2$-type}\label{sec:gl2}

In the first part of this section, we define and discuss abelian varieties of product $\gl2$-type. Then we state Theorem \ref{thm:gl2} and Proposition \ref{prop:h} which provide explicit inequalities relating the stable Faltings height and the conductor of abelian varieties over $\QQ$ of product $\gl2$-type. In the second part, we prove Theorem \ref{thm:gl2} and Proposition \ref{prop:h}.

Let $S$ be a connected Dedekind scheme, with field of fractions $K$ a number field.  
Let $A$ be an abelian scheme over $S$ of relative dimension $g\geq 1$. We say that $A$ is of $\gl2$-type if there exists a number field $F$ of degree $[F:\QQ]=g$ together with an embedding $$F\hookrightarrow \nd(A)\otimes_\ZZ \QQ.$$ 
The terminology $\gl2$-type comes from the following property: If $A$ is of $\gl2$-type and if $V_\ell(A)=T_\ell(A)\otimes_{\ZZ_\ell}\QQ_\ell$ denotes the rational $\ell$-adic Tate module associated to the generic fiber of $A$, 
then $V_\ell(A)$ is a free  $F\otimes_\QQ \QQ_\ell$-module of rank 2 and thus 
the action on  $V_\ell(A)$ of the absolute Galois group of $K$  defines a representation with values in  $$\textnormal{GL}_2(F\otimes_\QQ \QQ_\ell).$$ 
Abelian varieties of $\gl2$-type were studied by several authors. For example, we mention the fundamental contributions of Ribet \cite{ribet:realmultdiv,ribet:gl2}. 
Further, we remark that elliptic curves and rational points on Hilbert modular varieties provide natural examples of abelian varieties of $\gl2$-type, and there exists a vast literature on special classes (e.g. Hilbert-Blumenthal type) of abelian varieties of $\gl2$-type; see for instance \cite{vandergeer:hilbertmodular}.

More generally, we say that $A$ is of product $\gl2$-type if $A$ is isogenous 
to a product $A_1\times_S\dotsc\times_S A_n$ of abelian schemes $A_1,\dotsc,A_n$ over $S$ which are all of $\gl2$-type. 

\subsection{Height and conductor}

Let $A$ be an abelian variety over $\QQ$ of dimension $g\geq 1$. We denote by $h_F(A)$ the stable Faltings height of $A$, and we denote by $N_A$ the conductor of $A$. See Section \ref{sec:heights} for the definitions of $h_F(A)$ and $N_A$. We obtain the following result.

\begin{theorem}\label{thm:gl2}
If $A$ is of product $\gl2$-type, then the following statements hold. 
\begin{itemize}
\item[(i)] There is an effective constant $k$, depending only on $g,N_A$, such that $h_F(A)\leq k.$
\item[(ii)] It holds $h_F(A)\leq (3N_A)^{12}+(8g)^6\log N_A.$
\end{itemize}
\end{theorem}

We point out that Theorem \ref{thm:gl2} generalizes Proposition \ref{prop:he} which holds for elliptic curves over $\QQ$, since any elliptic curve is an abelian variety of $\gl2$-type.

Further, we mention that the proofs of Theorem \ref{thm:gl2} (i) and (ii) are in principle the same. The only difference is that in (i) we use isogeny estimates based on ``essentially algebraic" methods, and in (ii) we apply isogeny estimates coming from transcendence. On calculating explicitly the constant $k$ in our proof of (i), it turns out that the resulting inequality is worse than (ii). However, for certain abelian varieties the method of (i) is capable to produce inequalities which are better than (ii). For example, on refining the proof of (i) for semi-stable elliptic curves over $\QQ$, we obtain the following asymptotic result in which  $\gamma=0.5772\dotsc$ denotes Euler's constant and  $f(x)=x\log (x)\log \log x$ for $x\gg 1$.

\begin{proposition}\label{prop:h}
If $E$ is a semi-stable elliptic over $\QQ$, then $$h_F(E)\leq \frac{e^{\gamma}}{4\pi^2}f(N_E)+o(f(N_E)).$$
\end{proposition}

We remark that if $E$ is a semi-stable elliptic curve over $\QQ$  with good reduction at the primes 2 and 3, then Proposition \ref{prop:h} can be slightly improved to $$h_F(E)\leq \frac{e^\gamma}{6\pi^2}f(N_E)+o(f(N_E)).$$
Indeed, this inequality follows on replacing in the proof of Proposition \ref{prop:h} the asymptotic estimate (\ref{eq:ullmo}) of Ullmo by the asymptotic formula (\ref{eq:jk}) of Jorgenson-Kramer.

\subsection{Proof of Theorem \ref{thm:gl2} and Proposition \ref{prop:h}}

In the first part of this section, we collect some useful properties of abelian varieties over $\QQ$ of $\gl2$-type. In the second part, we combine these properties with results obtained in Sections \ref{sec:heights}, \ref{sec:var} and \ref{sec:hf} to prove Theorem \ref{thm:gl2} and Proposition \ref{prop:h}.

\subsubsection{Preliminaries}\label{sec:serremod}

To prove Theorem \ref{thm:gl2} we use Serre's modularity conjecture \cite[(3.2.4)$_{\textnormal{?}}$]{serre:representations}. Building on the work of many mathematicians, Khare-Wintenberger \cite{khwi:serre} recently proved Serre's modularity conjecture. 
Furthermore, Ribet  generalized the arguments of Serre \cite[Th\'eor\`eme 5]{serre:representations} and he showed in \cite[Theorem 4.4]{ribet:gl2} that Serre's modularity conjecure has the following consequence. 
Suppose that $A$ is an abelian variety over $\QQ$ of $\gl2$-type. If $A$ is $\QQ$-simple, then there exists an integer $N\geq 1$ together with a surjective morphism 
\begin{equation}\label{eq:serreconj}
J_1(N)\to A
\end{equation}
of abelian varieties over $\QQ$. 
Here $J_1(N)$ denotes the usual modular Jacobian, defined for example in Section \ref{sec:hf}. We note that Serre and Ribet used inter alia the Tate conjecture \cite{faltings:finiteness} to prove the implication ``Serre's modularity conjecture $\Rightarrow$ (\ref{eq:serreconj})".

We now collect additional results 
which shall be used in the proof of Theorem \ref{thm:gl2}. Assume that $A$ is an abelian variety over $\QQ$ of $\gl2$-type. Then there exists  a $\QQ$-simple abelian variety $B$ over $\QQ$ of $\gl2$-type, an integer $n\geq 1$ and a $\QQ$-isogeny
\begin{equation}\label{eq:factorization}
A\to B^n
\end{equation}
for $B^n$ the $n$-fold product of $B$. This result was established by Ribet in course of his proof of \cite[Theorem 2.1]{ribet:gl2}. 
Further, we shall use the following important property of abelian varieties of $\gl2$-type. Suppose that $A$ and $A'$ are $\QQ$-isogenous abelian varieties over $\QQ$. Then $A$ is of $\gl2$-type if and only if $A'$ is of $\gl2$-type. Indeed this follows directly from $\dim(A)=\dim(A')$ and $\nd(A)\otimes_\ZZ \QQ\cong \nd(A')\otimes_\ZZ \QQ$. 

For any abelian variety $A$ over $\QQ$, we denote by $N_A$ the conductor of $A$. Let $N\geq 1$ be an integer and consider the congruence subgroup  $\Gamma_1(N)\subset \textnormal{SL}_2(\ZZ)$. For any normalized newform $f\in S_2(\Gamma_1(N))$, we let $A_f=J_1(N)/I_fJ_1(N)$ be the abelian variety over $\QQ$ associated to $f$ by Shimura's construction; see Section \ref{sec:modulardegree} for the definitions with respect to $\Gamma_0(N)$ and replace therein $\Gamma_0(N)$ by $\Gamma_1(N)$. 
The abelian variety $A_f$ is $\QQ$-simple and the $\QQ$-simple ``factors" of $J_1(N)$ are unique up to $\QQ$-isogenies. Thus \cite[Proposition 2.3]{ribet:realmult} implies that any $\QQ$-simple quotient of $J_1(N)$ is $\QQ$-isogenous to $A_f$ for some normalized newform $f\in S_2(\Gamma_1(M))$ of level $M$ with $M\mid N$. 
Further, a result of Carayol in \cite{carayol:conductor} gives for any positive integer $M$ that any normalized newform $f\in S_2(\Gamma_1(M))$ satisfies $N_{A_f}=M^{\dim(A_f)}$.  
We now suppose that $A$ is the $\QQ$-simple abelian variety over $\QQ$ of $\gl2$-type, which appears in  (\ref{eq:serreconj}). Then on combining the above observations, we see that one can choose the number $N$ in (\ref{eq:serreconj}) such that
\begin{equation}\label{eq:serrecond}
N_{A}=N^{\dim(A)}.
\end{equation}
We are now ready to prove Theorem \ref{thm:gl2} and Proposition \ref{prop:h}. 

\subsubsection{Proofs}
We continue the notation of the previous section. For any abelian variety $A$ over $\QQ$, we denote by $h_F(A)$ the stable Faltings height of $A$.

As already mentioned, our proofs of Theorem \ref{thm:gl2} (i) and (ii) are essentially the same. We now  describe the proof of (ii), which is divided into the following two parts. In the first part, we use (\ref{eq:serreconj}) and (\ref{eq:serrecond}) to show that any abelian variety $A$ as in Theorem \ref{thm:gl2} is $\QQ$-isogenous to a product $\prod A_i^{e_i}$, where $e_i\geq 1$ is an integer and $A_i$ is an abelian subvariety of $J_1(N_i)$ for $N_i\geq 1$ an integer dividing $N_{A_i}$.  In the second part, we  combine results from  Sections \ref{sec:heights} and \ref{sec:var} to deduce then an upper bound for $h_F(A)$ in terms of $h_F(J_1(N_i))$, $\dim(J_1(N_i))$ and $g=\dim(A)$, and then in terms of $N_i$ and $g$ by Section \ref{sec:hf}, and finally in terms of $N_A$ and $g$ since each $N_i$ divides $N_{A_i}$ and since  $\prod N_{A_i}^{e_i}=N_A$.

\begin{proof}[Proof of Theorem \ref{thm:gl2}]
We take an abelian variety $A$ over $\QQ$ of dimension $g\geq 1$ and we assume that $A$ is of product $\gl2$-type.

1. Poincar\'e's reducibility theorem gives positive integers $e_i$ together with $\QQ$-simple abelian varieties $A_i$ over $\QQ$ such that $A$ is $\QQ$-isogenous to the product $\prod A_i^{e_i}$. We write $g_i$ for the dimension of $A_i$. The $\QQ$-simple ``factors" $A_i$ of $A$ are unique up to $\QQ$-isogeny, and by assumption $A$ is $\QQ$-isogenous to a product of abelian varieties over $\QQ$ of $\gl2$-type. Therefore (\ref{eq:factorization}) implies that $A_i$ is $\QQ$-isogenous to an abelian variety over $\QQ$ of $\gl2$-type, and thus  $A_i$ is of $\gl2$-type as well. 
Then the results collected in (\ref{eq:serreconj}) and (\ref{eq:serrecond}) provide a positive integer $N_i$ with $N_i^{g_i}=N_{A_i}$ together with a surjective morphism 
\begin{equation}\label{eq:quotient}
J_i=J_1(N_i)\to A_i
\end{equation}
of abelian varieties over $\QQ$. 
Let $B_i$ be the identity component of the kernel of $J_i\to A_i$. It is an abelian subvariety of $J_i$. 
Then Poincar\'e's reducibility theorem gives a complementary abelian subvariety $A_i'$ of $J_i$ together with a $\QQ$-isogeny $A_i'\times_\QQ B_i\to J_i$ induced by addition. 
We next verify that $A_i$ and $A_i'$ are $\QQ$-isogenous. The kernel of the surjective morphism $J_i\to A_i$ is a $\QQ$-subgroup scheme of $J_i$ 
whose dimension coincides with $\dim(B_i)$. 
Hence, the dimension formula implies that the dimensions of $A_i$ and $A_i'$ coincide. 
Let $A_i'\to A_i$ be the morphism obtained by composing the natural inclusion $A_i'\hookrightarrow A_i'\times_\QQ B_i$ with the $\QQ$-isogeny $A_i'\times_\QQ B_i\to J_i$ and then with the morphism $J_i\to A_i$. We recall that the surjective morphism $A_i'\times_\QQ B_i\to J_i$ is induced by addition, and $B_i$ is the identity component of the kernel of $J_i\to A_i$. Therefore we see that the morphism $A_i'\to A_i$ is surjective, 
and thus it is a $\QQ$-isogeny 
 since $\dim(A_i')=\dim(A_i)$.
Hence, after replacing $A_i$ by $A_i'$, we may and do assume that  $A_i$ is an abelian subvariety of $J_i$ and that there exists a $\QQ$-isogeny
\begin{equation}\label{eq:subvariety}
A_i\times_\QQ B_i\to J_i.
\end{equation}

2. We now begin to estimate the heights. For any abelian variety $B$ over $\QQ$, we denote by $v_B$ the maximal variation of the stable Faltings height $h_F$ in the $\QQ$-isogeny class of $B$; that is $v_B=\sup \lvert h_F(B)-h_F(B')\rvert$ with the supremum taken over all abelian varieties $B'$ over $\QQ$ which are $\QQ$-isogenous to $B$. The abelian variety $A'=\prod A_i^{e_i}$ satisfies
\begin{equation}\label{eq:sumhfa}
h_F(A')=\sum e_i h_F(A_i), \ \  \textnormal{ and } \ \ h_F(A)\leq v_{A'}+h_F(A')
\end{equation}
since $A$ is an abelian variety over $\QQ$ which is $\QQ$-isogenous to $A'$.
Write $\di=\dim(J_i)$ and define $J_i'=A_i\times_\QQ B_i$.
It holds that $h_F(A_i)=h_F(J_i')-h_F(B_i)$, and it follows from (\ref{eq:subvariety}) that $h_F(J_i')$ is at most $v_{J_i}+h_F(J_i)$. Therefore the lower bound for $h_F(B_i)$ in (\ref{eq:bost}) implies 
\begin{equation}\label{eq:hfai}
h_F(A_i)\leq v_{J_i}+h_F(J_i)+\frac{\di}{2}\log(2\pi^2).
\end{equation}
Here we used that the dimension of $B_i$ is at most $\dim(J_i')=\di$ and that the lower bound for $h_F(B_i)$ in (\ref{eq:bost}) holds in addition for $B_i=0$ since  $h_F(0)=0$. To control $\di$ in terms of $N_i$, we consider the modular curve
$X_1(N_i)=X(\Gamma_1(N_i))$ over $\QQ$ defined in Section \ref{sec:hf}. We recall that $J_i=J_1(N_i)$ is the Jacobian of $X_1(N_i)$ and hence the genus of $X_1(N_i)$ coincides with the dimension $\di$ of  $J_i$. Therefore (\ref{eq:index}) together with \cite[p.107]{dish:modular} implies 
\begin{equation}\label{eq:\di}
\di\leq \frac{1}{24}N_i^2.
\end{equation}
To bound $N_{J_i}$ in terms of $N_i$, we use the classical result of Igusa which says that $X_1(N_i)$ has good reduction at all primes $p\nmid N_i$.  
In particular, the Jacobian $J_i=\textnormal{Pic}^0(X_1(N_i))$ of $X_1(N_i)$ has good reduction at all primes $p\nmid N_i$. It follows that all prime factors of $N_{J_i}$ divide $N_i$. Then (\ref{eq:abcondineq}) gives an effective bound for $N_{J_i}$ in terms of $\di$ and $N_i$, which together with (\ref{eq:\di}) shows that there exists an effective constant $c_i$, depending only on $N_i$, such that
\begin{equation}\label{eq:nji}
N_{J_i}\leq c_i. 
\end{equation}
We recall that $A'=\prod A_i^{e_i}$ is an abelian variety over $\QQ$ which is $\QQ$-isogenous to $A$. Hence, we get that $N_A=N_{A'}$ and then the equality $N_{A_i}=N_i^{g_i}$ in statement (\ref{eq:quotient}) gives 
\begin{equation}\label{eq:condna}
N_{A}=\prod N_{A_i}^{e_i}=\prod N_i^{e_ig_i}.
\end{equation}
In addition the dimension formula gives that $g=\sum e_ig_i$. In particular we obtain $e_i\leq g$. 

We now prove (i). Lemma \ref{lem:abisogvar} (i) gives an effective upper bound for $v_{A'}$ in terms of  $\dim(A')=g$ and $N_{A'}=N_A$, and for $v_{J_i}$ in terms of $\di$ and $N_{J_i}$. On combining these upper bounds with (\ref{eq:sumhfa}) and (\ref{eq:hfai}), we obtain an effective estimate for $h_F(A)$ in terms of $g$, $N_A$, $\di$, $h_F(J_i)$ and $N_{J_i}$; then in terms of $g$, $N_A$ and $N_i$ by (\ref{eq:\di}), Lemma \ref{lem:hj1n} and (\ref{eq:nji}); and finally in terms of $g$ and $N_A$ by (\ref{eq:condna}). This completes the proof of (i).

To show (ii) we use  (\ref{eq:garehdiff}). It gives an upper bound for $v_{A'}$ in terms of  $\dim(A')=g$ and $h_F(A')$, and for $v_{J_i}$ in terms of $\di$ and $h_F(J_i)$. On combining these upper bounds with (\ref{eq:sumhfa}) and (\ref{eq:hfai}), we obtain an estimate for $h_F(A)$ in terms of $g$, $e_i$, $\di$ and $h_F(J_i)$. 
Then (\ref{eq:\di}) together with the upper bound for $h_F(J_i)$ in Lemma \ref{lem:hj1n} lead to an estimate for $h_F(A)$ in terms of $g$, $e_i$ and $N_i$. More precisely, on computing in each step the bounds explicitly, one obtains for example the following estimate 
\begin{equation}\label{eq:ma}
h_F(A)\leq (3N'_A)^{12}+(8g)^6\log N'_A, \ \ \ \ N'_A=\prod N_i^{e_i}.
\end{equation}
To simplify the bound we used here that one can assume $N_i\geq 2$ and $g\geq 2$. Indeed, the equality $N_{A_i}=N_i^{g_i}$ in statement (\ref{eq:quotient}) together with Fontaine's result \cite{fontaine:noabz} implies that $N_i\geq 2$, and if $g=1$ then Proposition \ref{prop:he} combined with $h_F(A)\leq h(A)$ gives an inequality which is even better than (\ref{eq:ma}). Finally, it follows from (\ref{eq:condna}) that $N_A'$ divides $N_A$ and then (\ref{eq:ma}) implies (ii). This completes the proof of Theorem \ref{thm:gl2}. 
\end{proof}

We now discuss a possible variation of the proof of Theorem \ref{thm:gl2}.
For any integer $N\geq 1$, we denote by $J(N)$ the modular Jacobian defined in Section \ref{sec:hf}. 
In the proof of Theorem \ref{thm:gl2} it is possible to work  with $J(N)$ instead of $J_1(N)$ by using the canonical morphism $J(N)\to J_1(N)$.  
However, the resulting inequality would not be as good as the inequalities provided by Theorem \ref{thm:gl2}, since the bounds for $\dim(J(N))$ and $h_F(J(N))$ in terms of $N$ (see Section \ref{sec:hf}) are worse than the corresponding estimates for $J_1(N)$.

\begin{proof}[Proof of Proposition \ref{prop:h}]
We take a semi-stable elliptic curve $E$ over $\QQ$ as in the proposition. Write $N=N_E$ for the conductor of $E$, and let $J_0(N)$ be as in Section \ref{sec:modulardegree}. The Shimura-Taniyama conjecture gives a surjective morphism $$J_0(N)\to E$$ of abelian varieties over $\QQ$, see (\ref{eq:j0nq}) for details. Then, as in the first part of the proof of Theorem \ref{thm:gl2}, we find an abelian subvariety $E'$ of $J_0(N)$  which is $\QQ$-isogenous to $E$. 
The conductor $N$ of the semi-stable abelian variety $E$ is square-free.
Therefore we get that $J_0(N)$ is semi-stable 
and then an application of Lemma \ref{lem:abisogvar} (iii) with the abelian subvariety $E'$ of $J_0(N)$ gives the inequality $$h(E')\leq h(J_0(N))+\frac{g}{2}\log(8\pi^2).$$ 
Here $g$ denotes the dimension of $J_0(N)$, and  $h(A)$ denotes the relative Faltings height of an arbitrary abelian variety $A$ over $\QQ$. The dimension $g$ coincides with the genus of the modular curve $X_0(N)$ defined in Section \ref{sec:cuspforms}. Therefore the upper bound for the genus of $X_0(N)$  in (\ref{eq:destimate}) together with standard analytic estimates leads to $$g\leq \frac{e^\gamma}{2\pi^2}N\log\log N+o(N).$$
Here $\gamma=0.5772\dotsc$ denotes Euler's constant. We write $f(N)=N\log (N)\log\log N$. It holds that $h(J_0(N))=h_F(J_0(N))$, since $J_0(N)$ is semi-stable.
Hence, on combining the displayed inequalities with Ullmo's upper bound for $h_F(J_0(N))$ in (\ref{eq:ullmo}), we deduce $$h(E')\leq \frac{e^\gamma}{4\pi^2}f(N)+o(f(N)).$$
Further, Lemma \ref{lem:abisogvar} (ii) provides that $h(E)\leq h(E')+\frac{1}{2}\log 163$, and we get that $h_F(E)=h(E)$ since $E$ is semi-stable. Hence, we see that the displayed upper bound for $h(E')$ in terms of $f(N)$ proves Proposition \ref{prop:h}.
\end{proof}

We remark that the proof of Proposition \ref{prop:h} shows in addition that any semi-stable elliptic curve $E$ over $\QQ$, with conductor $N_E$ and relative Faltings height $h(E)$, satisfies
\begin{equation}\label{eq:he2}
h(E)\leq 10^8(N_E\log N_E)^6.
\end{equation}
Indeed this follows on replacing in the proof of Proposition \ref{prop:h} the asymptotic estimate for $h_F(J_0(N))$ by the explicit Lemma \ref{lem:hj1n}. We mention that on working in the proof of Theorem \ref{thm:gl2} (ii) with $J_0(N)$ instead of $J_1(N)$ (similar as in Proposition \ref{prop:h}), one can remove\footnote{One uses in addition that $h(E)\leq h_F(E)+\frac{5}{12}\log N_E+2\log 2$. This inequality follows by combining the Noether formula, \cite[Proposition 5.2 (iv)]{rvk:szpiro} and the classification of Kodaira-N\'eron.} the semi-stable assumption in (\ref{eq:he2}) and therewith one obtains (\ref{eq:he2}) for all elliptic curves over $\QQ$.  We notice that (\ref{eq:he2}) improves \cite[Theorem 2.1]{rvk:height} in the case of elliptic curves over $\QQ$. On the other hand,  (\ref{eq:he2}) is worse than  Proposition \ref{prop:he}. 

We point out that the above described proof of (\ref{eq:he2}) for all elliptic curves over $\QQ$, which uses e.g. isogeny estimates and results from Arakelov theory, is very different to the proof of Proposition \ref{prop:he} which applies inter alia the theory of modular forms. In fact the only common tool is the ``geometric" version (\ref{eq:j0nq}) of the Shimura-Taniyama conjecture.

\section{Effective Shafarevich conjecture}\label{sec:es}

In the first part of this section, we discuss several aspects of the effective Shafarevich conjecture. In the second part, we give our explicit version of the effective Shafarevich conjecture for abelian varieties of product $\gl2$-type and we deduce some applications.  In the third part, we then prove the results of Section \ref{sec:es} and finally in the last part we give a generalization for $\QQ$-virtual abelian varieties of $\gl2$-type.

\subsection{Effective Shafarevich conjecture}

Let $S$ be a non-empty open subscheme of $\sp(\ZZ)$ and let $g\geq 1$ be an integer. We denote by $h_F(A)$ the stable Faltings height of an abelian scheme $A$ over $S$. See Section \ref{sec:heights} for the definition. We now recall the effective Shafarevich conjecture.

\vspace{0.3cm}
\noindent{\bf Conjecture $\ces$.}
\emph{There exists an effective constant $c$, depending only on $S$ and $g$, such that any abelian scheme $A$  over $S$ of relative dimension $g$ satisfies $h_F(A)\leq c.$}
\vspace{0.3cm}

The case $g=1$ of this conjecture was proven in course of the proof of Theorem \ref{thm:ms}. In fact, up to a height comparison, this case was already established by Coates \cite{coates:shafarevich} who used the theory of logarithmic forms. Conjecture $\ces$ is widely open when $g\geq 2$.

We mention that Conjecture $\ces$ would have striking applications to classical Diophantine problems. For example, the following proposition gives that Conjecture $\ces$ implies the effective Mordell conjecture for curves over number fields.

\begin{proposition}\label{prop:mordell}
Suppose that Conjecture $\ces$ holds. If $X$ is a smooth, projective and geometrically connected curve of genus at least $2$, defined over an arbitrary number field, then one can determine in principle all rational points of $X$.
\end{proposition}

Let $K$ be a number field. In what follows, by a curve over $K$ we always mean a smooth, projective and geometrically connected curve over $K$. For any curve $X$ over $K$, we denote by $h_F(X)$ the stable Faltings height of the Jacobian $\textnormal{Pic}^0(X)$ of $X$. In the proof of Proposition \ref{prop:mordell} we will show that Conjecture $\ces$ implies in particular the following ``classical'' effective Shafarevich conjecture $\ces^*$ for curves over $K$.

\vspace{0.3cm}
\noindent{\bf Conjecture $\ces^*$.}
\emph{Let $T$ be a finite set of places of $K$. There exists an effective constant $c$, depending only on $K$, $T$ and $g$, such that any curve $X$ over $K$ of genus $g$, with good reduction outside $T$, satisfies $h_F(X)\leq c$.}
\vspace{0.3cm}

We now discuss several aspects of Conjectures $\ces$ and $\ces^*$. First, we mention that Conjecture $\ces$, which implies $\ces^*$, is a priori considerably stronger than $\ces^*$. 
For example, if $X$ is a curve over $K$, then Conjecture $\ces$ would allow in addition to control the finite places of $K$ where the reductions of $X$ and $\textnormal{Pic}^0(X)$ are different; see the discussion at the end of Section \ref{sec:es} for more details. Furthermore, it is shown in \cite{vkkr:intpointsshimura} that already special cases of Conjecture  $\ces$, as Theorem \ref{thm:es} below, have direct applications to the effective study of Diophantine equations. On the other hand, one needs to prove Conjecture $\ces^*$ in quite general situations to get effective Diophantine applications. 

We remark that de Jong-R\'emond \cite{jore:shafarevich} established Conjecture $\ces^*$ for curves over $K$ which are geometrically cyclic covers of prime degree of the projective line $\mathbb P^1_K$. They combined the method introduced by Par{\v{s}}in \cite{parshin:shafarevich} with the theory of logarithmic forms; see also the proof of \cite[Theorem 3.2]{rvk:szpiro} for some refinements. 
The results in \cite{jore:shafarevich} and \cite{rvk:szpiro}  
are not general enough to deduce effective applications for Diophantine equations via the known constructions of Kodaira or Par{\v{s}}in \cite{parshin:construction}.

We point out that a geometric analogue of Conjecture $\ces$ was  established by Faltings \cite{faltings:arakelovtheorem}, see also Deligne \cite[p.14]{deligne:monodromie} for some refinements. 

Conjecture $\ces$ is in fact equivalent to the following conjecture: For any non-empty open subscheme $S$ of the spectrum of the ring of integers of $K$, there exists an effective constant $c$, depending only on $K$, $S$ and $g$, such that any abelian scheme $A$ over $S$ of relative dimension $g$ satisfies $h_F(A)\leq c$. To prove the equivalence one uses inter alia the Weil restriction. We refer to the proof of Proposition \ref{prop:mordell} for details.

 Finally, we mention that one can formulate Conjecture $\ces$ more classically in terms of $\QQ$-isomorphism classes of abelian varieties over $\QQ$ of dimension $g$, with good reduction outside  a finite set of rational prime numbers. However, our formulation of Conjecture $\ces$ in terms of abelian schemes is more compact and is more convenient for the effective study of integral points on moduli schemes.

\subsection{Abelian schemes of product $\gl2$-type}\label{sec:esgl2}

We continue the notation of the previous section. Let $S$ be a non-empty open subscheme of $\sp(\ZZ)$ and let $g\geq 1$ be an integer.  We write $N_S=\prod p$
with the product taken over all rational prime numbers $p$ which are not in $S$.  The following theorem establishes the effective Shafarevich conjecture $\ces$ for all abelian schemes of product $\gl2$-type.

\begin{theorem}\label{thm:es}
Let $A$ be an abelian scheme over $S$ of relative dimension $g$. If $A$ is of product $\gl2$-type, then $$h_F(A)\leq (3g)^{144g}N_S^{24}.$$
\end{theorem}

We point out that the bound in Theorem \ref{thm:es} is polynomial in terms of $N_S$. In course of the proof of Theorem \ref{thm:es} we shall obtain the more precise inequality (\ref{eq:precisebound}), which improves in particular the estimate of Theorem \ref{thm:es} and which is polynomial in terms of the relative dimension $g$ of $A$.  
Moreover, it is possible to refine (\ref{eq:precisebound}) in special cases. For example, we obtain the following result for semi-stable abelian varieties of product $\gl2$-type.

\begin{proposition}\label{prop:ss}
Let $A$ be an abelian scheme over $S$ of relative dimension $g$. If $A$ is of product $\gl2$-type and if the generic fiber of $A$ is semi-stable, then $$h_F(A)\leq g(3N_S)^{12}+(6g)^7\log N_S.$$ 
\end{proposition}

Next, we deduce from Theorem \ref{thm:es} new cases of the ``classical" effective Shafarevich conjecture $\ces^*$.  We say that a curve $X$ over $\QQ$ is of product $\gl2$-type if the Jacobian $\textnormal{Pic}^0(X)$ of $X$ is of product $\gl2$-type. There exist many curves over $\QQ$ of genus $\geq 2$ which are of product $\gl2$-type, see for example the articles in \cite{ribet:gl2conference}. Let $T$ be a finite set of rational prime numbers and write $N_T=\prod p$ with the product taken over all $p\in T$.

\begin{corollary}\label{cor:esc}
Let $X$ be a curve over $\QQ$ of genus $g$ which is of product $\gl2$-type. If $\textnormal{Pic}^{0}(X)$ has good reduction outside $T$, then $$h_F(X)\leq (3g)^{144g}N_T^{24}.$$ 
\end{corollary}
\begin{proof}
The  N\'eron model of $\textnormal{Pic}^0(X)$ over $S=\sp(\ZZ)-T$ 
is an abelian scheme, since $\textnormal{Pic}^{0}(X)$ has good reduction outside $T$. Therefore  Theorem \ref{thm:es} implies Corollary \ref{cor:esc}.
\end{proof}
If $X$ is a curve  over $\QQ$ with good reduction at a rational prime $p$, then $\textnormal{Pic}^0(X)$ has good reduction at $p$. This shows that Corollary \ref{cor:esc} establishes in particular the ``classical" effective Shafarevich conjecture $\ces^*$ for all curves over $\QQ$ of product $\gl2$-type. 

We now derive new isogeny estimates for abelian varieties over $\QQ$ of product $\gl2$-type. Masser-W\"ustholz bounded in \cite{mawu:abelianisogenies,mawu:factorization} the minimal degree of isogenies of abelian varieties. On combining Theorem \ref{thm:es} with the most recent version of the Masser-W\"ustholz results, due to Gaudron-R\'emond \cite{gare:isogenies}, we obtain the following corollary.

\begin{corollary}\label{cor:isoest}
Suppose that $A$ and $B$ are isogenous abelian schemes over $S$ of relative dimension $g$. If $A$ or $B$ is of product $\gl2$-type, then the following statements hold.
\begin{itemize}
\item[(i)] There exist  isogenies $A\to B$ and $B\to A$ of degree at most $(14g)^{(12g)^5}N_S^{(37g)^3}.$
\item[(ii)] In particular it holds $\lvert h_F(A)-h_F(B)\rvert\leq (30g)^3\log N_S+(9g)^6.$
\end{itemize}
\end{corollary}

We point out that these isogeny estimates are independent of  $A$ and $B$. This is absolutely crucial for certain Diophantine applications such as for example \cite{vkkr:intpointsshimura} or Theorem \ref{thm:qes} below. On calculating the constant of Lemma \ref{lem:abisogvar} (i) explicitly, we see that Corollary \ref{cor:isoest} (ii) is  exponentially better in terms of $N_S$ and $g$ than Lemma \ref{lem:abisogvar} (i). We also note that Corollary \ref{cor:isoest} (ii) holds with $h_F$ replaced by the relative Faltings height $h$.

We denote by $M_{\gl2,g}(S)$ the set of isomorphism classes  of abelian schemes over $S$ of relative dimension $g$ which are of product $\gl2$-type.  Corollary \ref{cor:isoest} (i) is one of the main ingredients for the proof of the following quantitative finiteness result for $M_{\gl2,g}(S)$.

\begin{theorem}\label{thm:qes}
The cardinality of $M_{\gl2,g}(S)$ is at most $(14g)^{(9g)^6}N_S^{(18g)^4}.$
\end{theorem}

We refer  to Section \ref{sec:hcell} for a discussion of the important special case $g=1$ of Theorem \ref{thm:qes}. To state some consequences of Theorem \ref{thm:qes} for $\QQ$-isomorphism classes of abelian varieties over $\QQ$, we recall that $T$ denotes a finite set of rational prime numbers and we let $N_T=\prod_{p\in T}p$ be as above. We obtain the following corollary.

\begin{corollary}\label{cor:qisos}
Let $A$ be an abelian variety over $\QQ$ of dimension $g$. We assume that $A$ has the following properties: $(a)$ $A$ is of product $\gl2$-type and $(b)$ $A$ has good reduction outside $T$. Then the following statements hold.
\begin{itemize}
\item[(i)]  Up to $\QQ$-isomorphisms, there exist at most $(14g)^{(9g)^6}N_T^{(18g)^4}$ abelian varieties over $\QQ$ which are $\QQ$-isogenous to $A$.
\item[(ii)] Up to $\QQ$-isogenies, there exist at most $(3g)^{32g^2}N_T^{4g}$  abelian varieties over $\QQ$ of dimension $g$ which have the properties $(a)$ and $(b)$.
\end{itemize}
\end{corollary}

We remark that it is possible to prove a considerably more general version of Corollary \ref{cor:qisos} (ii) by refining Faltings' proof of \cite[Satz 5]{faltings:finiteness} with an effective \v Cebotarev density theorem; see for example Deligne \cite{deligne:isoclasses}. However, the resulting unconditional bound for the number of isogeny classes would be worse than the estimate in Corollary \ref{cor:qisos} (ii).

\subsection{Proof of the results of Section \ref{sec:es}}

In the first part of this section, we collect useful results for abelian schemes. In the second part, we first show Theorem \ref{thm:es}, Proposition \ref{prop:ss} and Corollary \ref{cor:isoest},  then we prove Theorem \ref{thm:qes} and Corollary \ref{cor:qisos}, and finally we give the proof of Proposition \ref{prop:mordell}.

\subsubsection{Preliminaries}

Let $S$ be a connected Dedekind scheme, with field of fractions $K$. We begin to prove useful properties of morphisms of abelian schemes over $S$.  Suppose that $A$ and $B$ are abelian schemes over $S$ with generic fibers $A_K$ and $B_K$ respectively. Then base change from $S$ to $K$ induces an isomorphism of abelian groups
\begin{equation}\label{eq:homs}
\textnormal{Hom}(A,B)\cong \textnormal{Hom}(A_K,B_K).
\end{equation}
We now verify (\ref{eq:homs}). Any abelian scheme over $S$ is the N\'eron model of its generic fiber, see for example \cite[p.15]{bolura:neronmodels}. 
Thus the N\'eron mapping property gives that any $K$-scheme morphism
 $\varphi_K:A_K\to B_K$ extends to an unique $S$-scheme morphism $\varphi:A\to B$. 
In addition, if $\varphi_K$ is a $K$-group scheme morphism, then  $\varphi$ is a $S$-group scheme morphism. Therefore we see that base change from $S$ to $K$ induces a bijection of sets $\textnormal{Hom}(A,B)\cong \textnormal{Hom}(A_K,B_K)$.  Finally, base change properties of group schemes show that this bijection is in fact a homomorphism of abelian groups and hence we conclude (\ref{eq:homs}). 
Furthermore, base change from $S$ to $K$ induces an isomorphism of rings 
\begin{equation}\label{eq:rings}
\nd(A)\cong \nd(A_K).
\end{equation}
Indeed, if $A=B$ then the isomorphism of abelian groups in (\ref{eq:homs}) is an isomorphism of rings, since base change from $S$ to $K$ is a covariant functor from $S$-schemes to $K$-schemes.

 We shall use the following property of semi-stable abelian varieties. Let $A$ and $B$ be abelian varieties over $K$ and let $v$ be a closed point of $S$. 
If there is a surjective morphism 
\begin{equation}\label{eq:semi-stable}
A\to B
\end{equation}
of abelian varieties over $K$ and if $A$ has semi-stable reduction at $v$, then $B$ has semi-stable reduction at $v$. 
We now verify this statement. Let $C$ be the reduced underlying scheme of the identity component of the kernel of $A\to B$. There exists an abelian variety $B'$ over $K$ which is $K$-isogenous to $B$ and which fits into an exact sequence $0\to C\to A\to B'\to 0$ of abelian varieties over $K$. Therefore the semi-stability of $A$ at $v$ together with \cite[p.182]{bolura:neronmodels} gives that $B'$ has semi-stable reduction at $v$, and then \cite[p.180]{bolura:neronmodels} shows that $B$ is semi-stable at $v$ since $B$ and $B'$ are $K$-isogenous. This proves the assertion in (\ref{eq:semi-stable}).

Next, we give Lemma \ref{lem:condest} which will allow us later to control the conductor of certain abelian varieties. In this lemma, we assume that $K$ is a number field with ring of integers $\OK$ and we assume that $S$ is a non-empty open subscheme of  $\sp(\OK)$. We write $d=[K:\QQ]$ for the degree of $K$ over $\QQ$ and we define $N_S=\prod N_v$ with the product taken over all $v\in\sp(\OK)-S$. Let $g\geq 1$ be an integer and let $\rho=\rho(S,g)$ be the number of rational primes $p$ such that $p\leq 2g+1$ and such that there exists $v\in \sp(\OK)-S$ with $v\mid p$. If $A$ is an abelian scheme over $S$, then we denote by $N_A$ the conductor of $A$ defined in Section \ref{sec:cond}. The following global result in Lemma \ref{lem:condest} (i) uses inter alia the local conductor estimates of Brumer-Kramer \cite{brkr:conductor} stated in (\ref{eq:abcondineq}).

\begin{lemma}\label{lem:condest}
Suppose that $A$ is an abelian scheme over $S$ of relative dimension $g$. Then the following statements hold. 
\begin{itemize}
\item[(i)] There exists a positive integer $\nu$, depending only on $K$, $S$ and $g$, such that $N_A\mid \nu$ and such that $\nu\leq (2g+1)^{6gd\rho}N_S^{2g}.$
\item[(ii)] If the generic fiber of $A$ is semi-stable, then $N_A\mid N_S^g$.
\end{itemize}
\end{lemma}
\begin{proof}
The generic fiber $A_K$ of $A$ has good reduction at all closed points of $S$, since $A$ is an abelian scheme over $S$. 
Therefore $N_A$ takes the form $$N_A=\prod N_v^{f_v}$$
with the product taken over all $v\in \sp(\OK)-S$,  
where $f_v=\varepsilon_v+\delta_v$ for $\varepsilon_v$ and $\delta_v$ the tame and the wild conductor of $A_K$ at $v$ respectively; see for example \cite[Section 2.1]{serre:conductor}. 

We now prove (i). Let $v$ be a closed point of $\sp(\OK)$. We see that $\varepsilon_v\leq\dim V_\ell(A)=2g$ for $V_\ell(A)$ the rational $\ell$-adic Tate module of $A_K$.  
If the residue characteristic $p$ of $v$ satisfies $p>2g+1$, then $\delta_v=0$ and hence $f_v=\varepsilon_v\leq 2g$. 
We denote by $b_v$ the right hand side of the inequality of Brumer-Kramer stated in (\ref{eq:abcondineq}). This $b_v$ is an integer which depends only on $K$, $S$, $g$ and which satisfies $f_v\leq b_v$. Then we observe that $N_A$ divides $$\nu=N_S^{2g}\prod N_v^{(b_v-2g)}$$ with the product taken over all $v\in \sp(\OK)-S$ of residue characteristic at most $2g+1$.  
On using the definition of $b_v$ via (\ref{eq:abcondineq}), we deduce an upper bound for $b_v$ which then leads  to an estimate for $\nu$ as claimed. This completes the proof of (i).

To show (ii) we take again a closed point $v$ of $\sp(\OK)$. Let $\mathcal A_v$ be the fiber at $v$ of the N\'eron model of $A_K$ over $\sp(\OK)$. The identity component $\mathcal A_v^0$ of $\mathcal A_v$ is an extension of an abelian variety $C_v$ by the product of a torus part $T_v$ with a unipotent part $U_v$. 
Let $t_v$, $u_v$ and $a_v$ be the dimensions of $T_v$, $U_v$ and $C_v$ respectively. It holds that $\dim(\mathcal A_v^0)=\dim(\mathcal A_v)=g$ 
 and then the dimension formula gives $g=(t_v+u_v)+a_v$. 
Further, it is known  that $\varepsilon_v=t_v+2u_v$, see for example \cite[p.364]{grra:neronmodels}. 
Our additional assumption in (ii), that $A_K$ is semi-stable, implies that $u_v=0$ and $\delta_v=0$. 
Therefore we deduce that $f_v=\varepsilon_v=t_v$ and this together with $t_v\leq t_v+u_v+a_v=g$ leads to $f_v\leq g$.
Then the displayed formula for $N_A$ shows that $N_A\mid N_S^g$ which proves (ii). This completes the proof of Lemma \ref{lem:condest}. 
\end{proof}

We are now ready to prove the results of Section \ref{sec:es}. 

\subsubsection{Proofs}

We continue the notation of the previous section and we assume that $K=\QQ$. Let   $S$ be a non-empty open subscheme of $\sp(\ZZ)$, and let $N_S$ and $g\geq 1$ be as above. We assume that $A$ is an abelian scheme over $S$ of relative dimension $g$  which is of product $\gl2$-type. Let $h_F(A)$ be  the stable Faltings height of $A$.

The principal ideas of the proof of Theorem \ref{thm:es} are as follows.   
Theorem \ref{thm:gl2} together with Lemma \ref{lem:condest} implies directly Conjecture $\ces$ for $A$, with an inequality of the form $h_F(A)\leq c(g) N_S^{24g}$ for $c(g)$  a constant depending only on $g$. However, to obtain the better bound $h_F(A)\leq c(g)N_S^{24}$ and to improve the dependence on $g$ of  $c(g)$, we go into the proof of Theorem \ref{thm:gl2} and  we apply therein Lemma \ref{lem:condest} with the simple ``factors" of $A$. 

\begin{proof}[Proof of Theorem \ref{thm:es}]
1.   By assumption $A$ is isogenous to a product of abelian schemes over $S$ which are all of $\gl2$-type.  Then (\ref{eq:rings}) gives that the generic fibers of these abelian schemes are all of $\gl2$-type as well. It follows that the generic fiber $A_\QQ$ of $A$ is $\QQ$-isogenous to a product of abelian varieties over $\QQ$ of $\gl2$-type. 
In other words, the abelian variety $A_\QQ$ is of product $\gl2$-type and thus satisfies all the assumptions of Theorem \ref{thm:gl2}.

2. We now go into the proof of Theorem \ref{thm:gl2} (ii). Therein we showed the existence of positive integers $N_i$ and $e_i$, together with $\QQ$-simple abelian varieties $A_i$ over $\QQ$ of dimension $g_i$,   such that $A_\QQ$ is $\QQ$-isogenous to $\prod A_i^{e_i}$ and such that
\begin{equation}\label{eq:nai}
N_i^{g_i}=N_{A_i}.
\end{equation}
Here $N_{A_i}$ denotes the conductor of $A_i$. Furthermore, on following the proof of Theorem \ref{thm:gl2} (ii) and on calculating the first term in the upper bound for $h_F(A)$ given in (\ref{eq:ma}) more precisely, we obtain the sharper inequality 
\begin{equation}\label{eq:hfs}
h_F(A)\leq \sum e_i\bigl(18\cdot 10^3N_i^{12}+(8g)^6\log N_i\bigl).
\end{equation}

3. Next, we estimate the numbers $N_i$ in terms of $g$ and $S$. There exists a surjective morphism $A_\QQ\to A_i$ of abelian varieties over $\QQ$, and the abelian variety $A_\QQ$ has good reduction at all closed points of $S$ since it extends to an abelian scheme over $S$. Therefore \cite[Corollary 2]{seta:goodreduction} provides that $A_i$ has good reduction at all closed points of $S$, and this shows that the N\'eron model $\mathcal A_i$ of $A_i$ over $S$ is an abelian scheme over $S$. 
Then an application of Lemma \ref{lem:condest} with the abelian scheme $\mathcal A_i$ over $S$ of conductor $N_{A_i}$ gives that $N_{A_i}\leq (2g_i+1)^{6g_i\rho_i}N_S^{2g_i},$
where $\rho_i=\rho(S,g_i)$ denotes the number of rational primes $p\notin S$ with $p\leq 2g_i+1$. Further, since $A_\QQ$ is isogenous to $\prod A_i^{e_i}$, we obtain 
\begin{equation}\label{eq:ei}
g=\sum e_ig_i.
\end{equation}
It follows that $g_i\leq g$ and this leads to $\rho_i\leq \rho=\rho(S,g)$. Then the above upper bound for $N_{A_i}$ together with (\ref{eq:nai}) proves that
$N_i\leq (2g+1)^{6\rho}N_S^{2}.$

4. We observe that $\rho\leq 2g$ and (\ref{eq:ei}) implies that $\sum e_i\leq g$. Therefore, on combining (\ref{eq:hfs}) with the above estimate for $N_i$, we deduce an inequality as claimed by Theorem \ref{thm:es}. To simplify the form of the final result, we assumed here that $g\geq 2$. In fact, in course of the proof of Theorem \ref{thm:ms} we obtained an inequality which directly implies the remaining case $g=1$. This completes the proof of Theorem \ref{thm:es}.
\end{proof}

We recall that $\rho=\rho(S,g)$ denotes the number of rational primes $p\notin S$ with $p\leq 2g+1$. The proof of Theorem \ref{thm:es} gives in addition the following more precise result: If $A$ is an abelian scheme over $S$ of relative dimension $g$ and if $A$ is of product $\gl2$-type, then
\begin{equation}\label{eq:precisebound}
h_F(A)\leq g\bigl(18\cdot 10^3\nu_0^{12}+(8g)^6\log \nu_0\bigl), \ \ \ \nu_0=(2g+1)^{6\rho}N_S^2.
\end{equation}
Let $s$ be the number of rational primes which are not in $S$. It follows that $\rho\leq s<\infty$ and then we see that (\ref{eq:precisebound}) is polynomial in terms of $g$, since $s$ depends only on $S$. Furthermore,  on looking for example at products of elliptic curves over $S$, we see that any upper bound for $h_F(A)$ has to be at least linear in terms of $g$. This shows that the polynomial dependence on $g$ of (\ref{eq:precisebound}) is already ``quite close" to the optimum. On the other hand, Lemma \ref{lem:condest} implies that Frey's height conjecture \cite[p.39]{frey:linksulm} would give an upper bound for $h_F(A)$ which is linear in terms of $\log N_S$, while (\ref{eq:precisebound}) depends polynomially on $N_S$.
We remark that an (effective) estimate for $h_F(A)$ which is linear in terms of $\log N_S$ would be very useful, since such an estimate would imply inter alia a (effective) version of the $abc$-conjecture.

In the following proof of Proposition \ref{prop:ss}, we use the arguments of Theorem \ref{thm:es} and we replace therein Lemma \ref{lem:condest} (i) by Lemma \ref{lem:condest} (ii).

\begin{proof}[Proof of Proposition \ref{prop:ss}]
We freely use the notations and definitions of the proof of Theorem \ref{thm:es}. In addition, we assume that $A_\QQ$ is semi-stable. Therefore (\ref{eq:semi-stable}) implies that $A_i$ is semi-stable, since there exists a surjective morphism $A_\QQ\to A_i$ of abelian varieties over $\QQ$.
We showed that $A_i$ extends to an abelian scheme $\mathcal A_i$ over $S$. Thus an application of Lemma \ref{lem:condest} (ii) with the abelian scheme $\mathcal A_i$ over $S$ of conductor $N_{A_i}$ and relative dimension $g_i$ gives that $N_{A_i}\mid N_S^{g_i}$. Hence, the equality $N_i^{g_i}=N_{A_i}$ in (\ref{eq:nai}) implies that $N_i\leq N_S$ and then (\ref{eq:ei}) together with the upper bound for $h_F(A)$  in (\ref{eq:hfs}) leads to an inequality as claimed. This completes the proof of Proposition \ref{prop:ss}.
\end{proof}

To prove Corollary \ref{cor:isoest} we combine Theorem \ref{thm:es} with the most recent version of the Masser-W\"ustholz results \cite{mawu:abelianisogenies,mawu:factorization}, due to Gaudron-R\'emond \cite{gare:isogenies}.

\begin{proof}[Proof of Corollary \ref{cor:isoest}]
We suppose that $A$ and $B$ are isogenous abelian schemes over $S$ of relative dimension $g$. Let  $A_\QQ$ and $B_\QQ$ be the generic fibers of $A$ and $B$  respectively. By assumption $A$ or $B$ is of product $\gl2$-type. Thus both are of product $\gl2$-type.

To show (i) we observe that the constant $\kappa(A_\QQ)$ in \cite{gare:isogenies} depends only on $g$ and  $h_F(A)$. Let $\kappa$ be the constant which one obtains by replacing the number $h_F(A)$ with $(3g)^{144g}N_S^{24}$ in the definition of $\kappa(A_\QQ)$; notice that $\kappa$ depends only on $N_S$ and $g$. An application of Theorem \ref{thm:es} with $A$  shows that $\kappa(A_\QQ)\leq \kappa$. The abelian varieties $A_\QQ$ and $B_\QQ$ are $\QQ$-isogenous. Therefore \cite[Th\'eor\`eme 1.4]{gare:isogenies} gives $\QQ$-isogenies $\varphi_\QQ:A_\QQ\to B_\QQ$ and $\psi_\QQ:B_\QQ\to A_\QQ$ of degree at most $\kappa(A_\QQ)\leq \kappa$. As in the proof of (\ref{eq:homs}) we see that $\varphi_\QQ$ and $\psi_\QQ$ extend to $S$-group scheme morphisms $\varphi:A\to B$ and $\psi:B\to A$ respectively. Furthermore, it follows from \cite[p.180]{bolura:neronmodels} that $\varphi$ and $\psi$ are  isogenies since  $A$ and $B$ are in particular semi-abelian schemes over $S$.
Hence we conclude (i).

It remains to prove (ii). We showed in (i) that there is a $\QQ$-isogeny $\varphi_\QQ:A_\QQ\to B_\QQ$ with $\deg(\varphi_\QQ)\leq \kappa$, and   (\ref{eq:hdegdiff}) gives that $\lvert h(A_\QQ)-h(B_\QQ)\rvert\leq \frac{1}{2}\log \deg(\varphi_\QQ)$ for $h$ the relative Faltings height. Hence we deduce a version of (ii) involving $h$. To prove the version involving $h_F$ we use \cite{grra:neronmodels}. It provides a number field $L$  such that $A_L$ and $B_L$ are semi-stable, and  thus $h_F(A)=h(A_L)$ and $h_F(B)=h(B_L)$. If $\varphi_L:A_L\to B_L$ is the base change of $\varphi_\QQ$, then (\ref{eq:hdegdiff}) gives that $\lvert h(A_L)-h(B_L)\rvert\leq \frac{1}{2}\log \deg(\varphi_L)$. Therefore $\deg(\varphi_L)=\deg(\varphi_\QQ)\leq \kappa$ leads to (ii). This completes the proof of Corollary \ref{cor:isoest}.
\end{proof}

We refer to the introduction for an outline of the following proof of Theorem \ref{thm:qes}.

\begin{proof}[Proof of Theorem \ref{thm:qes}]
We recall that $M_{\gl2,g}(S)$ denotes the set of isomorphism classes of abelian schemes over $S$ of relative dimension $g$ which are of product $\gl2$-type. To bound $\lvert M_{\gl2,g}(S)\rvert$ we may and do assume that $M_{\gl2,g}(S)$ is not empty.

1. We denote by $M_{\gl2,g}(S)_\QQ$ the set of $\QQ$-isomorphism classes of abelian varieties over $\QQ$ of dimension $g$ which extend to an abelian scheme over $S$ and which are of product $\gl2$-type. Base change from $S$ to $\QQ$ induces a canonical bijection
$$M_{\gl2,g}(S)\cong M_{\gl2,g}(S)_\QQ.$$
To verify this statement we observe that $M_{\gl2,g}(S)$ coincides with the set of $S$-scheme isomorphism classes generated by abelian schemes over $S$ of relative dimension $g$ which are of product $\gl2$-type. Further, it follows from (\ref{eq:rings}) that the generic fiber $A_\QQ$ of any $[A]\in M_{\gl2,g}(S)$ is of product $\gl2$-type. Thus base change from $S$ to $\QQ$ induces a map $M_{\gl2,g}(S)\to M_{\gl2,g}(S)_\QQ$, which is surjective by  (\ref{eq:rings}) and \cite[p.180]{bolura:neronmodels}.
The abelian scheme $A$ is the N\'eron model of $A_\QQ$ over $S$ and then the N\'eron mapping property shows that $M_{\gl2,g}(S)\to M_{\gl2,g}(S)_\QQ$ is injective. We conclude that $M_{\gl2,g}(S)\cong M_{\gl2,g}(S)_\QQ$.

2. Next, we estimate the number of distinct $\QQ$-isogeny classes of abelian varieties over $\QQ$ generated by $M_{\gl2,g}(S)_\QQ$.  Let $[A]\in M_{\gl2,g}(S)_\QQ$. In the proof of Theorem \ref{thm:gl2} we constructed positive integers $N_i$ and $e_i$, together with $\QQ$-simple abelian varieties $A_i$ over $\QQ$ of dimension $g_i$ and of conductor $N_{A_i}=N_i^{g_i}$, such that $A$ is $\QQ$-isogenous to $\prod A_i^{e_i}$ and such that $A_i$ is a $\QQ$-quotient of $J_1(N_i)$. Here $J_1(N)$ denotes the usual modular Jacobian of level $N\in \ZZ_{\geq 1}$ defined in Section \ref{sec:hf}. 
The abelian variety $A$ extends to an abelian scheme over $S$, since $[A]\in M_{\gl2,g}(S)_\QQ$. Thus each $A_i$ extends to an abelian scheme over $S$ and then the arguments of the proof of Lemma \ref{lem:condest} together with $g_i\leq g$ lead to $N_{A_i}\mid \nu^{g_i}$ for 
\begin{equation*}
\nu=N_S^{2}\prod p^{c_p}, \ \ \ c_p=6+2\lfloor\log(2g)/\log p\rfloor.
\end{equation*}
Here the product is taken over all rational primes $p\notin S$ with $p\leq 2g+1$, and for any real number $x$ we write $\lfloor x\rfloor$  for the largest integer at most $x$. We warn the reader that the displayed number $\nu$ is related to the number appearing in Lemma \ref{lem:condest} (i), but  these numbers are not necessarily the same. It follows that $N_i\mid \nu$ since $N_{i}^{g_i}=N_{A_i}$,  and this implies that $J_1(N_i)$ is a $\QQ$-quotient of $J_1(\nu)$. On using that $A_i$ is a $\QQ$-quotient of $J_1(N_i)$, we then see that there exists a surjective morphism of abelian varieties over $\QQ$
$$
J_1(\nu)\to A_i.
$$
Hence Poincare's reducibility theorem shows that each $A_i$ is $\QQ$-isogenous to a $\QQ$-simple ``factor" of $J_1(\nu)$. Furthermore, the dimension of $J_1(\nu)$ coincides with the genus $g_\nu$ of the modular curve $X_1(\nu)=X(\Gamma_1(\nu))$ defined in Section \ref{sec:hf}, and the abelian variety $J_1(\nu)$ (resp. $A$) has at most $g_\nu$ (resp. $g$) $\QQ$-simple ``factors" up to $\QQ$-isogenies. 
Therefore there exists a set of at most $g\cdot g_\nu^g$  distinct abelian varieties over $\QQ$ such that any $[A]\in M_{\gl2,g}(S)_\QQ$ is $\QQ$-isogenous to some abelian variety in this set. In other words, the abelian varieties in $M_{\gl2,g}(S)_\QQ$ generate at most $g\cdot g_\nu^g$  distinct $\QQ$-isogeny classes of abelian varieties over $\QQ$.

3. To bound the size of each $\QQ$-isogeny class we take an arbitrary $[A]\in M_{\gl2,g}(S)_\QQ$. We denote by $\mathcal C$ the set of $\QQ$-isomorphism classes of abelian varieties over $\QQ$ which are $\QQ$-isogenous to $A$.  Let $\kappa$ be the constant which appears in the proof of Corollary \ref{cor:isoest}. If $[B]\in \mathcal C$ then the proof of Corollary \ref{cor:isoest} provides a $\QQ$-isogeny $\varphi:A\to B$ of degree at most $\kappa$. Furthermore, the quotient of $A$ by the kernel of $\varphi$ is an abelian variety over $\QQ$ 
which is $\QQ$-isomorphic to $B$. 
On combining the above observations, we  see that $\lvert\mathcal C\rvert$ is bounded from above by the number of subgroups of $A^t$ of order at most $\kappa$, where $A^t$ is the group of torsion points of $A$. 
It holds that $A^t\cong (\QQ/\ZZ)^{2g}$, and    \cite[Lemma 6.1]{mawu:abelianisogenies} gives that $(\QQ/\ZZ)^{2g}$ has at most $\kappa^{2g}$ subgroups of order at most $\kappa$. Hence we deduce that $\lvert\mathcal C\rvert\leq \kappa^{2g}$.

4. The results obtained in 1.-3. imply that $\lvert M_{\gl2,g}(S)\rvert \leq g(g_\nu\kappa^2)^g$, and (\ref{eq:index}) together with \cite[p.107]{dish:modular} proves that $g_\nu\leq \frac{1}{24}\nu^2$. 
Therefore the definitions of $\kappa$ and $\nu$ lead to an upper bound for $\lvert M_{\gl2,g}(S)\rvert$ as claimed in Theorem \ref{thm:qes}. 
\end{proof}

The arguments used in the proof of Theorem \ref{thm:qes} give in addition Corollary \ref{cor:qisos}.

\begin{proof}[Proof of Corollary \ref{cor:qisos}]
We observe that part 3. of the proof of Theorem \ref{thm:qes} implies (i), and we notice that (ii) follows from part 2. of the proof of Theorem \ref{thm:qes}.
\end{proof}

We now prove Proposition \ref{prop:mordell}. In the first part of the proof we show that Conjecture $\ces$ implies Conjecture $\ces^*$, and in the second part we use the effective version of the Kodaira construction due to R\'emond \cite{remond:construction}.

\begin{proof}[Proof of Proposition \ref{prop:mordell}]
We recall some notation. Let $K$ be a number field of degree $d=[K:\QQ]$, with ring of integers $\OK$. We denote by $D_K$ the absolute value of the discriminant of $K$ over $\QQ$. Let $h_F$ be the stable Faltings height and let $T$ be a finite set of places of $K$. We write $N_T=\prod N_v$ with the product taken over all finite places $v\in T$. Let $X$ be a smooth, projective and geometrically connected curve over $K$ of genus $g\geq 1$.

1. To prove that  Conjecture $\ces$ implies $\ces^*$ we assume that Conjecture $\ces$ holds.  In addition, we suppose that the Jacobian $J_K=\textnormal{Pic}^0(X)$ of $X$ has good reduction outside $T$.   The Weil restriction $A_\QQ=\textnormal{Res}_{K/\QQ}(J_K)$ of $J_K$ is an abelian variety over $\QQ$ of dimension $n=dg$, which is geometrically isomorphic to $\prod J_K^\sigma$. Here the product is taken over all embeddings $\sigma$ from $K$ into an algebraic closure of $K$, and $J_K^\sigma$ is the base change of $J_K$ with respect to $\sigma$. 
The  Galois invariance $h_F(J_K)=h_F(J_K^\sigma)$ implies that  $h_F(A_\QQ)=dh_F(J_K)$. Let $S$ be the open subscheme of $\sp(\ZZ)$ formed by the generic point  together with the closed points where $A_\QQ$ has good reduction. The N\'eron model $A$ of $A_\QQ$ over $S$ is an abelian scheme. Therefore an application of Conjecture $\ces$ with $A$, $S$ and $n$ gives an effective constant $c$, depending only on $S$ and $n$, such that 
\begin{equation}\label{eq:weilrest}
dh_F(J_K)=h_F(A_\QQ)\leq c.
\end{equation}
We write $\mathcal D=\{D_K,d,g,N_T\}$ and we now construct an effective constant $c'$, depending only on $\mathcal D$, such that $c\leq c'$. The  finite places in $T$ form a closed subset of $\sp(\OK)$, whose complement $S'$ has the structure of an open subscheme of $\sp(\OK)$. The N\'eron model $J$ of $J_K$ over $S'$ is an abelian scheme, since $J_K$ has good reduction outside $T$. We denote by $N_J$ and $N_A$ the conductors of $J_K$ and $A_\QQ$ respectively. 
A result of Milne \cite[Proposition 1]{milne:arithmetic} gives that $N_A=N_JD_K^{2g}$, and an application of Lemma \ref{lem:condest} (i) with the abelian scheme $J$ over $S'$ of relative dimension $g$ implies that $N_J\leq \Omega D_K^{-2g}$ for $\Omega=(3g)^{12g^2d}(N_TD_K)^{2g}.$ We deduce that $N_A\leq \Omega$ and this leads to $$N_S\leq \Omega,$$ since $N_S$ divides $N_A$ by the construction of $S$.  Here for any  open subscheme $U$ of $\sp(\ZZ)$ we write $N_U=\prod p$ with the product taken over all rational primes $p\notin U$.
It follows that $S\in \mathcal U$ for $\mathcal U$ the set of open subschemes $U$ of $\sp(\ZZ)$ with $N_{U}\leq \Omega$.  An application of Conjecture $\ces$ with $U\in \mathcal U$ and $n$ gives an effective constant $c_U\geq 1$,  depending only on $U$ and $n$. We define $c'=\max c_U$ with the maximum taken over all $U\in \mathcal U$. If $\mathcal D$ is given, then the set $\mathcal U$ can be determined effectively. Thus we see that $c'$  is an effective constant, depending only on $\mathcal D$. On using that $S\in \mathcal U$, we obtain that $c\leq c'$ and then (\ref{eq:weilrest}) gives $$h_F(J_K)\leq c'.$$ 
In other words, we proved that Conjecture $\ces$ would give an effective constant $c'$, depending only on $\mathcal D$, with the following property: If $J_K$ has good reduction outside $T$, then $h_F(J_K)\leq c'$. Further, if $X$ has good reduction at a finite place $v$ of $K$, then $J_K$ has good reduction at $v$. Therefore we conclude that Conjecture $\ces$ implies $\ces^*$.

2. It follows from part 1. that Conjecture $\ces$ implies $\ces^*$. Furthermore,  \cite{remond:construction} gives that Conjecture $\ces^*$ implies that the set of rational points of $X$ can be determined effectively if $g\geq 2$. This completes the proof of Proposition \ref{prop:mordell}.
\end{proof}

We remark that the above proof of Proposition \ref{prop:mordell} assumes the validity of Conjecture $\ces$ in quite general situations. In particular, it is a priori not possible to use the above arguments in order to deduce special cases of the effective Mordell conjecture from special cases of Conjecture $\ces$ such as for example Theorem \ref{thm:es}. To ``transfer" special cases between these conjectures, an effective version of Par{\v{s}}in's construction  \cite{parshin:construction} would be more useful than Kodaira's construction which is used in the proof of Proposition \ref{prop:mordell}. 

We mention that the implication $\ces^*\Rightarrow\ces$ remains an interesting open problem, which is non-trivial since $\ces$ is a priori considerably stronger than $\ces^*$. To discuss parts of the additional information contained in Conjecture $\ces$, we consider an arbitrary hyperelliptic curve $X$ of genus $g\geq 2$ over a number field $K$. Let $T$ be the set of finite places of $K$ where $\textnormal{Pic}^0(X)$ has bad reduction. Suppose that $v$ is a finite place of $K$ where $X$ has bad reduction but $\textnormal{Pic}^0(X)$ has good reduction;  the minimal regular model of  $X$ over $\sp(\mathcal O_v)$ is then automatically semi-stable for $\mathcal O_v$ the local ring at $v$. Then on combining the arguments of \cite[Proposition 5.1 (i)]{rvk:szpiro} with part 1. of the proof of Proposition \ref{prop:mordell}, we see that already very special cases of Conjecture $\ces$  would give an  effective estimate for $N_v$ in terms of $K$, $g$ and $T$. We note that  Levin \cite{levin:siegelshaf} proved  that such an effective estimate for $N_v$ would solve the following classical problem: Give an effective version of Siegel's theorem for arbitrary hyperelliptic curves of genus $g\geq 2$ defined over a number field $K$. In fact the latter problem is already open for $g=2$ and $K=\QQ$.

We also point out that one can improve our inequalities for abelian varieties with ``real multiplications": Let $A$ be an abelian variety over $\QQ$ of positive dimension $g$, with $\textnormal{End}(A)\otimes_\ZZ\QQ$ a totally real number field of degree $g$ over $\QQ$.  Serre showed in \cite[Th\'eor\`eme 5]{serre:representations} that Serre's modularity conjecture (see Section \ref{sec:serremod}) 
gives that $A$ is a $\QQ$-quotient of $J_0(N)$, where $J_0(N)$ is defined in Section \ref{sec:modulardegree} and $N$ is the positive integer whose $g$-th power equals the conductor of $A$. Then,  on combining the bounds for $h_F(J_0(N))$ in Lemma \ref{lem:hj1n} with the arguments of Theorem \ref{thm:gl2} and Theorem \ref{thm:es}, we see that these results hold with better inequalities for abelian varieties such as $A$. 

\subsection{Effective Shafarevich for $\QQ$-virtual abelian varieties of $\gl2$-type}
In this section, we show that our method allows in addition to deal with certain more general abelian varieties over arbitrary number fields. This generalization is required for the effective study (see \cite{vkkr:intpointsshimura}) of those  $S$-points on $Y$ which correspond to  abelian schemes that are not necessarily defined over $S$. Here $S$ is a non-empty open subscheme of $\sp(\ZZ)$ and $Y$ is a certain coarse moduli scheme over $S$ (e.g. Hilbert modular variety).

Following Wu \cite{wu:thesis},  we now define $\QQ$-virtual abelian varieties of $\gl2$-type. They generalize in particular the $\QQ$-Hilbert-Blumenthal abelian varieties of Ribet \cite{ribet:fieldsofdef}. Let $\bar{\QQ}$ be an algebraic closure of $\QQ$. Write $G_\QQ=\textnormal{Gal}(\bar{\QQ}/\QQ)$ for the absolute Galois group of $\QQ$. Let $g\geq 1$ be an integer and let $A$ be an abelian variety over $\bar{\QQ}$ of dimension $g$. We assume that there is a number field $F$ of degree $[F:\QQ]=g$ together with an embedding 
\begin{equation}\label{eq:condition1}
F\hookrightarrow \textnormal{End}^0(A)=\textnormal{End}(A)\otimes_\ZZ\QQ.
\end{equation}
For any $\sigma\in G_\QQ$ and for any $\varphi\in \textnormal{End}^0(A)$, we denote by $A^\sigma$ and $\varphi^\sigma\in \textnormal{End}^0(A^\sigma)$ the by $\sigma:\bar{\QQ}\to\bar{\QQ}$ induced base changes of $A$ and $\varphi$  respectively. In addition we  assume  that for any $\sigma\in G_\QQ$ there exists an isogeny $\mu_\sigma:A^\sigma\to A$ such that 
\begin{equation}\label{eq:condition2}
\mu_\sigma\circ \varphi^\sigma=\varphi\circ\mu_\sigma \ \textnormal{ for all } \varphi\in \textnormal{End}^0(A).
\end{equation}
For any abelian variety $A$ over $\bar{\QQ}$ of dimension $g$, we say that $A$ is a $\QQ$-virtual abelian variety of $\gl2$-type 
if $A$ satisfies (\ref{eq:condition1}) and (\ref{eq:condition2}) and we say that $A$ is non-CM if  $\textnormal{End}^0(A)$ contains no commutative $\QQ$-algebra of degree $2g$. 
For instance, if $E$ is a non-CM elliptic curve over $\bar{\QQ}$ which is isogenous to all its $G_\QQ$-conjugates $E^\sigma$, then $E$ satisfies (\ref{eq:condition1}) and (\ref{eq:condition2}). Such elliptic curves $E$ were studied for example by Ribet \cite{ribet:gl2} and Elkies \cite{ribet:gl2conference}. 

Let $K\subset \bar{\QQ}$ be a number field of degree $d=[K:\QQ]$, with ring of integers $\OK$. We denote by $D_K$ the absolute value of the discriminant of $K$ over $\QQ$. Let $S$ be a non-empty open subscheme of $\sp(\OK)$. We write $N_S=\prod N_v$ with the product taken over all $v\in \sp(\OK)-S$. For any abelian scheme $A$ over $S$, let $h_F(A)$ be the stable Faltings height of $A$ defined in Section \ref{sec:heights}. We denote by $\textnormal{rad}(m)$ the radical of any $m\in \ZZ_{\geq 1}$. 
\begin{proposition}\label{prop:virtual}
There exists an effective constant $c$, depending only on  $d$ and $g$, with the following property. Let $A$ be an abelian scheme over $S$ of relative dimension $g$. If $A_{\bar{\QQ}}$ is a simple  $\QQ$-virtual abelian variety of $\gl2$-type which is non-CM, then $$h_F(A)\leq c\cdot \textnormal{rad}(N_SD_K)^{24}.$$
\end{proposition}

In the proof of Proposition \ref{prop:virtual}, we use a  result of Wu  \cite{wu:thesis}  to reduce the problem to abelian varieties over $\QQ$ of $\gl2$-type. Then we combine the techniques of the previous sections with a result of Silverberg \cite{silverberg:fieldsofdef} to deduce the statement. 

\begin{proof}[Proof of Proposition \ref{prop:virtual}]
 Let $A$ be an abelian scheme  over $S$ of relative dimension $g$. We suppose that $A_{\bar{\QQ}}$ is a simple  $\QQ$-virtual abelian variety of $\gl2$-type which is non-CM. 
 
1. Let $L$ be a finite field extension of $K$  with the following three properties:
\begin{itemize}
\item[(a)] $L/\QQ$ is a Galois extension of degree $l=[L:\QQ]$,
\item[(b)] all endomorphisms of $A_{\bar{\QQ}}$ are defined over $L$,
\item[(c)] and for all  $\sigma\in G_\QQ$ the isogenies $\mu_\sigma: A_{\bar{\QQ}}^\sigma\to A_{\bar{\QQ}}$ in (\ref{eq:condition2}) are defined over $L$.
\end{itemize}
We denote by $C_\QQ=\textnormal{Res}_{L/\QQ}(A_L)$ the Weil restriction of $A_L$. The proof of \cite[Theorem 2.1.13]{wu:thesis} shows in addition that $C_\QQ$ has a $\QQ$-quotient $B_\QQ$ which is of $\gl2$-type.

2. We now show that $A_L$ is $L$-isogenous to some abelian subvariety of $B_L$. The abelian variety $C_L$ is $L$-isomorphic to $\prod A_L^\sigma$ with the product taken over all $\sigma\in \textnormal{Gal}(L/\QQ)$, where $A_L^\sigma$ denotes the base change of $A_L$ with respect to $\sigma:L\to L$. 
Thus on using that $B_\QQ$ is a $\QQ$-quotient of $C_\QQ$, we obtain a surjective morphism 
$
\prod A_L^\sigma\cong C_L\to B_L
$
of abelian varieties over $L$. 
Further, our assumptions on $A_{\bar{\QQ}}$ imply that each  $A_L^\sigma$ is $L$-simple. Therefore Poincar\'e's reducibility theorem shows that there exists $\sigma \in \textnormal{Gal}(L/\QQ)$ such that $A_L^\sigma$  is $L$-isogenous to some abelian subvariety $A_L'$ of $B_L$. The abelian varieties  $A_L$ and $A_L^\sigma$ are $L$-isogenous, since $A_{\bar{\QQ}}$ is a $\QQ$-virtual abelian variety with isogenies $\mu_\sigma:A_{\bar{\QQ}}^\sigma\to A_{\bar{\QQ}}$ defined over $L$ by (c). Hence $A_L$ is $L$-isogenous to the abelian subvariety $A'_L$ of $B_L$.

3. We begin to estimate the stable Faltings height $h_F$. Let $N_{A_L}$ and $N_C$ be the conductors of $A_L$ and $C_\QQ$ respectively. Milne \cite[Proposition 1]{milne:arithmetic} gives that $N_C=N_{A_L}D_L^{2g}$ for $D_L$ the absolute value of the discriminant of $L$ over $\QQ$. Hence, if a rational prime number $p$ does not divide $N_{S'}=\textnormal{rad}(N_{A_L}D_L)$, then $C_\QQ$ has good reduction at $p$.  This shows that $C_\QQ$ extends to an abelian scheme $C$ over $S'=\sp(\ZZ[1/N_{S'}])$ and thus the $\QQ$-quotient $B_\QQ$ of $C_\QQ$ extends to an abelian scheme $B$ over $S'$.  Therefore an application of Theorem \ref{thm:es} with the abelian scheme $B$ over $S'$ of $\gl2$-type gives that
$
h_F(B)\leq (3n)^{144n}N_{S'}^{24}
$
for $n$ the relative dimension of $B$ which satisfies $n\leq \dim(C_\QQ)=lg$. Then on using that $A_L$ is $L$-isogenous to the abelian subvariety $A'_L$ of $B_L$, we see that the arguments of Theorem \ref{thm:gl2} (ii) lead to $h_F(A)\leq c'N_{S'}^{24}$  for $c'$ an effective constant depending only on  $l$ and $g$. 

4. It remains to control the quantities $N_{S'}$ and $l$. We observe that $N_{S'}=\textnormal{rad}(N_{A_L}D_L)$ divides $\textnormal{rad}(N_SD_L)$ since $A$ is an abelian scheme over $S$. 
To estimate $l$ and $\textnormal{rad}(D_L)$ we use  \cite[Theorem 4.2]{silverberg:fieldsofdef}. It implies the existence of a field extension $L$ of $K$, with the properties (a), (b) and (c), such that $\textnormal{rad}(D_L)\mid \textnormal{rad}(N_SD_K)$ and such that $l=[L:\QQ]$ is effectively bounded in terms of $d$ and $g$. It follows that $N_{S'}\mid \textnormal{rad}(N_SD_K)$ and then the inequality $h_F(A)\leq c'N_{S'}^{24}$ from 3. implies Proposition \ref{prop:virtual}. 
\end{proof}

On computing explicitly the constant $c$  of Proposition \ref{prop:virtual}, one sees that $c$ depends double exponentially on $d$ and $g$. However, in certain cases of interest it is possible to obtain that $c$ depends  exponentially on $d$ and $g$. Further, we mention that one can use the arguments of the proof of Theorem \ref{thm:gl2} to remove in Proposition \ref{prop:virtual} the  assumption  that $A_{\bar{\QQ}}$ is simple. In fact one can generalize all results of Section \ref{sec:esgl2} (except Proposition \ref{prop:ss}) by replacing Theorem \ref{thm:es} with Proposition \ref{prop:virtual} in the proofs of the previous section.

{\scriptsize
\bibliographystyle{amsalpha}
\bibliography{../../literature}
}

\noindent IH\'ES, 35 Route de Chartres, 91440 Bures-sur-Yvette, France\\
E-mail adress: {\sf rvk@ihes.fr}

\end{document}